\documentclass{amsart}
\usepackage[english]{babel}
\usepackage{amsmath,enumerate,hyperref,xcolor}
\usepackage{amsthm}
\usepackage{amssymb}
\usepackage{amscd}
\usepackage[all]{xy}
\usepackage{xcolor}
\usepackage[latin1]{inputenc}
\usepackage[dvips]{graphics}

\usepackage{turnstile}

\newenvironment{smallarray}[1]
 {\null\,\vcenter\bgroup\scriptsize
  \arraycolsep=.25em
  \hbox\bgroup$\array{@{}#1@{}}}
 {\endarray$\egroup\egroup\,\null}
  
\newtheorem{teor}{Theorem}[section]
\newtheorem{defi}[teor]{Definition}
\newtheorem{propdef}[teor]{Proposition-Definition}
\newtheorem{lema}[teor]{Lemma}
\newtheorem{prop}[teor]{Proposition}
\newtheorem{cor}[teor]{Corollary}
\newtheorem{rem}[teor]{Remark}
\newtheorem{ejem}[teor]{Example}

\newtheorem{ques}[teor]{Question}

\def\hom{\mathop{\rm Hom}\nolimits}
\def\Hom{\mathop{\rm Hom}\nolimits}

\def\mod {\mathop{\rm mod}\nolimits}
\def\End {\mathop{\rm End}\nolimits}

\def\Der {\mathop{\bf D}\nolimits}
\def\Ext {\mathop{\rm Ext}\nolimits}
\def\Soc {\mathop{\rm Soc}\nolimits}
\def\Ker {\mathop{\rm Ker}\nolimits}
\def\Im {\mathop{\rm Im}\nolimits}

\def\P{\mathcal{P}}
\def\B{\mathcal B}
\def\fp{\mathrm{fp}}
\def\Ab{\mathrm{Ab}}
\def\t{\mathbf{t}}
\def\mod#1{\mathrm{Mod}\text{-}#1}
\def\fpmod#1{\mathrm{mod}\text{-}#1}
\def\proj#1{\mathrm{proj}\text{-}#1}
\def\Pres{\mathrm{Pres}}
\def\Ht{\H_\t}
\def\Scal{\mathcal{S}}
\def\Gen{\mathrm{Gen}}
\def\Z{\mathbb Z}
\def\Ch{\mathbf{C}}
\def\S{\mathcal S}
\def\Cogen{\mathrm{Cogen}}
\def\Copres{\mathrm{Copres}}
\def\Prod{\mathrm{Prod}}
\def\Coker{\mathrm{Coker}}

\def\op{\mathrm{op}}
\def\id{\mathrm{id}}

\def\add{\mathrm{add}}
\def\Set{\mathrm{Set}}

\def\A{\mathcal {A}}

\def\D{\mathcal {D}}
\def\U{\mathcal {U}}
\def\W{\mathcal {W}}
\def\V{\mathcal {V}}
\def\C{\mathcal C}
\def\Ob{\mathrm{Ob}}
\def\colim{\mathrm{colim}}
\def\F{\mathcal{F}}
\def\X{\mathcal{X}}
\def\G{\mathcal {G}}
\def\fg{\mathrm{fg}}

\def\H{\mathcal {H}}
\def\T{\mathcal{T}}
\def\fp{\mathrm{fp}}
\def\K{\mathbf{K}}

\def\Inj{\mathrm{Inj}\text{-}}
\def\ann{\mathrm{ann}}
\def\pres{\mathrm{pres}}
\def\copres{\mathrm{copres}}
\def\Ho{\mathrm{Ho}}
\def\tr{\mathrm{tr}}
\def\Flats#1{\mathrm{Flats}\text{-}#1}
\def\Qcoh{\mathrm{Qcoh}}
\def\Coprod{\mathrm{Coprod}}

\makeatletter
\@namedef{subjclassname@2020}{%
  \textup{2020} Mathematics Subject Classification}
\makeatother

\begin{document}

\title{Locally finitely presented and coherent hearts}

\author{Carlos E. Parra}
\address{Instituto de Ciencias F\'\i sicas y Matem\'aticas, Edificio Emilio Pugin, Campus Isla Teja, Universidad Austral de Chile, 5090000 Valdivia, CHILE}
\email{carlos.parra@uach.cl}
\thanks{The first named author was supported by CONICYT/FONDECYT/Iniciaci\'on/11160078} 
\author{Manuel Saor\'{\i}n}
\address{Departamento de Matem\'{a}ticas,
Universidad de Murcia,  Aptdo.\,4021,
30100 Espinardo, Murcia, SPAIN}
\email{msaorinc@um.es}
\thanks{
The second named author was supported by the grant PID2020-113206GB-I00, funded by MCIN/AEI/10.13039/501100011033, with a part of FEDER funds.}
\author{Simone Virili}
\address{Departament de Matem\`atiques, Facultat de Ci\`encies, Universitat Aut\`onoma de Barcelona, Edifici C, 08193 Bellaterra (Barcelona), SPAIN}
\email{simone.virili@uab.cat \textrm{ or } virili.simone@gmail.com}

\keywords{$t$-structure, heart, Happel-Reiten-Smal\o, locally finitely presented, locally coherent, silting, tilting, elementary cogenerator.}
\subjclass[2020]{18G80, 18E40, 18E10}

\date{}

\maketitle

\begin{abstract}Starting with a Grothendieck category $\mathcal{G}$ and a torsion pair $\mathbf{t}=(\mathcal{T},\mathcal{F})$ in $\G$, we study the local finite presentability and local coherence of {the} heart $\Ht$ of the associated Happel-Reiten-Smal\o\ $t$-structure in the derived category $\Der (\mathcal{G})$. We start by showing that, in this general setting, the torsion pair $\t$ is of finite type, if and only if it is quasi-cotilting, if and only if it is cosilting. 
We then proceed to study those $\t$ for which $\Ht$ is locally finitely presented, obtaining a complete answer under some additional assumptions on the ground category $\mathcal{G}$, which are general enough to include all locally coherent Grothendieck categories, all categories of modules and several categories of quasi-coherent sheaves over schemes. The third problem that we tackle is that of  {local coherence}. In this direction we characterize those torsion pairs $\t=(\T,\F)$ in a locally finitely presented $\G$ for which $\Ht$ is locally coherent in two cases: when  {the tilted t-structure in $\Ht$}  is assumed to restrict to finitely presented objects, and when $\F$ is cogenerating. In the last part of the paper we concentrate on the case when $\G$ is a category of modules over a small preadditive category, giving several examples and obtaining very neat (new) characterizations even in this more classical setting, also underlying connections with the notion of an elementary cogenerator.
\end{abstract}

\setcounter{tocdepth}{1}
\tableofcontents

\section*{Introduction}

$t$-Structures were introduced by Beilinson, Bernstein and Deligne \cite{BBD} in their study of perverse sheaves over an algebraic or analytic variety. A $t$-structure in an ambient triangulated category $\mathcal{D}$ is a pair $\tau=(\mathcal{U},\mathcal{W})$ of full subcategories satisfying some axioms (see Definition~\ref{def.$t$-structure}). The key point is that the intersection $\mathcal{H}=\mathcal{U}\cap\mathcal{W}$, called the heart of the $t$-structure, is an Abelian category whose short exact sequences are the triangles in $\mathcal{D}$ with the three vertices in $\mathcal{H}$. Moreover, such an Abelian category comes with a cohomological functor $H_\tau^0\colon\mathcal{D}\to\mathcal{H}$, allowing for the development of an intrinsic cohomology theory in $\mathcal{D}$, with cohomologies in $\H\subseteq \D$. Under suitable non-degeneracy and boundedness hypotheses on $\tau$, one can even recover the structure of $\mathcal{D}$ out of $\tau$. 

\smallskip
Once a new Abelian category is at hand, it is natural to ask if, with reasonable  hypotheses, this category acquires stronger properties. For example, in the classical hierarchy of Abelian categories introduced by Grothendieck \cite{Gro}, the so-called Grothendieck categories are high up in the list. In order to understand when the heart $\H$ of a $t$-structure $\tau$ is a Grothendieck category, it is almost unavoidable to require that the ambient triangulated category $\mathcal{D}$ has (arbitrary, set-indexed) coproducts. The problem of characterizing those $\tau$ whose heart is Grothendieck has deserved a lot of attention in recent years (see, for example, \cite{CGM}, \cite{CG}, \cite{PS1}, \cite{PS2}, \cite{PS3}, \cite{Bo}, \cite{Lu}, \cite{Laking}). Out of this deep work {and the conclusive papers \cite{Bo} and \cite{Saorin-Stovicek}, we now know that} the hearts of all the compactly generated $t$-structures  are Grothendieck categories, {even locally finitely presented.} The unique task still to be completed in this direction is to identify all the smashing $t$-structures,  for which the heart is a Grothendieck category. {The class of smashing $t$-structures strictly contains  that of the compactly generated ones and, in the middle of them, two other classes of $t$-structures have received attention. When $\mathcal{D}$ is compactly generated, one considers the $t$-structures with definable co-aisle (see \cite{Kr_coherent} for the definition of definable subcategory). Alternatively, if $\mathcal{D}$ is not necessarily compactly generated, but admits some enhancement that allows one to meaningfully define homotopy colimits (e.g., this is the case when $\D$ is the base of a strong and stable derivator, see \cite{Groth}, or when it is the homotopy category of a stable $\infty$-category, see \cite{Lu}) a $t$-structure $\tau$ is called homotopically smashing when its co-aisle is closed under directed homotopy colimits. Both the $t$-structures with definable co-aisle and the homotopically smashing $t$-structures have a heart that is a Grothendieck category (see \cite{SSV} for the former case and \cite{Saorin-Stovicek} for the latter), and actually both types of t-structures coincide when $\D$ is the compactly generated base of a strong and stable derivator (see \cite{Laking}). Out of these two subclasses, an identification of the remaining smashing $t$-structures whose heart is a Grothendieck category is still missing.}

%Once a new Abelian category is at hand, it is natural to ask if, under some reasonable set of hypotheses, this category acquires some stronger properties. For example, under the classical hierarchy on Abelian categories introduced by Grothendieck \cite{Gro}, the so-called Grothendieck categories are high up in the list. In order to study when the heart $\H$ of a $t$-structure $\tau$ is a Grothendieck category, it is almost unavoidable to require that the ambient triangulated category $\mathcal{D}$ has (arbitrary, set-indexed) coproducts. The problem of characterizing those $\tau$ whose heart is Grothendieck has deserved a lot of attention in recent years (see, for example, \cite{CGM}, \cite{CG}, \cite{PS1}, \cite{PS2}, \cite{PS3}, \cite{Bo}, \cite{Lu}, \cite{Laking}). Out of this deep work, we can conclude that, roughly speaking, the hearts of all the compactly generated $t$-structures appearing ``in nature'' are Grothendieck categories. \textcolor{blue}{\bf (Rev.\,1) este p\'arrafo debe modificarse} \textcolor{red}{ The unique task still to be completed in this direction is to identify all the (smashing) $t$-structures, at least when $\mathcal{D}$ is compactly generated, for which the heart is a Grothendieck category.}

\smallskip
The next natural problem to tackle is the following: suppose that we are given a $t$-structure $\tau$ with a heart $\H$ which is a Grothendieck category. Under which conditions does $\H$ satisfy some nice finiteness conditions? E.g., when is it locally finite, locally noetherian, locally coherent or, at least, locally finitely presented? The study of this general problem is only at its beginnings and few references have considered the problem so far (see, e.g., \cite{Sa}, \cite{Laking} and the recent {preprint}  \cite{Saorin-Stovicek}, where it is proved that the heart of any compactly generated $t$-structure in a triangulated category with coproducts is  locally finitely presented). In this paper we tackle the problem for the Happel-Reiten-Smal\o\ $t$-structure in the  derived category $\mathbf{D}(\mathcal{G})$ of a Grothendieck category $\mathcal{G}$ associated with a torsion pair $\mathbf{t}=(\mathcal{T},\mathcal{F})$ in $\mathcal{G}$. Concretely, we study when the associated heart is a locally finitely presented or a locally coherent Grothendieck category. Since, by the main results of \cite{PS1} and \cite{PS2}, the heart is a Grothendieck category if, and only if, $\mathbf{t}$ is of finite type (i.e., $\mathcal{F}$ is closed under taking direct limits in $\mathcal{G}$), the problem translates into that of characterizing the torsion pairs of finite type whose associated heart is locally finitely presented or locally coherent. 

\smallskip
Let us now describe (some of) the main results of the paper. The first one (see Theorem~\ref{main_thm_fintype_qcotilt_cosilt}) extends to Grothendieck categories a result recently proved for categories of modules (see Remark~\ref{rem.fintype-quasicotilting-cosilting}): 

\medskip\noindent
{\bf Theorem A.}
{\em The following are equivalent for a torsion pair $\mathbf{t}=(\mathcal{T},\mathcal{F})$ in a Grothendieck category $\G$:
 \begin{enumerate}[\rm (1)]
 \item $\mathbf{t}$ is of finite type;
 \item $\mathbf{t}$ is quasi-cotilting;
 \item $\mathbf{t}$ is the torsion pair associated with  a cosilting (pure-injective) object of $\G$.
 \end{enumerate}}

\medskip
The second main result (see Theorem~\ref{thm.locally-fp-hearts} for the precise statement) identifies the torsion pairs of finite type whose associated heart is locally finitely presented, for a wide class of ambient categories (see Corollary~C for a neater, though weaker, formulation of the result, which is often enough in practice):  

\medskip\noindent
{\bf Theorem B.}
{\em Let $\G$ be a locally finitely presented Grothendieck category, $\mathbf{t}=(\T,\F)$ a torsion pair in $\G$ and $\Ht$ the heart of the associated Happel-Reiten-Smal\o\ $t$-structure in $\Der(\G)$. Suppose also that either of the following two conditions is satisfied
\begin{enumerate}
\item[\rm ($\dag^\sharp$)] $\Ext_\G ^2(T,-)_{\restriction \F}\colon \F\to \Ab$ preserves direct limits for all $T\in\mathcal{T}\cap\fp(\G)$, and $\mathcal{T}\cap\fp(\mathcal{G})\subseteq\fp_2(\G)$;
\item[\rm ($\bullet$)]  $\G $ has a set of finitely presented generators which are compact in $\Der(\G )$.
\end{enumerate}
Then, the following assertions are equivalent:
\begin{enumerate}[\rm (1)]
  \item $\H_\mathbf{t}$ is a locally finitely presented Grothendieck category;
  \item $\mathcal{T}={\varinjlim} (\mathcal{T}\cap\fp(\G))$;
  %\item there is a set $\mathcal{S}\subseteq\fp(\G)$ such that $\mathcal{T}=\Gen(\mathcal{S})$ (or, equivalently, $\mathcal{T}=\Pres(\mathcal{S})$);
  \item $\mathbf{t}$ is generated by a set of finitely presented objects, i.e., there is a set $\Scal\subseteq \fp(\G)$ such that $\F = \S^{\perp}$. 
  \end{enumerate}}

\medskip
As an immediate consequence, one gets:

\medskip\noindent
{\bf Corollary C.}
{\em Let $\G$ be a Grothendieck category that satisfies one of the following conditions:
\begin{enumerate}[\rm (1)]
\item $\mathcal{G}$ is  locally coherent;
\item $\mathcal{G}$ is a category of modules over a ring or, more generally, over a small pre-additive category;
\item $\mathcal{G}=\Qcoh(\mathbb{X})$ is the category of quasi-coherent sheaves over a quasi-compact quasi-separated coherent regular scheme $\mathbb{X}$. 
\end{enumerate}
Given a torsion pair $\mathbf{t}=(\T,\F)$ in $\G$, the heart $\Ht$ of the associated $t$-structure in $\Der(\G)$ is a locally finitely presented Grothendieck category if, and only if, $\mathbf{t}$ is generated by finitely presented objects. 
}

\medskip
{We proceed with the characterization of the local coherence of $\Ht$. One would naturally expect $\Ht$ to be locally coherent when $\G$ is locally coherent and $\mathbf{t}$ is a finite type torsion pair that restricts to the subcategory $\fp(\G)$ of finitely presented objects. This was established in \cite{Sa}. Here we tackle the problem assuming only that $\G$ is locally finitely presented and $\mathbf{t}$ is  of finite type. No complete answer is given, but there are two interesting partial answers. The first one (see Theorem~\ref{thm.local-coherent-heart-general}),  gives necessary and  sufficient conditions for $\Ht$ to be  locally coherent when the Happel-Reiten-Smal\o\ tilt $\bar{t}=(\F[1],\T[0])$ is a torsion pair in $\Ht$ that restricts to $\fp(\Ht)$. As the statement is fairly technical we  do not include it in this introduction, but it has several useful consequences, as, for instance, a neat characterization of the local coherence of $\Ht$ when $\mathbf{t}$ is hereditary and $\G$ locally coherent (see Proposition~\ref{prop.after-referee-report}).

 \medskip 
We then give a complete answer to the problem when the torsionfree class $\mathcal{F}$ is generating (see Theorem~\ref{thm.main_thm_6c} for the precise statement, of which the following result is a shortened version).

\medskip\noindent
{\bf Theorem D.}
{\em Let $\G$ be a locally finitely presented Grothendieck category and $\t= (\T,\F)$ a torsion pair of finite type such that $\F$ is generating  in $\G$. Then,   
the following assertions are equivalent:
\begin{enumerate}[\rm (1)]
\item $\mathcal{H}_\mathbf{t}$ is locally coherent;
\item  $\mathbf{t}$ restricts to $\fp(\G)$ and $\F\cap\fp(\G)\subseteq\fp_\infty(\G)$.
\end{enumerate}
When these conditions hold, $\bar{\mathbf{t}}:=(\F[1],\T[0])$ restricts to $\fp(\Ht)$ if, and only if, $\G$ is locally coherent. }

\medskip
As a consequence (see Corollary~\ref{cor.cotiltingpair-locallycoh-heart-modcats} and Corollary~\ref{cor.F0}), one gets:

\medskip\noindent
{\bf Corollary E.}
{\em 
Let $\A$ be a small preadditive category, $\G:=\mod \A$ the category of right $\A$-modules, $\t= (\T,\F)$ a torsion pair of finite type in $\G$, with $\F$ generating in $\G$. The following  are equivalent:
\begin{enumerate}[\rm (1)]
\item $\Ht$ is locally coherent;
\item for all $X\in \fpmod\mathcal{A}:=\fp(\mod \A)$, $(1:t)(X):=X/t(X)$ admits a resolution by finitely generated projective modules;
\item $\mathbf{t}$ restricts to $\fpmod\mathcal{A}:=\fp(\mod \A)$ and $\F\cap\fpmod\A\subseteq\fp_\infty (\mod\A)$.
\end{enumerate}
  }

\medskip
Let us also remark that there is a way of linking the general problem of the local coherence of the heart $\Ht$ associated with a torsion pair of finite type $\t=(\T,\F)$, to the situation considered in Theorem~D. In fact, the subcategory $\underline{\F}$ of $\G$ consisting of the epimorphic images of the objects in $\F$ is a Grothendieck category on which the ``restricted'' torsion pair $\mathbf{t}'=(\T\cap\underline{\F},\F)$ is of finite type with $\F$ generating. Moreover, it follows from the local coherence of $\Ht$ that $\underline{\F}$ is locally finitely presented and $\mathcal{H}_{\mathbf{t}'}$ is locally coherent (see Corollary~\ref{cor.loc-coherence-implies-restriction}). It is then natural to ask for the conditions to add to the local coherence of $\mathcal{H}_{t'}$ in order to force the local coherence of $\Ht$. This is done in the final section of the paper, and the following result, which is part of Proposition~\ref{prop.after-referee-report}, is a sample:

\medskip\noindent
{{\bf Proposition F.} 
{\em  Let $\G$ be a locally finitely presented Grothendieck category,  $\mathbf{t}=(\mathcal{T},\mathcal{F})$ a torsion pair of finite type in $\G$ and $\mathbf{t}'=(\T\cap\underline{\F},\F)$. Consider the following assertions:
\begin{enumerate}[\rm (1)]
\item $\Ht$ is locally coherent;
\item the following conditions hold true:
\begin{enumerate}[\rm ({2.}1)]
\item $\mathbf{t}$  is generated by finitely presented objects;
\item the restricted heart $\mathcal{H}_{\mathbf{t}'}$ is locally coherent;
\item for some (resp., each) set of generators $\S\subseteq \fp(\G)$ of $\G$, the functor $\G ((1:t)(S),-)\colon\G\to\Ab$ preserves direct limits of objects in $\T$, for all $S\in \S$.
\end{enumerate}
\end{enumerate}  
The implication {\rm``(1)$\Rightarrow$(2)''} holds and, when $\G$ is locally coherent, both assertions are equivalent.}}

\medskip
In the particular case of module categories, linking this result with the study of torsion pairs associated to cosilting modules which are elementary cogenerators, we emphasize the following consequence (see Proposition~\ref{last_prop_laking}).

\medskip\noindent
{\bf Corollary G.}
{\em Let $R$ be a right coherent ring and $\mathbf{t}=(\mathcal{T},\mathcal{F})$ a torsion pair in $\mod R$. Then, the following assertions are equivalent:
\begin{enumerate}[\rm (1)]
\item $\Ht$ is a locally coherent Grothendieck category;
\item the following conditions hold true:
\begin{enumerate}[\rm ({2.}1)]
\item $\mathbf{t}$ is generated by finitely presented modules;
\item $\mathbf{t}=({}^\perp Q,\Cogen (Q))$, for a cosilting module $Q$ that is an elementary cogenerator in $\mod R$;
\item $\mathrm{ann}_{(-)}t(R)\colon\mod R\to\mod R$
% \quad\text{such that}\quad M\mapsto\mathrm{ann}_M t(R), 
preserves direct limits of modules in $\T$;
\end{enumerate}
\item the following conditions hold true:
\begin{enumerate}[\rm ({3.}1)]
\item $\mathbf{t}$ is generated by finitely presented modules;
\item the restricted heart $\mathcal{H}_{\mathbf{t}'}$ is locally coherent;
\item $\mathrm{ann}_{(-)}t(R)\colon \mod R\to\mod R$ preserves direct limits of modules in $\T$.
\end{enumerate}
\end{enumerate}}

\medskip
As a complement to the results about the local coherence of $\Ht$, we give in Section~\ref{Sec_consequences} several examples of categories of modules for which there are torsion pairs $\t$ such that $\Ht$ is locally coherent, but either $\mod\A$ is not locally coherent or $\mod\A$ is locally coherent and $\mathbf{t}$ does not restrict to $\fpmod\A$.

\medskip
 The organization of the paper goes as follows. Section~\ref{Sec_Prelim} serves to recall some facts (most of which are, at least partially, known) about Grothendieck categories, especially about the properties of the finitely ($n$-)presented {objects}. In Section~\ref{Sec_Prelim_on_tors} we collect the needed background about torsion pairs of finite type in Grothendieck categories. Then, Section~\ref{Sec_Prelim_on_der} contains the basic definitions about triangulated categories, derived categories, general $t$-structures and some constructions related to the Happel-Reiten-Smal\o\ tilts of a $t$-structure.  Section~\ref{Section_finite_type} is devoted to the proof of Theorem~A, following a {route different to} that followed {in the literature}  when the ambient category is a category of modules. In Section~\ref{Section_fp}, for a given torsion pair of finite type $\mathbf{t}$ in a Grothendieck category, we characterize the objects of the heart $\Ht$ which are finitely presented and finitely $2$-presented, giving special emphasis to the identification of the stalk complexes in $\fp(\Ht)$. In Section~\ref{lfp_section} we study conditions on $\t$ for $\Ht$ to be locally finitely presented, proving, in particular, Theorem~B and Corollary~C. In Section~\ref{Section_loc_coh_hearts}, we study those torsion pairs $\t$ for which $\Ht$ is locally coherent; the main results in this direction are Theorem~\ref{thm.local-coherent-heart-general} and (a more precise version of) Theorem~D. In the final Section~\ref{Sec_consequences}, we specify our main results to  the case when $\mathcal{G}$ is a category of modules, we prove Corollary~E and, inspired by the work of Rosanna Laking \cite{L} and highly motivated by  her suggestions,  we try to understand the relations between the local coherence of the heart and the property of the associated cosilting module being an elementary cogenerator. In particular, we prove the general Proposition~\ref{last_prop_laking} and a  version of Corollary~F that applies to categories of modules over small preadditive categories.

 \bigskip\noindent
{\bf Acknowledgement.} 
We are grateful to Jan \v{S}\v{t}ov\'{\i}\v{c}ek, Henning Krause and Flavio Coelho for several comments and suggestions.
It is also a pleasure for us to thank Lidia Angeleri H\"ugel and Rosanna Laking for several helpful discussions, started during a fruitful research stay of the three authors in Verona, in February 2019. We also thank the referee for the careful reading of the manuscript and the subsequent comments and suggestions that helped to improve the presentation of the paper.

 \section*{List of most commonly used symbols}
We collect here some of the standard symbols that are more frequently used throughout the paper, so that they are easily available to the reader at any time.
 \begin{itemize}
 \item $\C$ denotes a category, and $\C(C,C'):=\Hom_\C(C,C')$ for all $C,\, C'$ in $\C$;
 \item $I$, $J$ and $\Lambda$ usually denote small categories or sets (i.e., discrete small categories);
\item $\Set$ denotes the category of sets;
\item $\Ab$ denotes the category of Abelian groups;
 \item $\G$ denotes a Grothendieck category;
 \item $\Inj \G$ denotes the class of injective objects in $\G$.
\end{itemize}
All the subcategories that we consider throughout the paper are meant to be full, so we generally treat ``subclass" (of objects)  of a given category and ``subcategory'' as synonyms. Let now $I$ be a small category, $X\in \C$ an object, and $(X_i)_I\in\mathrm{Func}(I,\C)$ an $I$-shaped diagram in a category $\C$. We use the following notations to denote different types of (co)limits (whenever they exist):
\begin{itemize}
\item if $I$ is a set, $\prod_IX_i$ ($\coprod_IX_i$) denotes the (co)product of $(X_i)_I$;
\item if $I$ is a set, $X^{I}$ ($X^{(I)}$) denotes the (co)product of $|I|$-many copies of $X$;
\item $\lim_IX_i$ ($\colim_IX_i$) denotes the (co)limit of an $I$-shaped diagram $(X_i)_I$;
\item if $I$ is directed, ${\varinjlim}_IX_i$ denotes the direct limit of $(X_i)_I$; 
\item if $I$ is directed, and all the transitions $X_{i,j}\colon X_i\to X_j$ are monomorphisms, the direct limit of $(X_i)_I$ is said to be a {\bf direct union} and it is denoted by $\bigcup_IX_i$.
\end{itemize}
When we want to emphasize that a (co)product, (co)limit, direct limit or direct union is taken in a specific category $\C$, we use the following notation 
\begin{equation}\label{relative_lims}
\xymatrix{
\coprod_I^{(\C)}X_i,\quad\prod_I^{(\C)}X_i,\quad \lim_I^{(\C)}X_i,\quad \colim_I^{(\C)}X_i,\quad {\varinjlim}_I^{(\C)}X_i,\quad \bigcup_I^{(\C)}X_i. 
}
\end{equation}
Given a subcategory $\S$ of an Abelian category $\C$, supposing the needed (co)limits exist in each case,
\begin{itemize}
\item $\Prod(\S)$ ($\Coprod(\S)$) is the class of all direct summands of the (co)products of families of the form $(S_i)_I$, with $S_i\in \S$ for all $i\in I$, with $I$ a set;
 \item $\Gen(\S)$ ($\Cogen(\S)$) is the class of quotients (subobjects) of objects in $\Coprod(\S)$ ($\Prod(\S)$);
\item  $\underline{\Cogen}(\S)$ is the class of quotients of objects in $\Cogen(\S)$;
\item $\Pres(\mathcal{S})$ is the class of cokernels of maps $X\to Y$, with $X,\,Y\in \Coprod(\S)$;
\item $\Copres(\mathcal{S})$ is the class of kernels of maps $X\to Y$, with $X,\,Y\in \Prod(\S)$;
\item ${\varinjlim}\, \S$ is the class of direct limits of direct systems $(S_i)_{I}$, with $S_i\in \S$ for all $i\in I$;
\item $\mathrm{add}(\S)$ is the class of direct summands of {finite coproducts of}  objects in $\S$;
\item $\mathrm{sum}(\S)$ is the class of finite coproducts of objects in $\S$;
\item $\pres_n(\S)$, for $n\geq {0}$, is the class of those $N\in \C$ that admit an exact sequence 
\[
X_{n}\longrightarrow \cdots \longrightarrow X_{1}\longrightarrow X_0\longrightarrow N\longrightarrow 0,\qquad \text{with $X_{k}\in \mathrm{sum}(\S)$, for all $k=0,\dots,n$.}
\] 
\item $\copres_n(\S)$, for $n\geq {0}$, is the class of those $N\in \C$ that admit an exact sequence 
\[
0\longrightarrow N\longrightarrow X_{0}\longrightarrow X_1\longrightarrow  \cdots \longrightarrow X_{n},\qquad \text{with $X_{k}\in \mathrm{sum}(\S)$, for all $k=0,\dots,n$.}
\] 
\item $\mathrm{gen}(\S):=\pres_{ {0}}(\S)$  {and $\pres (\S):=\pres_1(\S)$}.
\end{itemize}
As we did in \eqref{relative_lims}, if we want to make clear that the products defining $\Prod(\S)$ are taken in a specific category $\C$, we  then write $\Prod_\C(\S)$, and we  adopt similar notations for the rest of the classes defined above. More specific symbols and conventions will be introduced in the body of the paper. 

\section{Preliminaries on Grothendieck categories}\label{Sec_Prelim}

In this first section we give the necessary definitions and preliminaries about Grothendieck categories and finitely ($n$-)presented objects. Moreover, we introduce locally finitely generated, locally finitely presented, and locally coherent categories. Some of our results, although probably known to experts, are stated in a generality that is not available in the literature; in those cases we include a proof. 

\subsection{Finitely generated and  finitely presented objects}\label{fg_and_fp_objects}

An Abelian category $\C$  is 
\begin{itemize}
\item (Ab.$3$) (resp., (Ab.$3^*$)) when it is cocomplete (resp., complete);
\item (Ab.$5$) when it is cocomplete and direct limits are exact;
\item a {\bf Grothendieck} category if it is (Ab.$5$) and it has a generator.
\end{itemize}
Recall that a Grothendieck category is automatically (Ab.$3^*$) and it has enough injectives. 

Let $\G$ be a Grothendieck category, then an object $X$ in $\G$ is said to be 
\begin{itemize}
\item {\bf finitely generated} if $\G(X,-)\colon \G\to \Ab$ commutes with direct unions;
\item  {\bf finitely presented} if $\G(X,-)\colon \G\to \Ab$ commutes with direct limits.
\end{itemize}
These definitions coincide with the usual ones if $\G$ is a category of modules. In what follows we let
\[
\fg(\G):=\{\text{fin.\ gen.\ objects in $\G$}\}\quad\text{and}\quad\fp(\G):=\{\text{fin.\ pres.\ objects in $\G$}\}.
\]
Note that finitely presented objects are, in particular, finitely generated, that is, $\fp(\G)\subseteq \fg(\G)$. The following result gives alternative characterizations of the objects in $\fg(\G)$:

\begin{prop} \label{prop.fg-objects}
Let $\G$ be a Grothendieck category.  The following are equivalent for  $X\in \G$:
\begin{enumerate}[\rm (1)]
\item $X$ is finitely generated;
\item for any directed family $(X_i)_{I}$ of subobjects such that $\bigcup_{I}X_i=X$, there is $j\in I$ such that $X_j=X$;
\item for any direct system $(X_i)_{I}$ in $\G$, the following canonical map is a monomorphism 
\[
\xymatrix{
{\varinjlim}_I \G(X,X_i)\longrightarrow\G(X,{\varinjlim}_I X_i).
}
\] 
\end{enumerate}
As a consequence, the subcategory $\fg(\G)$ is closed under taking extensions and quotients in $\G$.
\end{prop}
\begin{proof}
(1)$\Leftrightarrow$(2)  {is \cite[Proposition~V.3.2]{St}, and (1,2)$\Rightarrow$(3) is included in the proof of \cite[Proposition~V.3.4]{St} since, with the notation of {\rm [op.\ cit.]}, the proof that $\phi$ is monic only uses that $C\in\fg(\G)$.} 

\smallskip\noindent
(3)$\Rightarrow$(2). {Let $(X_i)_{I}$ be an upward directed system of subobjects of $X$ such that $X=\bigcup_{I}X_i$. This induces the following direct system of short exact sequences in $\G$:
\[
(0\longrightarrow X_i \longrightarrow X\stackrel{p_i}{\longrightarrow}X/X_i\longrightarrow 0)_{I},
\] 
from which we get that $\varinjlim (X/X_i)=0$. By (3), we deduce that \mbox{$\varinjlim_I\G(X,X/X_i)=0$.} Take now an arbitrary $i\in I$, and consider the canonical projection $p_i\colon X\to X/X_i$. Then, $p_i$ is mapped to zero by the following canonical map to the direct limit in $\Ab$:
$u_i\colon\G(X,X/X_i)\stackrel{}{\to}\varinjlim_I\G(X,X/X_i).$ 
This means that there exists an index $j\geq i$ such that the composition $p_{ij}\circ p_i\colon X\stackrel{}{\to}X/X_i\stackrel{}{\to}X/X_j$ is zero. 
But this composition is precisely the projection $p_j\colon X\to X/X_j$, and so we get that $X=X_j$.}
\\
{Finally, the fact that $\fg(\G)$ is closed under taking extensions and quotients follows by  (3).}
\end{proof}

\subsection{Finitely $n$-presented objects}\label{subsectionEqII}
%Apart from the subcategories $\fg(\G)$ and $\fp(\G)$, we \textcolor{red}{recall the following ones, already used in \cite{BGP} when $\G$ is locally finitely presented. Concretely, Lemma~\ref{lem.n-fp subcategories}, Cor. \ref{coro_1_fpn} and Proposition~\ref{prop_mono} below appear in  [op.\ cit.] in that more restrictive setting with different proofs.}
%
{Apart from the subcategories $\fg(\G)$ and $\fp(\G)$, we consider the following ones, already appeared in \cite{BGP} when $\G$ is locally finitely presented. Concretely, Lemma~\ref{lem.n-fp subcategories}, Corollary~\ref{coro_1_fpn} and Proposition~\ref{prop_mono} below appear in  [op.\ cit.] in that more restrictive environment. However {our treatment is different in that we can find a much shorter proof of Proposition~\ref{prop_mono}, which holds in any Grothendieck category, and allows us to work in this more general setting.}}
% thanks to our different proof of the Proposition~\ref{prop_mono} we could extended the others results.}
\begin{defi}[\cite{BGP}] \label{def.n-fp object}
Let $\G$ be a Grothendieck category and $n>0$ a positive integer. An object $X$ in $\G$ is {\bf finitely $n$-presented} when the functors 
\[
\Ext_\G^k(X,-)\colon \G\longrightarrow\Ab
\] 
preserve direct limits for $0\leq k<n$. We then define the following classes:
\begin{itemize}
\item $\fp_0(\G):=\fg(\G)$
\item  $\fp_n(\G)$ is the class of the finitely $n$-presented objects in $\G$, for all $n>0$;
\item $\fp_\infty(\G):=\bigcap_{n\geq 0}\fp_n(\G)$. 
\end{itemize}
\end{defi}
Let us remark that $\fp_1(\G)=\fp(\G)$. Furthermore, there is a chain of inclusions 
\[
\fp_0(\G)\supseteq\fp_1(\G)\supseteq \cdots\supseteq\fp_n(\G)\supseteq\fp_{n+1}(G) \supseteq \cdots \supseteq \fp_\infty(\G).
\]
In the following two examples we specialize the above definition to the particular setting of categories of modules and to that of categories of quasi-coherent sheaves:

\begin{ejem}
In case $\G=\mod R$ is a category of modules over a ring $R$ or, more generally, over a small preadditive category, the finitely $n$-presented modules are those $X$ that admit a resolution 
\[
\cdots \longrightarrow P^{-k} \longrightarrow \cdots \longrightarrow P^{-1} \longrightarrow P^0 \longrightarrow X \longrightarrow 0,
\] 
where all $P^{-k}$ are finitely generated projective modules, for $k=0,1,\dots,n$.
\end{ejem}

\begin{ejem}
Let $(\mathbb{X},\mathcal O_{\mathbb X})$ be a quasi-compact quasi-separated scheme, in which case the category $\G_\mathbb{X}:=\Qcoh(\mathbb{X})$ of quasi-coherent sheaves is locally finitely presented (see \cite[Proposition~3.1]{Ga}). As mentioned in \cite[p.\ 6]{BGP} (using arguments of \cite{EG}), a given $\F\in\G_{\mathbb{X}}$ belongs in $\fp_n(\G_{\mathbb{X}})$ if, and only if, $\F(U)\in\fp_n(\mod {{\mathcal O}_{\mathbb X}(U)})$, for each affine open $U$ in $\mathbb{X}$. 
\end{ejem}
Let now $\C$ be an Abelian category and suppose that the needed products and coproducts exist in $\C$. Given two objects $X$ and $Y$ in $\C$, we define two functors: 
 \begin{equation}\label{def_Phi_Psi}
\Phi_X\colon \C\longrightarrow \C\qquad\text{and}\qquad\Psi_Y\colon \C\longrightarrow \C,
 \end{equation}
each defined by the composition of three functors: $\Phi_X$ is the $\hom$-functor $\C(X,-)\colon \C\to \Ab$, followed by the forgetful functor $|-|\colon \Ab\to \Set$, and finally the functor $X^{(-)}\colon \Set\to \C$, mapping a set $S$ to the coproduct $X^{(S)}$; similarly $\Psi_Y$ is the $\hom$-functor $\C(-,Y)\colon \C^\op\to \Ab$, followed by the forgetful functor $|-|\colon \Ab\to \Set$, and finally the functor $Y^{-}\colon \Set^{\op}\to \C$, mapping a set $S$ to the product $Y^{S}$. So, given $C\in\C$, we get:
\[
\Phi_X(C)=X^{(\C(X,C))}\qquad\text{and}\qquad\Psi_Y(C)=Y^{\C(C,Y)}.
\]
These functors come with natural transformations $\rho\colon \Phi_X\Rightarrow \id_{\C}$ and $\iota\colon \id_{\C} \Rightarrow \Psi_Y$, 
where $\rho$ is epimorphic if $X$ is a generator, and $\iota$ is monomorphic if $Y$ is a cogenerator. 
In fact, given two sets of objects $\mathcal X$ and $\mathcal Y$ in $\C$, one defines
\begin{align*} 
\xymatrix{\Phi_{\mathcal X}\colon \C\longrightarrow \C\quad \text{such that}\quad \Phi_{\mathcal X}(C):=\coprod_{X\in \mathcal X}\Phi_X(C),}\\
\xymatrix{\Psi_{\mathcal Y}\colon \C\longrightarrow \C\quad \text{such that}\quad \Psi_{\mathcal Y}(D):=\prod_{Y\in \mathcal Y}\Psi_Y(D).}
\end{align*}
Again, these functors come together with natural transformations $\rho\colon \Phi_{\mathcal X}\Rightarrow  \id_{\C}$ and $\iota\colon \id_{\C} \Rightarrow \Psi_{\mathcal Y}$, where $\rho$ is epimorphic if $\mathcal X$ is a set of generators, and $\iota$ is monomorphic if $\mathcal Y$ is a set of cogenerators. 

\begin{lema}\label{lema_BP}
Let $\G$ be a Grothendieck category and fix an injective cogenerator $E$ in $\G$. For $n>1$, an object $X$ is in $\fp_n(\G)$ if and only if $X\in\fp(\G)$ and $\Ext_\G^k(X,-)$ vanishes on direct limits of objects in $\Prod(E)=\Inj\G$, for all $k=1,\dots,n-1$.
\end{lema}
\begin{proof}
{The {\it ``only if''} part is clear. For the {\it ``if''} part, we use induction on  $n\geq 2$. For $n=2$, consider the functor $\Psi_E$ described above and note that $\Im(\Psi_E)\subseteq \Prod(E)$. Now the vanishing of the functor $\Ext^{1}_{\G}(X,-)$  on direct limits of objects in $\Prod(E)$, for $X\in \fp(\G)$, implies that $X\in \fp_2(\G)$ by \mbox{\cite[Theorem~B.1]{BP2}.} Suppose now that $n>2$. The inductive hypothesis says that $X\in\fp_{n-1}(\G)$ and, by assumption,  $\Ext_\G^{n-1}(X,-)$ vanishes on direct limits of objects in $\Prod (E)$, so that $X\in\fp_n(\G)$, by a second application of \cite[Theorem~B.1]{BP2}.}  
\end{proof}

 \begin{prop}\label{prop_mono}
 Let $\G$ be a Grothendieck category,  $X\in\fp_n(\G)$ for some $n\geq 0$, and  $(Y_i)_{ I}$  a direct system in $\G$. Then, the following canonical map is a monomorphism:
 \[
 \xymatrix{
 f\colon {\varinjlim}_I\Ext_\G^n(X,Y_i)\longrightarrow\Ext_\G^n(X,{\varinjlim}_I Y_i).
 }
 \]
\end{prop}
\begin{proof}
The case $n=0$ follows by Proposition~\ref{prop.fg-objects}. For $n>0$, choose an injective cogenerator $E$ in $\G$, so that $\iota\colon \id_\G\Rightarrow \Psi_E$ is monomorphic. One obtains the following exact sequences in $\G$:
\[
(\xymatrix{0 \ar[r] & Y_{i} \ar[r]^-{\iota_{i} \hspace{0.2 cm}} & \Psi_E(Y_{i}) \ar[r] & \Coker(\iota_{i}) \ar[r] & 0})_I,
\]
from which one gets the following commutative diagram with exact rows and obvious vertical maps:
\[ \scalebox{0.91}{
\xymatrix@C=25pt{
{\varinjlim}_I \Ext_\G^{n-1}(X,\Psi_E(Y_{i})) \ar[r] \ar[d]|{\cong} & {\varinjlim}_I {\Ext_\G^{n-1}}(X,\Coker(\iota_{i})) \ar[r] \ar[d]|{\cong} & {\varinjlim}_I \Ext^{n}_{\G}(X,Y_{i}) \ar[r] \ar[d]^{f}& 0 \ar[d]\\ 
\Ext_\G^{n-1}(X, {\varinjlim}_I\Psi_E(Y_{i})) \ar[r] & \Ext_\G^{n-1}(X,{\varinjlim}_I\Coker(\iota_{i})) \ar[r] & \Ext^{n}_{\G}(X, {\varinjlim}_I Y_{i}) \ar[r] & \Ext^{n}_{\G}(X,{\varinjlim}_I\Psi_E(Y_{i})).
}
}\]
The two leftmost vertical arrows are isomorphisms, forcing $f$ to be a monomorphism.
\end{proof}

In what follows we give some results about closure properties of the classes $\fp_n(\G)$.

\begin{lema} \label{lem.n-fp subcategories}
Let $\G$ be a Grothendieck category and consider a short exact sequence in $\G$:
\begin{equation}\label{ses_fp_n_eq}
0\longrightarrow X\longrightarrow Y\longrightarrow Z\longrightarrow 0.
\end{equation}
Then, the following assertions hold true for all $n\geq 0$:  
\begin{enumerate}[\rm (1)]
\item if $X\in\fp_n(\G)$ and $Y\in\fp_{n+1}(\G)$, then $Z\in\fp_{n+1}(\G)$;
\item if $Z\in\fp_{n+1}(\G)$ and $Y\in\fp_{n}(\G)$, then $X\in\fp_{n}(\G)$;
\item if both $X$ and $Z\in \fp_n(\G)$, then $Y\in\fp_n(\G)$.
\end{enumerate}
\end{lema}
\begin{proof}
We first introduce two diagrams  we will use in proving all the assertions. The former is a fragment of the long exact sequence of functors $\G\to\Ab$ associated with the  exact sequence \eqref{ses_fp_n_eq}:
\begin{equation}\label{long_ex_fp_eq}
{\xymatrix@R=10pt@C=15pt{
\Ext_\G^k(Y,-)\ar[r]&\Ext_\G^k(X,-)\ar[r] &\Ext_\G^{k+1}(Z,-)\ar[r]&\Ext_\G^{k+1}(Y,-)
}}
%Original%%%%%%%
%\xymatrix@R=10pt@C=15pt{
%\Ext_\G^k(Z,-)\ar[r]&\Ext_\G^k(Y,-)\ar[r]&\Ext_\G^k(X,-)\ar[d]\\
%&&\Ext_\G^{k+1}(Z,-)\ar[r]&\Ext_\G^{k+1}(Y,-)\ar[r]&\Ext_\G^{k+1}(X,-),
%}
\end{equation}
for {all $k\geq 0$.} The latter is the following commutative diagram with exact rows, associated with any given direct system $(M_i)_I$ in $\G$: 
\begin{equation}\label{comm_dia_fp_eq}
\xymatrix@C=20pt{ 0 \ar[r] & {\varinjlim}_I \G(Z,M_{i}) \ar[r] \ar[d]_{f_1} & {\varinjlim}_I \G(Y,M_{i}) \ar[r] \ar[d]_{f_2}& {\varinjlim}_I \G(X,M_{i}) \ar[r] \ar[d]_{f_3} & {\varinjlim}_I \Ext^{1}_{\G}(Z,M_{i}) \ar[d]_{f_4} \\ 0 \ar[r] & \G(Z,{\varinjlim}_I M_{i}) \ar[r] & \G(Y,{\varinjlim}_I M_{i}) \ar[r] & \G(X,{\varinjlim}_I M_{i}) \ar[r] & \Ext^{1}_{\G}(Z,{\varinjlim}_I M_{i}).}
\end{equation}

\noindent
(1). Suppose first that $n=0$, in which case $X\in\fg(\G)$ and $Y\in\fp(\G)$ by hypothesis. Take a direct system $(M_{i})_{I}$ in $\G$ and consider the diagram in \eqref{comm_dia_fp_eq}, then $f_2$ is an isomorphism since $Y\in\fp(\G)$ and $f_3$ is a monomorphism by Proposition~\ref{prop.fg-objects}. By the Five Lemma, $f_1$ is an isomorphism, and so $Z\in\fp(\G)$. \\
%%%%%%%%%%%%%%%%%Ya estaba comentado%%%%%%%%%%%%%%%%%%%%%%%%%%%%%%%%%%%%%%%%%%%
 %(1). Suppose first that $n=0$, in which case $X\in\fg(\G)$ and $Y\in\fp(\G)$ by hypothesis. Hence, also $Z\in\fg(\G)$, by  Proposition~\ref{prop.fg-objects}. Take a direct system $(M_{i})_{I}$ in $\G$ and consider the diagram in \eqref{comm_dia_fp_eq}, then $f_2$ is an isomorphism since $Y\in\fp(\G)$ and $f_3$ is a monomorphism by Proposition~\ref{prop.fg-objects}. By the Five Lemma, $f_1$ is then an epimorphism and, as it is also a monomorphism by Proposition~\ref{prop.fg-objects}, it is an isomorphism, and so $Z\in\fp(\G)$. \\
 Suppose now $n>0$, so that  $X\in\fp_n(\G)$ and $Y\in\fp_{n+1}(\G)$. In particular, $X,\, Y\in \fp(\G)$ and, by the case $n=0$, also $Z\in\fp(\G)$. Consider the long exact sequence in \eqref{long_ex_fp_eq}: {the second and fourth functors} in the sequence {vanish on direct limits of objects in $\Inj \G$,} for $k=1,\dots,{n-1}$, by {Lemma~\ref{lema_BP}. In particular, the functor $\Ext^{k}_{\G}(Z,-)$ vanishes on direct limits of objects in $\Inj\G$, for $k=2,\dots,n$. The fact that $\Ext^{1}_{\G}(Z,-)$ also vanishes on those direct limits follows by looking at the diagram \eqref{comm_dia_fp_eq}, with $(M_i)_I$ in $\Inj \G$. In such case the two right most horizontal arrows are epimorphisms since $Y\in \fp_{n+1}(\G)\subseteq\fp_2(\G) $, and hence $\varinjlim \Ext_{\G}^1(Y,M_i)=0=\Ext^{1}_{\G}(Y,\varinjlim M_i)$. Moreover $f_1$, $f_2$ and $f_3$ are isomorphisms, which implies that so is $f_4$ and hence $\Ext^{1}_{\G}(Y,\varinjlim M_i)=0$. Finally $Z\in \fp_{n+1}(\G)$ by Lemma~\ref{lema_BP}.}   %Consider  the long exact sequence in \eqref{long_ex_fp_eq}: the second, third, fifth and sixth functors in the sequence preserve direct limits for $k=0,1,\dots,n-2$, which implies that $\Ext_\G^{k+1}(Z,-)$ also preserves direct limits for $k=0,1,\dots,n-2$. On the other hand, for $k=n-1$ the second, third and fifth functor preserve direct limits, while \textcolor{blue}{\bf (Rev.\,3)} \textcolor{red}{by Proposition~\ref{prop_mono} }the canonical morphism ${\varinjlim}_J\Ext_\G^{n}(X,N_j)\to\Ext_\G^{n}(X,{\varinjlim}_JN_j)$ is a monomorphism, for any direct system $(N_j)_{J}$ in $\G$. It then follows by a repeated application of the Snake Lemma, that also $\Ext_\G^{n}(Z,-)$ preserves direct limits and so $Z\in\fp_{n+1}(\G)$.
%Original del parrafo de arriba
%Suppose now $n>0$, so that  $X\in\fp_n(\G)$ and $Y\in\fp_{n+1}(\G)$. In particular, $X,\, Y\in \fp(\G)$ and, by the case $n=0$, also $Z\in\fp(\G)$. Consider  the long exact sequence in \eqref{long_ex_fp_eq}: the second, third, fifth and sixth functors in the sequence preserve direct limits for $k=0,1,\dots,n-2$, which implies that $\Ext_\G^{k+1}(Z,-)$ also preserves direct limits for $k=0,1,\dots,n-2$. On the other hand, for $k=n-1$ the second, third and fifth functor preserve direct limits, while \textcolor{blue}{\bf (Rev.\,3)} \textcolor{red}{by Proposition~\ref{prop_mono} }the canonical morphism ${\varinjlim}_J\Ext_\G^{n}(X,N_j)\to\Ext_\G^{n}(X,{\varinjlim}_JN_j)$ is a monomorphism, for any direct system $(N_j)_{J}$ in $\G$. It then follows by a repeated application of the Snake Lemma, that also $\Ext_\G^{n}(Z,-)$ preserves direct limits and so $Z\in\fp_{n+1}(\G)$.

\smallskip\noindent
(2). Assume first that $n=0$. Given any direct system $(M_i)_{I}$ in $\G$, consider the diagram in \eqref{comm_dia_fp_eq}:
the map $f_1$ is an isomorphism, while $f_2$  and $f_4$ are monomorphisms {by Proposition~\ref{prop_mono}}. It follows that $f_3$ is a monomorphism, which, by Proposition~\ref{prop.fg-objects}, means that $X\in\fp_0(\G)=\fg(\G)$. 
\\
Suppose now that $n>0$, so that $Z\in \fp_{n+1}(\G) \subseteq \fp_2(\G)$ and $Y\in \fp_n(\G)\subseteq \fp(\G)$. In particular, $X\in \fp(\G)$ by a similar argument to the case $n=0$. On the other hand, {by Lemma~\ref{lema_BP}}, the first and {third}  functors in  \eqref{long_ex_fp_eq} {vanish on} direct limits {of objects in $\Inj\G$, for $k=1,\dots,n-1$. Hence, also $\Ext^{k}_{\G}(X,-)$ vanishes on} direct limits {of objects in $\Inj\G$, for $k=1,\dots,n-1$ and so $X\in \fp_n(\G)$, by Lemma~\ref{lema_BP}}. 

\smallskip\noindent
(3). For $n=0$, this follows by Proposition~\ref{prop.fg-objects}. \\
For $n=1$, consider the diagram in \eqref{comm_dia_fp_eq} for a direct system $(M_i)_{I}$ in $\G$:
by assumption, $f_1$ and $f_3$ are isomorphisms and, by Proposition~\ref{prop_mono}, $f_4$ is a monomorphism. By the Five Lemma, $f_2$ is an isomorphism, so $Y\in \fp(\G)$.\\
Finally, if $n>1$, $\fp_n(\G)$ is closed under extensions by Lemma~\ref{lema_BP} and the already proved case $n=1$.
\end{proof}

\begin{cor}\label{coro_1_fpn} 
Let $\G$ be a Grothendieck category. Then:
\begin{enumerate}[\rm (1)]
\item $\fp_0(\G)$ and $\fp_1(\G)$ are closed under extensions and cokernels of arbitrary morphisms. 
\end{enumerate}
Furthermore, the following statements hold true, for $n\geq 1$:
\begin{enumerate}[\rm (1)] \setcounter{enumi}{1}
\item $\fp_n(\G)$ is closed under extensions and cokernels of monomorphisms;
\item if $\fp_n(\G)=\fp_{n+1}(\G)$, then $\fp_n(\G)$ is also closed under kernels of epimorphisms;
\item $\fp_\infty(\G)$ is closed under extensions, kernels of epimorphisms and cokernels of monomorphisms.
\end{enumerate}
\end{cor}
\begin{proof}
(1). The statements about $\fp_0(\G)$ follow by Proposition~\ref{prop.fg-objects} while closure under extensions of $\fp_1(\G)$ is proved in Lemma~\ref{lem.n-fp subcategories}. To conclude, take a morphism $\phi$ in $\fp_1(\G)\subseteq \fp_0(\G)$. Therefore,  $\Im(\phi)\in \fp_0(\G)$ (by closure under quotients), so $\Coker(\phi)\in\fp_1(\G)$ by Lemma~\ref{lem.n-fp subcategories}(1). 

\smallskip\noindent
(2), (3) and (4) are immediate consequences of Lemma~\ref{lem.n-fp subcategories}.
\end{proof}

 Recall that a  subcategory $\X$ of an Abelian category $\C$ is  an {\bf Abelian exact subcategory }when it is Abelian and the inclusion functor $\X\hookrightarrow\C$ is exact. Equivalently, one can ask that $\X$ is closed under  finite coproducts and, given any morphism $\phi\colon X\to X'$ with $X,\,X'\in\X$, both the kernel and the cokernel of $\phi$, as computed in $\C$, do belong in $\X$.

\begin{cor}\label{coro_2_fpn}
Let $\G$ be a Grothendieck category. If $n\in\{0,1\}$  and $\fp_n(\G)=\fp_{n+1}(\G)$, then $\fp_n(\G)$ is an Abelian exact subcategory of $\G$.
\end{cor}
\begin{proof}
By Corollary~\ref{coro_1_fpn}, $\fp_n(\G)$ is closed under extensions (so also under finite coproducts) and arbitrary cokernels. Furthermore, as $\fp_n(\G)=\fp_{n+1}(\G)$, it is also closed under kernels of epimorphisms. Now, given $\phi\colon X\to Y$ in $\fp_n(\G)$ we have that $\Coker(\phi)\in \fp_n(\G)$, so $\Im(\phi)\in \fp_n(\G)$ (as this is the kernel of the epimorphism $Y\to \Coker(\phi)$). Thus, $\Ker(\phi)=\Ker(X\to \Im(\phi))\in \fp_n(\G)$, showing that $\fp_n(\G)$ is also closed under arbitrary kernels. 
%The key point is that any morphism in $\fp_0(\G)$ or in $\fp_1(\G)$ has image in $\fp_0(\G)=\fg(\G)$. ``Closedness under kernels of epimorphisms'' is then tantamount to ``closedness under kernels of arbitrary morphisms'' for these two subcategories, since they are  closed under taking cokernels. But the equality $\fp_n(\G)=\fp_{n+1}(\G)$ implies that this subcategory is closed under taking kernels of epimorphisms, . If $\fp_n(\G)=\fp_{n+1}(\G)$ and $n\in\{0,1\}$,  we then get that this subcategory is closed under taking cokernels, kernels and finite coproducts, as desired.     
\end{proof}

\subsection{Locally finitely presented and locally coherent Grothendieck categories}

Recall that an (Ab.$3$) Abelian category $\G$ is said to be 
\begin{itemize}
\item {\bf locally finitely generated} if $\fg(\G)$ is skeletally small and it generates $\G$;
\item  {\bf locally finitely presented} if $\fp(\G)$ is skeletally small and it generates $\G$.
\end{itemize}
In fact, this is equivalent to say that in $\G$ there is a  set of finitely generated (resp., finitely presented) generators. 
An (Ab.3) locally finitely presented Abelian category is automatically a Grothendieck category (see \cite[Section~(2.4)]{CB}); furthermore, in a locally finitely presented Grothendieck category, an object is finitely generated if and only if it is a quotient of a finitely presented object.

\begin{lema}\label{fg_lemma}
The following are equivalent for a Grothendieck category $\G$:
\begin{enumerate}[\rm (1)]
\item $\G$ is locally finitely generated;
\item for each $X\in \G$, there is $(X_i)_{I}\subseteq \fg(\G)$ directed, such that $X={\varinjlim}_{ I}X_i$
%\item $\fg(\G)$ is skeletally small and, given $X\in \G$, there is $(X_i)_{I}\subseteq \fg(\G)$ directed, such that $X={\varinjlim}_{ I}X_i$;
\item for each $X\in \G$, there is $(X_i)_{I}\subseteq \fg(\G)$ directed of subobjects of $X$, such that $X=\bigcup_{ I}X_i$.
%\item $\fg(\G)$ is skeletally small and, given $X\in \G$, there is a direct system $(X_i)_{ I}$ of finitely generated subobjects of $X$ such that $X=\bigcup_{ I}X_i$.

\end{enumerate}
\end{lema}
\begin{proof}
We claim that $\fg(\G)$ is skeletally small, whenever $\G$ is a Grothendieck category. Indeed, let $G$ be a generator of $\G$. Using the fact that each object in $\G$ is a directed union of those of its subobjects which are isomorphic to quotients of $G^n$ (see below), where $n$ is a natural number, we deduce that there is an injective map from (a skeleton of) $\fg(\G)$ to $\{G^n/X:n>0 \text{ and } X \text{ is a subobject of }G^n\}$, so that our claim follows. Hence the implications ``(3)$\Rightarrow$(2)$\Rightarrow$(1)'' are clear, so let us concentrate on the implication ``(1)$\Rightarrow$(3)''. Indeed, given $X\in \G$, there exists a set $I$, a finitely generated object $G_i$, for each $i\in I$, and an epimorphism $p\colon \coprod_{ I}G_i\to X$. Given a finite subset $J\subseteq I$, let \mbox{$\iota_J\colon \coprod_{J}G_j\to \coprod_{ I}G_i$} be the inclusion, and $X_J:=\Im(p\circ \iota_J)$. Then, $X\cong \bigcup\{X_J:J\subseteq I\text{ finite}\}$.
\end{proof}
As for the case of locally finitely generated categories, when $\G$ is locally finitely presented, any object of $\G$ can be written as a direct limit (though not as a direct union in general) of finitely presented objects. We want to give a  more detailed statement than Lemma~\ref{fg_lemma} (see Proposition~\ref{fp_prop}), for which {we first need to isolate} the following  trick due to Lazard \cite{L}, that is useful in a variety of situations. 

\begin{lema}[Lazard's Trick] \label{lem.Lazard Trick}
Let $\G$ be a Grothendieck category, consider two classes  \mbox{$\S\subseteq \mathcal K\subseteq \fg(\G)$} of finitely generated objects of $\G$ and suppose that the following properties hold:
\begin{enumerate}[\rm (1)]
\item both $\S$ and $\mathcal K$ are closed under finite coproducts;
\item $\mathcal K$ is closed under cokernels;
\item given $S\in\S$, $K\in \mathcal K$ and an arbitrary epimorphism $S\twoheadrightarrow K$, we have that $K\in \S$.
\end{enumerate}
Consider now an arbitrary  exact sequence of the form:
\[
\xymatrix@C=15pt{
\coprod_{\Lambda}K_\lambda\ar[rr]^-f&& \coprod_{I} S_i\ar[rr]^-q && X\ar[rr]&& 0,
}
\] 
with $(K_\lambda)_\Lambda\subseteq \mathcal K$, $(S_i)_I\subseteq \S$, and $X\in \G$. Then, $X\in \varinjlim\,\S$.
\end{lema}
\begin{proof}
Given finite subsets $\Lambda'\subseteq \Lambda$ and $I'\subseteq I$ denote, respectively, by 
\[
\xymatrix{
\iota_{\Lambda'}\colon \coprod_{\Lambda'}K_\lambda\longrightarrow \coprod_{\Lambda}K_\lambda\qquad\text{and}\qquad\varepsilon_{I'}\colon \coprod_{I'}S_i\longrightarrow \coprod_{I}S_i
}
\]
the inclusions into the coproduct, and define the following set 
\[
\Upsilon:=\{(\Lambda',I'):\Lambda'\subseteq \Lambda,\, I'\subseteq I\ \text{finite},\ \text{such that $f\circ\iota_{\Lambda'}$ factors through $\varepsilon_{I'}$}\}.
\]
In other words, $(\Lambda',I')\in \Upsilon$ if and only if we have a commutative diagram like the following one:
\[
\xymatrix@R=-2pt@C=50pt{
\coprod_{\Lambda'}K_\lambda\ar[rr]^{f\circ\iota_{\Lambda'}}\ar@/_+4pt/[dr]_{f_{(\Lambda',I')}}&&\coprod_{I}S_i.\\
&\coprod_{I'}S_i\ar@/_+4pt/[ur]_{\varepsilon_{I'}}
}
\] 
Note  that, if it exists, the map $f_{(\Lambda',I')}$ is uniquely determined by $f$. Endow $\Upsilon$ with the product order, {i.e., $(\Lambda',I')\leq (\Lambda'',I'')$ if and only if $\Lambda'\subseteq \Lambda''$ and $I'\subseteq I''$}. Since, for each finite subset $I'\subseteq I$, the map $\varepsilon_{I'}$ is a monomorphism, one  sees that the poset $\Upsilon$ is directed and ${\varinjlim}_\Upsilon f_{(\Lambda',I')}=f$. Therefore, by the right exactness of ${\varinjlim}_\Upsilon$, we have that $X\cong\Coker(f)\cong{\varinjlim}_\Upsilon\Coker(f_{(\Lambda',I')})$, with  $\Coker(f_{(\Lambda',I')})\in \S$, for all $(\Lambda',I')\in\Upsilon$, by hypotheses (1--3).%, since $\S$ is closed under taking quotients in $\fp(\G)$. 
\end{proof}

\begin{prop}\label{fp_prop}
Let $\G$ be a locally finitely presented Grothendieck category, $\S\subseteq \fp(\G)$ a set of finitely presented generators of $\G$, and $X$ an object of $\G$. Then, the following statements hold true:
\begin{enumerate}[\rm (1)]
\item $X$ can be written as a direct limit of objects in $\bar \S:=\{\text{quotients of objects in $\mathrm{sum}(\S)$}\}\cap \fp(\G)$;

%\item $X$ can be written as a direct limit of objects in $\bar \S:=\{\text{fin.\ pres.\ quotients of objects in $\mathrm{sum}(\S)$}\}$;
\item actually, $\fp(\G)=\bar \S=\pres_{{1}}(\S)$.
\end{enumerate}
\end{prop}
\begin{proof}
(1). We have two classes $\bar \S\subseteq \fp(\G)$ that satisfy the three hypotheses of Lazard's Trick. Furthermore, as $\S$ is a set of finitely presented generators, there is an exact sequence of the form:
\[
\xymatrix@C=15pt{
\coprod_{\Lambda}K_\lambda\ar[rr]^-f&& \coprod_{I} S_i\ar[rr]^-q && X\ar[rr]&& 0,
}
\] 
with $(K_\lambda)_\Lambda\subseteq \S\subseteq \fp(\G)$, $(S_i)_I\subseteq \S\subseteq \bar \S$. Then, $X\in \varinjlim\,\bar \S$.

\smallskip\noindent
(2). Let $X\in \fp(\G)$ so, by (1), {we can write $X=\varinjlim_I {S}_i$, with ${S}_i\in \bar{\S}$ for all $i\in I$. Since $X\in \fp(\G)$, the identity $\id_X\colon X\to X$ factors through the canonical morphism $\iota_j\colon {S}_j \stackrel{}{\to}  \varinjlim_I {S}_i=X$, for some $j\in I$. Hence, }$X\in \add(\bar \S)$. Therefore, $\fp(\G)=\bar \S$, as $\bar \S=\add(\bar \S)$. Take now $Y\in \bar \S$ and consider a short exact sequence
%(2). Suppose $X\in \fp(\G)$. By part (1), $X\in \varinjlim\, \bar \S$ and, since $X\in \fp(\G)$, even $X\in \add(\bar \S)$. Therefore, $\fp(\G)=\bar \S$, as $\bar \S=\add(\bar \S)$. Finally, let $Y\in \bar \S$ and consider a short exact sequence
$
0\to K\to S\to Y\to 0,
$
where $S\in \mathrm{sum}(\S)$ and so $K\in \fg(\G)$. Since $\Gen(\S)=\G$, there is an epimorphism $p\colon\coprod_{\Lambda}T_\lambda\twoheadrightarrow K$, with $T_\lambda\in \S$ for all $\lambda\in \Lambda$. For each finite subset $J\subseteq \Lambda$, let $\iota_J\colon \coprod_{J}T_\lambda\to \coprod_{\Lambda}T_\lambda$ be the inclusion. Then, $K\cong \bigcup{}_{J\subseteq \Lambda\text{ finite}}\Im(p\circ \iota_J)$ is a direct union and, since $K\in\fg(\G)$, there is a finite $J\subseteq \Lambda$ such that $K=\Im(p\circ \iota_J)$, showing that $Y\in \pres_{1}(\S)$, as desired.
\end{proof}

A locally finitely presented Grothendieck category $\G$ is said to be 
\begin{itemize}
\item {\bf locally coherent} if finitely generated subobjects of finitely presented objects are finitely presented. Equivalently, one may ask that $\fp(\G)$ is an exact Abelian subcategory of $\G$.
\end{itemize}
Note that, to ask that $\fp(\G)$ is an exact Abelian subcategory of $\G$, is the same as asking that $\fp(\G)$ is closed under taking kernels in $\G$. 
%\item if, for each epimorphism $\phi\colon S\to F$ with $S\in \mathrm{sum}(\S)$ and $F\in \fp(\G)$, we have that $\Ker(\phi)\in \fp(\G)$, then $\fp(\G)=\pres_2(\S)$.
%
%
%\smallskip\noindent
%(3). Given $X\in \fp(\G)$, by part (2) there is an epimorphism $\phi\colon S\twoheadrightarrow X$. Furthermore, by hypothesis, $\Ker(\phi)\in \fp(\G)$, \NB
In the following proposition we collect some equivalent characterizations of locally coherent Grothendieck categories.

\begin{prop}\label{prop. Manolo}
The following assertions are equivalent for a Grothendieck category $\G$, where $E$ is an injective cogenerator of $\G$:
\begin{enumerate}[\rm (1)]
\item $\G$ is locally coherent;
\item  $\G$ is locally finitely presented and $\fp(\G)=\fp_\infty(\G)$ (resp., $\fp(\G)=\fp_2(\G)$);
\item   $\G$ is locally finitely presented and, for all $X\in \fp(\G)$, $\Ext_\G^1(X,-)$ vanishes on $\varinjlim(\Prod(E))$.
%\item  $\G$ has a set $\mathcal{S}$ of finitely presented generators  that satisfies the following two conditions:
%\begin{enumerate}[\rm ({4.}1)]
% \item for each $S\in\mathcal{S}$, the functor $\Ext_\G^1(S,-)\colon\G\to\Ab$ vanishes on direct limits of objects of $\Prod(E)$;
% \item for each  $X\in\fp(\G)$, there is an epimorphism $p\colon \coprod_{i=1}^nS_i\to X$, with the $S_i$  in $\mathcal{S}$, that has a finitely presented kernel. 
%\end{enumerate}
\end{enumerate}
\end{prop}
\begin{proof}
It is a direct consequence of Lemma~\ref{lema_BP}, Corollary~\ref{coro_2_fpn} and  \cite[Proposition~3.5(2)]{Sa}.
\end{proof}

\section{Preliminaries on torsion pairs}\label{Sec_Prelim_on_tors}

In this section we give the necessary definitions and results about torsion pairs in Grothendieck categories, with special emphasis on those torsion pairs that are of finite type and/or restrict to $\fp(\G)$. Some of these results are already known; in those cases we omit the proofs. 

\subsection{Generalities on torsion pairs in Grothendieck categories}
Let $\G$ be {an Abelian} category. For a class of objects $\X$ in $\G$, we use the following notations:
\[
\X^{\perp}:=\{C \in \G:\G(X,C)=0, \text{ for all }X \in \X\}\quad\text{and}\quad{}^{\perp}\X:=\{C \in \G :\G(C,X)=0, \text{ for all }X \in \X\}.
\] 
A {\bf torsion pair} in $\G$ is a pair $\mathbf{t}=(\T,\F)$ of subcategories satisfying the following two conditions:
\begin{enumerate}[\rm ({Tors.}1)]
\item $\T={}^{\perp}\F$ and $\F=\T^{\perp}$;
\item for each object $X$ of $\G$ there are objects $T_X \in \T$ and $F_X\in \F$, and an exact sequence
\[
\xymatrix{
0 \ar[r] & T_X \ar[r] & X \ar[r] & F_X \ar[r] & 0.
}
\]
\end{enumerate}
%A {\bf torsion pair} in $\G$ is a pair $\mathbf{t}=(\T,\F)$ of subcategories satisfying the following two conditions:
%\begin{enumerate}[\rm ({Tors.}1)]
%\item $\T={}^{\perp}\F$ and $\F=\T^{\perp}$;
%\item for each object $X$ of $\G$ there are objects $T_X \in \T$ and $F_X\in \F$, and an exact sequence
%\[
%\xymatrix{
%0 \ar[r] & T_X \ar[r] & X \ar[r] & F_X \ar[r] & 0.
%}
%\]
%\end{enumerate}
{When $\G$ is a Grothendieck category condition (Tors.2) is actually a consequence of (Tors.1) (see Lemma~\ref{construction_of_tor_rad} below).} We say that $\T$ is a {\bf torsion class} and $\F$ is a {\bf torsionfree class}. In the above sequence, $T_X$ and $F_X$ depend functorially on $X$, and the corresponding functors:
\[
t\colon \G\longrightarrow \T\qquad \text{and}\qquad (1:t)\colon \G\longrightarrow \F
\]
are called, respectively, the {\bf torsion radical} and the {\bf torsion coradical}. In fact, $t$ is the right adjoint of the inclusion $\T\to \G$, while $(1:t)$ is the left adjoint to the inclusion $\F\to \G$. In what follows we  often abuse notation and consider $t$ and $(1:t)$ as endofunctors $\G\to \G$.

\begin{lema}[{\rm \cite[Proposition~VI.2.1]{St}}]\label{construction_of_tor_rad}
Let $\G$ be a Grothendieck category. A class $\T\subseteq \G$ is torsion if, and only if, it is closed under taking quotients, extensions and coproducts. Dually, a class $\F\subseteq \G$ is  torsionfree if, and only if, it is closed under  subobjects, extensions and products.
\end{lema}
%\begin{proof}
%It is easy to prove that any torsion class is closed under taking quotients, extensions and coproducts. On the other hand, given $\T$ with these closure properties and letting $\F:=\T^{\perp}$, it is not difficult to verify that $\T={}^{\perp}\F$ so (Tors.1) is verified. Furthermore, given $X\in \Ob(\G)$, let $T_X:=\sum\{T\leq X:T\in \T\}$ and $F_X:=X/T_X$. It is then easy to prove that $T_X\in \T$ and $X/T_X\in \F$, so also (Tors.2) is satisfied.
%\end{proof}

A torsion class $\T$ is a {\bf TTF} (=torsion and torsionfree) class if it is closed under taking subobjects and products. For such a $\T$, one calls $(\C:=^{\perp}\T,\T,\F:=\T^{\perp})$  a {\bf TTF triple} where, by Lemma~\ref{construction_of_tor_rad}, both $(\C,\T)$ and $(\T,\F)$ are torsion pairs called, respectively, the {\bf left} and {\bf right constituent} of the TTF triple.

\subsection{Torsion pairs of finite type}

Consider a Grothendieck category $\G$ and let $\t=(\T,\F)$ be a torsion pair in $\G$, with torsion radical  $t\colon \G\to \G$. Then, $\t$ is said to be
\begin{itemize}
\item {\bf hereditary} if $t\colon \G\to \G$ is a left-exact functor. Equivalently, one may ask that $\T$ is closed under subobjects or that $\F$ is closed under taking injective envelopes {(see \cite[Proposition~VI.3.2]{St})};
\item {\bf of finite type} if $\F$ is closed under taking direct limits, in symbols $\F=\varinjlim\, \F$.
\end{itemize} 

In the following lemma we give some equivalent characterizations of torsion pairs of finite type; we omit the proof as the argument is  analogous to the one commonly used in categories of modules.

\begin{lema}\label{preservation_colimits_fin_type}
The following are equivalent for a torsion pair $\t=(\T,\F)$ in a Grothendieck category  $\G$:
\begin{enumerate}[\rm (1)]
\item $\t$ is of finite type;
\item the torsion radical $t\colon \G \rightarrow \G$ preserves direct limits;
\item the torsion coradical $(1:t)\colon \G \rightarrow \G$ preserves direct limits.
\end{enumerate}
\end{lema}

Given a torsion pair of finite type, we have the following criterion for an object to be finitely presented:

\begin{prop} \label{prop.fpobjects-wrt-torsionpairs}
Let $\G$ be a Grothendieck category and $\t=(\T,\F)$ a torsion pair of finite type in $\G$. The following assertions are equivalent for an object $M\in \G$:
\begin{enumerate}[\rm (1)]
\item $M$ is finitely presented;
\item $M$ satisfies the following three conditions:
\begin{enumerate}[\rm ({2.}1)]
\item $\G(M,-)\colon \G\to\Ab$ preserves direct limits of objects  in $\T$;
\item $\G(M,-)_{\restriction\F}\cong\F((1:t)(M),-)\colon \F\to\Ab$ preserves direct limits;
\item the  canonical map ${\varinjlim}_I\Ext_\G^1 (M,T_i)\to\Ext_\G^1(M,{\varinjlim}_I T_i)$ is monic, for all $(T_i)_{I}\subseteq\T$ directed. 
\end{enumerate}
\end{enumerate}
\end{prop}
\begin{proof}
Almost by definition, condition (1) implies conditions (2.1) and (2.2), while condition (2.3) follows by {Proposition~\ref{prop_mono}}. Conversely, let us assume that $M$ satisfies condition (2) and let us check that it is  finitely presented. Take a direct system $(N_i)_{I}$ in $\G$ and  the following  exact sequences:
\[
(0\to t(N_i)\to N_i\to (1:t)(N_i)\to 0)_{ I},\quad 0\to {\varinjlim}_I t(N_i)\to {\varinjlim}_I N_i\to {\varinjlim}_I (1:t)(N_i)\to 0.
\]
Applying the functor $\G(M,-)$, we obtain the following commutative diagram with exact rows:
\[
\xymatrix@R=18pt@C=15pt{
0\ar[r]&{\varinjlim}_I \G(M,t(N_i))\ar[d]^{f_1}\ar[r]&{\varinjlim}_I \G(M,N_i)\ar[r]\ar[d]^{f_2}& {\varinjlim}_I \G(M,(1:t)(N_i))\ar[r]\ar[d]^{f_3}&{\varinjlim}_I \Ext^{1}_\G(M,t(N_i))\ar[d]^{f_4}\\
0\ar[r]& \G(M,{\varinjlim}_I t(N_i))\ar[r]& \G(M, {\varinjlim}_I N_i)\ar[r]&  \G(M,{\varinjlim}_I(1:t)(N_i))\ar[r]& \Ext^{1}_\G(M,{\varinjlim}_I t(N_i)),
}
\]
with $f_1,\dots,f_4$ induced by universality of colimits. Observe that $f_1$ and $f_3$ are  isomorphisms by (2.1) and (2.2), respectively, and $f_4$ is monic by (2.3). Hence, $f_2$ is  an isomorphism by the Five Lemma.
\end{proof}

In the following lemma we give a criterion for a torsion object to be finitely presented. 

\begin{lema}\label{fp_torsion_lema}\label{prop.fp torsion objects}
Let $\G$ be a Grothendieck category and $\t=(\T,\F)$ a torsion pair of finite type in $\G$. Given $T\in \T$ such that ${\G}(T,-)_{\restriction \T}\colon \T \rightarrow \Ab$ preserves direct limits, we have that  $T\in \fp(\G)$.
\end{lema}
\begin{proof}
Given $(M_i)_{I}\subseteq\G$ directed, ${\varinjlim}_I t(M_i) \cong t({\varinjlim}_I M_i)$,  {by Lemma~\ref{preservation_colimits_fin_type}}. We now conclude as follows:
\begin{equation*}
\xymatrix{\varinjlim_I {\G}(T, M_{i})\cong \varinjlim_I {\G}(T,t(M_{i})) \cong  {\G}(T, \varinjlim_I t(M_i)) \cong {\G}(T,t(\varinjlim_IM_i)) \cong {\G}(T, \varinjlim_I M_i).}
\qedhere
\end{equation*}
\end{proof}

In the following subsection we give a series of characterizations of (hereditary) torsion pairs of finite type in a locally finitely presented  category $\G$ (see Lemma~\ref{ext_closed_fp_gives_torsion}). An important piece of this is the following criterion, that holds in a general Grothendieck category, for an extension of two objects in ${\varinjlim}\,\Scal$ to belong in $\Gen(\S)$, where $\S$ is a class of finitely presented objects. 

\begin{lema}\label{lem. fp(A) is ab}
Let $\G$ be a Grothendieck category and $\S\subseteq \fp(\G)$ a subcategory closed under extensions. If $\Ext_\G^1(S,-)\colon \G\to\Ab$ preserves direct limits of objects in $\S$, for all $S\in \S$, then an extension in $\G$ of two objects in ${\varinjlim}\,\Scal$ belongs in $\Gen(\S)$.
\end{lema}
\begin{proof}
Let $(S_i)_{ I}$ and $(T_j)_{ J}$ be direct systems in $\Scal$ and let 
\begin{equation}\label{fp_is_ab_1_eq}
\xymatrix@C=15pt{
0\ar[rr]&&{\varinjlim}_I S_i\ar[rr]^-u&& X\ar[rr]^-p &&{\varinjlim}_J T_j\ar[rr]&& 0
}
\end{equation}
be an exact sequence in $\G$. For each $k\in J$, the pullback of the canonical map $T_k\to{\varinjlim}_J T_j$ with $p$ yields an exact sequence $0\to {\varinjlim}_I S_i\to X_k\to T_k\to 0$. When $k$ varies in $J$, we obtain a direct system of exact sequences in $\G$ whose direct limit is the sequence in \eqref{fp_is_ab_1_eq}. In particular, ${\varinjlim}_J X_j=X$. Hence, we are reduced to check that  $Y\in{\varinjlim}\,\Scal$ if there is an exact sequence like the following, with $T\in\Scal$:
\begin{equation}\label{fp_is_ab_2_eq}
\xymatrix@C=15pt{
\epsilon:&0\ar[rr]&&{\varinjlim}_I S_i\ar[rr]^-u&& Y\ar[rr] && T\ar[rr]&& 0,
}
\end{equation}
Let $[\epsilon]\in \Ext_{\G}^1(T,{\varinjlim}_I S_i)$ be the equivalence class of the extension $\epsilon$ in \eqref{fp_is_ab_2_eq}. By assumption, the canonical map $\Psi\colon {\varinjlim}_I\Ext_{\G}^1(T,S_i)\to\Ext_{\G}^1(T,{\varinjlim}_I S_i)$ is an isomorphism, so there is $k\in I$ and $[\epsilon_k]\in\Ext_{\G}^1(T,S_k)$ such that $[\epsilon]=\Psi\circ\iota^*_k([\epsilon_k])$, where $\iota^*_k\colon  \Ext_{\G}^1(T,S_k)\to \Ext_{\G}^1(T,{\varinjlim}_I S_i)$ is  induced by the {canonical map} \mbox{$\iota_k\colon S_k\to {\varinjlim}_I S_i$}. We then get  a commutative diagram with exact rows
\begin{equation}\label{epsilon_is_in_the_image}
\xymatrix@C=35pt@R=16pt{
\epsilon_k:&0 \ar[r] & S_k \ar@{}[dr]|{\text{P.O.}}\ar[r]^{u_k} \ar[d]_{\iota_k} & Y_k \ar[r] \ar[d] & T \ar[r] \ar@{=}[d]& 0 \\ 
\epsilon:&0 \ar[r] & {\varinjlim}_I S_i \ar[r] & Y \ar[r]  & T \ar[r] & 0,
}
\end{equation}
where the square on the left is a pushout. 
The subset $I_{\geq k}:=\{i\in I:i\geq k\}\subseteq I$ is cofinal. For each $h\in I_{\geq k}$, take the pushout of $\iota_{k,h}\colon S_k\to S_h$ with the map $u_k\colon S_k\to Y_k$ defined in \eqref{epsilon_is_in_the_image}; we obtain a direct system of exact sequences $(0\rightarrow S_h\stackrel{u_h}{\longrightarrow}Y_h\to T\rightarrow 0)_{I_{\geq k}}\subseteq \fp(\G)\cap \Scal$. Taking colimits we get:
\[
\xymatrix@C=35pt@R=16pt{ 
0 \ar[r] & {\varinjlim}_{I_{\geq k}} S_h\ar@{}[dr]|{\text{P.O.}} \ar[r] \ar[d] & {\varinjlim}_{I_\geq k}Y_h \ar[r] \ar[d] & T \ar[r] \ar@{=}[d] & 0\\ 
0 \ar[r] & {\varinjlim}_I S_i \ar[r] & Y \ar[r]  & T \ar[r] & 0,
}
\]
(direct limits are left adjoints, so they preserve pushouts). By cofinality, the leftmost vertical map is an isomorphism, so $Y\cong {\varinjlim}_{I_{\geq k}}Y_h\in{\varinjlim}\,\Scal$, by the Five Lemma.
\end{proof}

\subsection{Torsion pairs of finite type in locally finitely presented categories}
In this subsection we start with a locally finitely presented Grothendieck category $\G$ and characterize its torsion pairs of finite type. An important step is the observation that, by Lemma~\ref{lem. fp(A) is ab}, and given a subclass $\S\subseteq \fp(\G)$ satisfying the hypotheses of the lemma, if ${\varinjlim}\,\S=\Gen(\S)$, then this class is closed under extensions. Observe that, for a direct system $(S_i)_{I}\subseteq\S$, letting $S_{(j,i)}:=S_j$ for all $j\leq i\in I$,  there is a presentation:
\[
\xymatrix{
\coprod_{j\leq i\in I}S_{(j,i)}\longrightarrow\coprod_{i\in I}S_i\longrightarrow{\varinjlim}_I S_i\longrightarrow 0.
}
\] 
Hence, ${\varinjlim}\,\S\subseteq \Pres(\S)\subseteq\Gen(\S)$. In particular, if ${\varinjlim}\,\S=\Gen(\S)$, then also $\Pres(\S)=\Gen(\S)$. 

\begin{lema}\label{ext_closed_fp_gives_torsion}
Let $\G$ be a locally finitely presented Grothendieck category and suppose that there is an $\S\subseteq \fp(\G)$, closed under extensions and quotients in $\fp (\G)$, such that $\Ext_\G^1(S,-)\colon \G\to\Ab$ preserves direct limits of objects in $\S$, for all $S\in \S$. Then, ${\varinjlim}\,\S=\Gen(\S)$ and this is a torsion class.
\end{lema}
\begin{proof}
It is enough to verify that ${\varinjlim}\,\S=\Gen(\S)$ as, in that case, this class is closed under extensions (by Lemma~\ref{lem. fp(A) is ab}), coproducts and quotients. Consider a short exact sequence
\[
\xymatrix@C=15pt{
0\ar[rr]&& K\ar[rr]^-u&& \coprod_{I} T_i\ar[rr]^-q && T\ar[rr]&& 0,
}\qquad\text{with $T_i\in \S$, for all $i\in I$.}
\]
As $\G$ is locally finitely presented, there is an epimorphism $p\colon\coprod_{\Lambda}K_\lambda\twoheadrightarrow K$, with $K_\lambda\in\fp(\G)$, for all $\lambda\in \Lambda$. To conclude that  $T\in{\varinjlim}\,\S$, apply Lazard's Trick (with $\mathcal K=\fp(\G)$) to:
\[
\xymatrix@C=15pt{
\coprod_{\Lambda}K_\lambda\ar[rr]^-{u\circ p}&& \coprod_{I} T_i\ar[rr]^-q && T\ar[rr]&& 0.
}\qedhere
\]  
\end{proof}

\begin{defi} \label{def.torsion-pair-fp(strongly)generated}
A torsion pair $\mathbf{t}=(\T,\F)$ is said to be 
\begin{itemize}
\item {\bf generated} {\bf by finitely presented objects} when there is a set $\mathcal{S}\subseteq\fp(\G)$ such that $\F=\mathcal{S}^\perp$;
\item {\bf strongly generated} {\bf by finitely presented objects} when $\T={\varinjlim} (\T\cap\fp(\G))$. 
\end{itemize}
\end{defi}

We can now prove our main characterization of torsion pairs of finite type:

\begin{prop} \label{prop_fixing_first_part_of_old_4.1}
Let $\G$ be a locally finitely presented Grothendieck category and $\mathbf{t}=(\mathcal{T},\mathcal{F})$  a torsion pair in $\G$. Consider the following assertions:
\begin{enumerate}[\rm (1)]
\item $\mathbf{t}$ is strongly generated by finitely presented objects;
\item there is a set $\mathcal{S}\subseteq\fp(\G)$ such that $\mathcal{T}=\Gen(\mathcal{S})$ (or, equivalently, $\mathcal{T}=\Pres(\mathcal{S})$);
\item $\mathbf{t}$ is generated by  finitely presented objects;
\item $\mathbf{t}$ is of finite type. 
\end{enumerate}
Then, the implications {\rm``(1)$\Leftrightarrow$(2)$\Rightarrow$(3)$\Rightarrow$(4)''} hold and, if $\mathbf{t}$ is hereditary, the assertions are all equivalent. 
\end{prop}
\begin{proof}
(1)$\Rightarrow$(2). Let $\mathcal{S}$ be a skeleton of $\mathcal{T}\cap\fp(\G)$, then $\mathcal{T}={\varinjlim}\, \S\subseteq\Pres(\mathcal{S})\subseteq\Gen(\mathcal{S})\subseteq\mathcal{T}$.

\smallskip\noindent
(2)$\Rightarrow$(1). Given $T\in\mathcal{T}=\Gen(\mathcal{S})$, with $\mathcal{S}\subseteq\fp(\G)$, take a short exact sequence
\[
\xymatrix{
0\longrightarrow K\stackrel{u}{\longrightarrow}\coprod_{ I}S_i\stackrel{p}{\longrightarrow}T\longrightarrow 0.
}\qquad\text{with $S_i\in \S$, for all $i\in I$,}
\] 
and pick a direct system $(K_\lambda )_{\Lambda}\subseteq \fp(\G)$ such that ${\varinjlim}_{\Lambda} K_\lambda\cong K$. 
Apply now Lazard's Trick to the inclusion of classes $\T\cap \fp(\G)\subseteq \fp(\G)$, to conclude that $T\in \varinjlim\, (\T\cap \fp(\G))$.

\smallskip\noindent
(2)$\Rightarrow$(3)$\Rightarrow$(4) are clear implications.

\smallskip\noindent 
(4)$\Rightarrow$(2), {\em provided $\mathbf{t}$ is hereditary}. This follows by \cite[Proposition~11.1.14]{Prest}.
\end{proof}

\begin{rem}
When $\G=\mod R$ is a category of modules, we will prove in Theorem~\ref{thm.locally-fp-hearts} (see also Question~\ref{ques_for_mod_gen_fp}) that conditions $(1)$, $(2)$ and $(3)$ of the above proposition are all equivalent in this case. On the other hand, even in categories of modules, one can produce examples of torsion pairs that satisfy $(4)$ but not $(3)$, e.g., see \cite[Corollary~4.2 and Example~4.3]{BPa} for a torsion pair $\t=(\C_{\frak a},\T_{\frak a})$ (which is the left constituent of a TTF triple, so it is necessarily of finite type) such that $\C_{\frak a}\cap \fp(\G)=0$.
\end{rem}

\subsection{Torsion pairs that restrict to $\fp(\G)$}

Consider a locally finitely presented Grothendieck category $\G$ and a torsion pair $\t=(\T,\F)$ in $\G$, with torsion radical  $t\colon \G\to {\G}$. Then $\t$ is said to
\begin{itemize}
\item {\bf restrict to $\fp(\G)$} if the torsion radical $t\colon\G\to\G$ preserves finitely presented objects. 
\end{itemize} 
Torsion pairs of finite type that restrict to {$\fp(\G)$}  give a useful criterion for local coherence:

\begin{prop} \label{prop.localcoherence-via-torsion}
A  locally finitely presented Grothendieck category $\G$ with  a torsion pair of finite type $\t=(\T ,\F)$ that restricts to $\fp(\G)$ is locally coherent if, and only if, the following conditions hold:
\begin{enumerate}[\rm (1)]
\item kernels of morphisms in $\T\cap\fp(\G)$ are finitely presented;
\item kernels of morphisms in $\F\cap\fp(\G)$ are finitely presented;
\item kernels of morphisms with source in $\F\cap\fp(\G)$ and target in $\T\cap\fp(\G)$ are finitely presented.
\end{enumerate}
\end{prop}
\begin{proof}
The ``only if'' part is clear. Conversely, let $f\colon M\to N$ be a morphism in $\fp (\G)$. Since $\t$ restricts to $\fp (\G)$ we have a commutative diagram in $\fp (\G)$, with exact rows:
\[
\xymatrix@C=35pt@R=17pt{ 0 \ar[r] & t(M) \ar[r] \ar[d]_{t(f)} & M \ar[r] \ar[d]^{f} & (1:t)(M) \ar[r] \ar[d]^{(1:t)(f)} & 0\\ 
0 \ar[r] & t(N) \ar[r] & N \ar[r] & (1:t)(N) \ar[r] & 0.}
\]
The Snake Lemma gives the following long exact sequence 
\[
0\rightarrow\Ker(t(f))\to\Ker(f)\to\Ker[(1:t)(f)]\stackrel{w}{\longrightarrow}\Coker(t(f))\to\Coker(f)\to\Coker[(1:t)(f)]\rightarrow 0, 
\]
where all cokernels are finitely presented, by Corollary~\ref{coro_1_fpn}, $\Ker(t(f))\in \fp(\G)$ by hypothesis (1), and $\Ker[(1:t)(f)]\in \fp(\G)$ by hypothesis (2). Furthermore, $\Ker[(1:t)(f)]\in\F$ and $\Coker(t(f))\in\T$, so $\Ker(w)\in\fp(\G)$ by hypothesis (3). Hence, $\Ker(f)\in\fp(\G)$, by Corollary~\ref{coro_1_fpn}. Therefore, $\fp (\G)$ is closed under taking kernels in $\G$, as desired.
\end{proof}

\section{Preliminaries on derived categories and $t$-structures}\label{Sec_Prelim_on_der}

In this section we recall some definitions and basic facts about derived categories and $t$-structures. This serves mostly  to introduce notation, fix conventions and recall known results that help the paper to be more self-contained. For this reason, most  proofs are omitted.

\subsection{Triangulated and derived categories}
Given a triangulated category $\mathcal D$ (for the definition, see  \cite{N}), we denote by $[1]\colon\mathcal D\to \mathcal D$ the {\bf suspension functor} and $[n]:=[1]^n$ for all $n\in\Z$. Furthermore, we denote ({\bf distinguished}) {\bf triangles} in $\mathcal D$ by  $X\to Y \to Z\to X[1]$.

\smallskip
Given a set $\mathcal X$ of objects in $\mathcal D$ and a subset $I\subseteq \Z$, we let
\begin{align*}
\mathcal{X}^{\perp_{I}}&:=\{Y\in\mathcal D:\mathcal D(X,Y[i])=0\text{, for all }X\in\mathcal{X}\text{ and }i\in I\}\\
{}^{\perp_{I}}\mathcal{X}&:=\{Z\in\mathcal D:\mathcal D(Z,X[i])=0\text{, for all }X\in\mathcal{X}\text{ and }i\in I\}.
\end{align*}
If $I=\{i\}$ for some $i\in \Z$, then $\mathcal{X}^{\perp_{i}}:=\mathcal{X}^{\perp_{I}}$ and ${}^{\perp_{i}}\mathcal{X}:={}^{\perp_{I}}\mathcal{X}$. If $i=0$,  $\mathcal{X}^{\perp_{}}:=\mathcal{X}^{\perp_{0}}$ and ${}^{\perp_{}}\mathcal{X}:={}^{\perp_{0}}\mathcal{X}$. A set $\S \subseteq \mathcal D$ is called a {\bf set of generators} (resp., {\bf cogenerators}) of $\mathcal D$ if   $\S^{\perp_\Z}=0$ (resp., ${}^{\perp_\Z}\S=0$). 
\\
If $\mathcal D$ has coproducts, we  say that an object $X$ is {\bf compact} when the functor $\mathcal D(X,-) \colon \mathcal D\to \Ab$ preserves coproducts.  $\mathcal D$ is {\bf compactly generated} when it has a set of compact generators.

\smallskip
Given an Abelian category $\C$, a subcategory $\C'\subseteq \C$, and an integer $n\in\Z$,
\begin{itemize}
\item $\Ch(\C)$ is  the category of the (cohomologically graded) {\bf cochain complexes} of objects of $\C$;
\item $\Ch(\C')$ is the subcategory of $\Ch(\C)$ of the complexes that are, degree-wise, in $\C'$;
\item $\Ch^{\leq n}(\C')\subseteq \Ch(\C')$ is the subcategory complexes concentrated in degrees $\leq n$;
\item $\Ch^{\geq n}(\C')\subseteq \Ch(\C')$ is the subcategory complexes concentrated in degrees $\geq n$;
\item$\Ch^-(\C')=\bigcup_{n\in\mathbb{Z}}\Ch^{\leq n}(\C')$ (resp., $\Ch^+(\C')=\bigcup_{n\in\mathbb{Z}}\Ch^{\geq n}(\C')$);
\item $\K(\C)$ is the {\bf homotopy category} of $\C$;
\item $\K(\C')\subseteq\K(\C)$ is the category of all complexes isomorphic, in $\K(\C)$, to a complex in $\Ch(\C')$;
\item $\K^{\leq n}(\C')$ (resp., $\K^{\geq n}(\C')$,  $\K^-(\C')$, $\K^+(\C')$) is the subcategory of $\K(\C')$ of the complexes which are isomorphic, in $\K(\C)$, to one in $\Ch^{\leq n}(\C')$ (resp., $\Ch^{\geq n}(\C')$, $\Ch^-(\C')$, $\Ch^+(\C')$).
\end{itemize}
$\Der(\C)$ denotes the (unbounded) {\bf derived category} of $\C$,  defined as the localization obtained inverting the quasi-isomorphisms in $\K(\C)$. In general the category $\Der(\C)$ may not exist, in the sense that $\Der(\C)(X,Y)$ may be a proper class for some $X,\, Y \in \Der(\C)$; this is never the case when $\C$ is a Grothendieck category. When it exists, $\Der(\C)$ is canonically a triangulated category. As for the homotopy category, $\Der^{\leq n}(\C)$ and $\Der^{\geq n}(\C)$ are defined as the subcategories of $\Der(\C)$ of those complexes that are isomorphic in $\Der(\C)$ to some complex in $\Ch^{\leq n}(\C)$ and $\Ch^{\geq n}(\C)$, respectively.

\subsection{$t$-Structures}\label{generalities_on_t}
An additive functor $H^0 \colon \mathcal D\to \C$ from a triangulated category $\mathcal D$ to an Abelian category $\C$ is said to be {\bf cohomological} if, given a triangle $X\to Y\to Z\to X[1]$ in $\D$, the sequence $H^0(X)\to H^0(Y)\to H^0(Z)$ is exact in $\C$. Letting $H^n(-) := H^0 ((-) [n])$ for all $n \in \Z$, one can  then attach to a triangle the following long exact sequence in $\C$:
\[
\cdots \longrightarrow H^{n-1}(Z)\longrightarrow H^{n}(X)\longrightarrow H^{n}(Y)\longrightarrow H^{n}(Z)\longrightarrow H^{n+1}(X)\longrightarrow \cdots.
\]

\begin{defi} \label{def.$t$-structure}
A {\bf $t$-structure} in $\mathcal D$ is a pair of classes $\tau=(\U, \W)$ satisfying the following axioms:
\begin{enumerate}[\rm ($t$-S.1)]
\item $\mathcal D(U,W[-1]) = 0$, for all $U \in \U$ and $W \in \W$;
\item $\W[-1] \subseteq \W$ (or, equivalently, $\U[1] \subseteq \U$);
\item for each $X \in \mathcal D$, there are $U_X \in \U$, $V_X \in \W[-1]$ and a so-calle {\bf $\tau$-truncation triangle}:
\begin{equation}\label{truncation_triangle_def_eq}
U_X \longrightarrow X \longrightarrow V_X\longrightarrow U_X [1].
\end{equation}
\end{enumerate}
\end{defi}

\begin{rem}\label{rem.t.def}
In the definition of a $t$-structure one can often find in the literature the extra assumption that $\U$ and $\W$ are closed under direct summands. In fact, this follows by the other axioms. Indeed, let $\tau=(\U,\W)$ be a $t$-structure (let us just assume the first half of ($t$-S.2), the part not in parenthesis), and suppose that $A\oplus B=U\in \U$. Then, $\W[-1]\subseteq\U^{\perp}\subseteq U^{\perp}\subseteq A^{\perp}$, so the map $g_A$ in the following truncation triangle of $A$ is trivial:
\[
\xymatrix{
U_A\ar[r]^-{f_A}&A\ar[r]^-{g_A}&V_A\ar[r]^-{h_A}&U_A[1],\qquad\text{with $U_A\in \U$ and $V_A\in \W[-1]$}.
}
\]
But then, $\U\ni U_A\cong A\oplus V_A[-1]$ and, using the above argument, $\W[-1]\subseteq (V_A[-1])^{\perp}$. On the other hand, $V_A[-1]\in \W[-2]\subseteq \W[-1]$, so that $\id_{V_A[-1]}=0$, showing that $V_A=0$, and so $A\cong U_A\in\U$. Hence, $\U$ is closed under summands, $\U={}^{\perp}\W[-1]$, so that the second half of ($t$-S.2) follows from the first one. The same argument now shows that also $\W$ is closed under summands. 
\end{rem}

By Remark~\ref{rem.t.def}, for a $t$-structure $\tau=(\U, \W)$, we have $\W = \U^{\perp}[1]$ and $\U = {}^{\perp}\W[-1] = {}^{\perp}(\U^{\perp})$. For this reason, we  write $t$-structures as $\tau=(\U, \V[1])$, meaning that $\V:=\U^{\perp}$. We  call $\U$ and $\U^{\perp}$ the {\bf aisle} and the {\bf co-aisle} of $\tau$, respectively. The objects $U_X$ and $V_X$  in \eqref{truncation_triangle_def_eq} are uniquely determined by $X$, up to a unique isomorphism, and define the so-called {\bf left} and {\bf right $\tau$-truncation functors}  
\[
\tau_{\U}\colon \mathcal D\longrightarrow \U\qquad\text{and}\qquad\tau^{\U^{\perp}}\colon \mathcal D\longrightarrow \U^{\perp},
\] 
which are right and left adjoints, respectively, to the corresponding inclusions.

The  subcategory $\mathcal H := \U \cap \V[1] = \U \cap \U^{\perp}[1]$ of $\D$ is called the {\bf heart} of $\tau$ and it is an Abelian category, where short exact sequences ``are'' the triangles of $\mathcal D$ with the three vertices in $\H$. Moreover, the assignments $X\mapsto \tau_\U \circ \tau^{\U^{\perp}[1]}(X)$ and $X\mapsto\tau^{\U^{\perp}[1]}\circ\tau_\U (X)$ define two naturally isomorphic cohomological functors (see \cite[Section~1.3]{BBD}) denoted indifferently by 
\[
H^0_\tau\colon \mathcal D\longrightarrow \H.
\] 
{Suppose that the triangulated category $\D$ has coproducts. Since $(\tau^{\U^{\perp}[1]})_{\restriction\U}=(H^{0}_{\tau})_{\restriction\U}\colon \U \to \H$ is a left adjoint (see \cite[Lemma~3.1]{PS1}), it preserve coproducts, so that $\coprod_{\H} H_i=\tau^{\U^{\perp}[1]}(\coprod_\D H_i)$, for each family $(H_i)_{i\in I}$ of objects of $\H$. Dually,  when $\D$ has products, we have $\prod_\H H_i=\tau_{\U}(\prod_\D H_i)$.}

\smallskip
Consider now a morphism $\phi\colon H_1\to H_2$ in $\H$ and complete it to a triangle in $\D$:
\[
K\longrightarrow H_1\overset{\phi}\longrightarrow H_2\longrightarrow K[1].
\]
Using the closure properties of $\U$ and $\V$, one verifies that $K\in \U[-1]\cap \V[1]$ (so  $K[1]\in \U[0]\cap \V[2]$). In particular, $\tau_\U(K),\, \tau^{\V[1]}(K[1])\in \H$, and the  compositions $\tau_\U(K)\to H_1$ and $H_2\to \tau^{\V[1]}(K[1])$ represent a kernel and a cokernel, respectively, of $\phi$ in $\H$ (see \cite[Section~1.3]{BBD}). }

\begin{ejem}\label{classical_t_str_ex}
Given a Grothendieck category  $\G$,  there is a canonical $t$-structure  in $\Der(\G)$ defined by $\tau=(\Der^{\leq0}(\G),\Der^{\geq0}(\G))$. The heart $\Der^{\leq0}(\G)\cap\Der^{\geq0}(\G)= \G[0]\cong \G$ of this $\tau$ is equivalent to $\G$.
\end{ejem}

A $t$-structure $\tau=(\U,\V[1])$ is said to be {\bf generated} by a class $\S\subseteq \D$ if $\V=\S^{\perp_{\leq0}}$. Furthermore, we say that $\tau$ is {\bf compactly generated} if it is generated by a set of compact objects.

%\medskip
%Suppose that the triangulated category $\D$ has coproducts and that $\tau=(\U, {\V[1]})$ is a $t$-structure in $\D$ with heart $\H$. Then, the {\bf coproduct} in $\H$ of a family $(H_i)_I\subseteq \H$ is given by $\tau^{\V[1]}(\coprod^{(\D)}_IH_i)$. {\bf Products} in $\H$ can be described dually. 

%\medskip

\subsection{Happel-Reiten-Smal\o\ $t$-structures and their hearts}
Let $\mathcal D$ be a triangulated category. Given two classes $\mathcal X,\,\mathcal Y\subseteq \mathcal D$, we define a new class:
\begin{itemize}
\item  $\mathcal X*\mathcal Y:=\{Z\in \D: \text{there is a triangle $X\to Z\to Y\to X[1]$ with $X\in\mathcal X$ and $Y\in\mathcal Y$}\}$. 
\end{itemize}
Similarly, given a Grothendieck category $\G$ and two subcategories $\B$ and $\C\subseteq \G$ we let $\B*\C\subseteq \G$ be the class of extensions of objects of $\C$ by objects in $\B$ (i.e., $(\B*\C)[0]=\B[0]*\C[0]$ in $\Der(\G)$).

\begin{defi}
Consider a $t$-structure $\tau=(\U,\V[1])$ in $\D$, with heart $\H:=\U\cap \V[1]$, and a torsion pair $\t=(\mathcal T,\mathcal F)$ in $\H$. We can define a new $t$-structure $\tau_\t=(\U_\t,\V_\t[1])$ on $\mathcal D$, called the {\bf Happel-Reiten-Saml{\o}} ({\bf HRS}) {\bf tilt of $\tau$ with respect to $\t$} (see \cite[Section~1.2]{HRS}), where 
\[
\U_\t:= \U[1]*\mathcal T,\quad \text{and}\qquad \V_\t:=\mathcal F*\V.
\]
\end{defi}
Almost by construction,  $(\mathcal F[1],\mathcal T[0])$ is a torsion pair in the Abelian category $\H_\t:=\U_\t\cap\V_\t[1]$.
We  mostly consider HRS tilts in the following situation: we take a Grothendieck category $\G$, we consider a torsion pair  $\t=(\T,\F)$ in $\G$ and we let $\tau=(\Der^{\leq0}(\G),\Der^{\geq0}(\G))$ be the natural $t$-structure in $\Der(\G)$ described in Example~\ref{classical_t_str_ex}. Then we let $\tau_\t=(\U_\t,\V_\t[1])$ be the tilt of $\tau$ with respect to $\t$, and we denote its heart by $\Ht$. In particular,
\[
\U_\t:= \Der^{\leq-1}(\G)*\mathcal T[0],\quad \qquad \V_\t:=\mathcal F[0]*\Der^{\geq1}(\G),\quad \text{and}\qquad \H_\t=\mathcal F[1]*\T[0].
\]
In this setting, $\Der(\G)$ has coproducts and products and, since $\Ht$ is the heart of a $t$-structure in $\Der(\G)$, it is bicomplete (see \cite[Proposition~3.2]{PS1}). We can also describe  some particular direct limits in $\Ht$:

\begin{lema}[{\rm\cite[Proposition~4.2]{PS1}}]\label{rem. stalks}
Let $\G$ be a Grothendieck category, $\t=(\T,\F)$ a torsion pair in $\G$, and $\Ht$ the heart of the associated $t$-structure. Then, the following statements hold true:
\begin{enumerate}[\rm (1)]
\item given a direct system $(T_\lambda)_{\Lambda}$ in $\T$, we have that ${\varinjlim}_\Lambda^{(\Ht)}(T_{\lambda}[0])\cong ({\varinjlim}_\Lambda^{(\G)}T_{\lambda})[0]$;
\item if $\t$ is of finite type and $(F_{\lambda})_{\Lambda}\subseteq\F$ is a direct system, then ${\varinjlim}_\Lambda^{(\Ht)}(F_{\lambda}[1])\cong ({\varinjlim}_\Lambda^{(\G)}F_{\lambda})[1]$.  
\end{enumerate}
\end{lema}

Similarly to Lemma~\ref{rem. stalks}, one can describe the kernels of maps between stalk complexes in $\H_\t$: 

\begin{lema} \label{lem.kernel of mixed morphism}
Let $\G$ be a Grothendieck category, $\t=(\T,\F)$ a torsion pair and $\Ht$ the heart of the associated $t$-structure. Consider a short exact sequence $0\to F\to K\to T\to 0$, with $F\in \F$ and $T\in \T$, and consider the associated triangle in $\Der(\G)$:
\[
F[0]\longrightarrow K[0]\longrightarrow T[0]\longrightarrow F[1].
\]
Then, $T[0],\, F[1]\in \Ht$ and the kernel of the map $T[0]\to F[1]$ in $\Ht$ is exactly, $t(K)[0]$.
%
%If $T\in\T$, $F\in\F$ and $\epsilon\colon T[0]\to F[1]$ is a morphism in $\Ht$ represented by the following short exact sequence in $\C$:
%\[
%0\rightarrow F\to K\stackrel{p}{\to} T\rightarrow 0
%\]  
%and we consider the composition $u\colon t(K)\stackrel{inclusion}{\hookrightarrow}K\stackrel{p}{\to} T$, then $u[0]\colon t(K)[0]\to T[0]$ is the kernel map of $\epsilon$ in $\Ht$.
\end{lema}
\begin{proof}
As we have mentioned at the end of Subsection~\ref{generalities_on_t}, $\Ker_{\Ht}(T[0]\to F[1])\cong \tau_{{\U_\t}}(K[0])$. Consider now the following approximation sequence in $\G$
%As we have mentioned at the end of Subsection~\ref{generalities_on_t}, $\Ker_{\Ht}(T[0]\to F[1])\cong \tau^{\U[1]*\T[0]}(K[0])$. Consider now the following approximation sequence in $\G$
\[
0\longrightarrow t(K)\longrightarrow K\longrightarrow (1:t)(K)\longrightarrow 0
\]
and the associated triangle $t(K)[0]\to K[0]\to (1:t)(K)[0]\to t(K)[1]$ in $\Der(\G)$. It is clear that $t(K)[0]\in \T[0]\subseteq {\Der^{\leq-1}(\G)*\T[0]=\U_\t}$ and $(1:t)(K)[0]\in \F[0]\subseteq {\mathcal F[0]*\Der^{\geq1}(\G)=\V_\t}$, so our triangle is the $\tau_\t$-truncation for $K[0]$, that is, $t(K)[0]\cong \tau_{{\U_\t}}(K[0])\cong \Ker_{\Ht}(T[0]\to F[1])$, as desired.
%We replace $F[1]$ by the complex $\xymatrix{E^{\bullet}:=\cdots \ar[r] & 0 \ar[r] & E(F) \ar[r]^{\pi} & \Omega^{-1}(F) \ar[r] & 0 \ar[r] & \cdots,}$ concentrated in degrees -1,0. Since $\epsilon \in \Hom_{\Ht}(T[0],F[1])=\Der(\G)(T[0],F[1])=\Hom_{\K(\G)}(T[0],E^{\bullet})$, we can represented $\epsilon$ by a chain map $f:T[0] \rightarrow E^{\bullet}$ up to homotopy. Thus, we get that $f$ can represented by a morphism $T \rightarrow \Omega^{-1}(F)$ in $\G$, which we will still denote by $f$. Now, note that the cone of $f$ in $\K(\G)$ and $\Der(\G)$ is given by the complex $\xymatrix{\cdots \ar[r] & 0 \ar[r]  & T\coprod E(F) \ar[r]^{(-f \ \pi)} & \Omega^{-1}(F) \ar[r] & 0 \ar[r] & \cdots }$ which is isomorphic to $K[1]$, where $K=\Ker(-f \ \pi)$. We consider the following diagram in $\Der(\G)$:
%$$\xymatrix{&& t(K)[1] \ar[d]& \\ T[0] \ar[r]^{\epsilon} & F[1] \ar[r] & K[1]=\mathrm{cone}(f) \ar[r]^{\hspace{1.2 cm}+} \ar[d] & \\ & & \frac{K}{t(K)}[1] \ar[d]^{+}\\ &&&}$$
%Since that $\frac{K}{t(K)}[1] \in \Ht$ and $t(K)[1]\in \Ht[1]$, from \cite{BBD} we obtain the following exact sequence in $\Ht$:
%$$\xymatrix{0 \ar[r] & t(K)[0] \ar[r] & T[0] \ar[r]^{\epsilon} & F[1] \ar[r] & \frac{K}{t(K)}[1] \ar[r] & 0}$$
%Hence, $K$ is the upper left corner of the pullback 
%$$\xymatrix{K \ar[r] \ar@{>>}[d]_{\pi'} \pushoutcorner& E(F) \ar@{>>}[d]^{\pi} \\ T \ar[r]_{f\hspace{0.5 cm}} & \Omega^{-1}(F)}$$
%For last, note that we get an induced short exact sequence $\xymatrix{0 \ar[r] & F \ar[r] & K \ar[r] & T \ar[r] & 0}$ which represents precisely $\epsilon \in \Ext^{1}_{\G}(T,F)$. 
\end{proof}

Let us conclude this subsection recalling the following important result that characterizes when the heart of an HRS tilt is a Grothendieck category:

\begin{prop}[{\rm \cite[Theorem~1.2]{PS2}}]\label{main_result_PS2}
Let $\G$ be a Grothendieck category,  $\t=(\T,\F)$  a torsion pair and  $\Ht$  the associated heart. Then, $\Ht$ is  Grothendieck  if and only if $\t$ is of finite type.
\end{prop}

\section{Torsion pairs of finite type, quasi-cotilting  and cosilting objects}\label{Section_finite_type}

In this section we clarify the relation between torsion pairs of finite type, quasi-cotilting torsion pairs and cosilting torsion pairs. 
The connections between these concepts are scattered in the literature and just partially known in the particular case where the ambient category is a category of modules (see Remark~\ref{rem.fintype-quasicotilting-cosilting}). 
After recalling the main definitions and some basic facts about pure injective objects (in Subsection~\ref{pure_subs}), cosilting objects (in Subsection~\ref{cosilt_obj_subs}), and quasi-cotilting objects (in Subsection~\ref{quasicotilt_obj_subs}) we are going to prove the following main result of this section:

\begin{teor}\label{main_thm_fintype_qcotilt_cosilt}
Let $\G$ be a Grothendieck category and  $\mathbf{t}=(\mathcal{T},\mathcal{F})$  a torsion pair in $\G$. The following assertions are equivalent:
 \begin{enumerate}[\rm (1)]
 \item $\mathbf{t}$ is of finite type;
 \item $\mathbf{t}$ is quasi-cotilting;
 \item $\mathbf{t}$ is the torsion pair associated with  a cosilting (pure-injective) object of $\G$.
 \end{enumerate}
\end{teor}

Furthermore, the following corollary is a byproduct of the methods used in the proof of Theorem \ref{main_thm_fintype_qcotilt_cosilt}:

\begin{cor}\label{main_cor_fintype_qcotilt_cosilt}
Let $\G$ be a Grothendieck category. An object $Q$ is quasi-cotilting if, and only if, there is a cosilting object $Q'$ such that $\Prod(Q)=\Prod(Q')$. Such an object is always pure-injective. In particular, all cosilting objects are pure-injective. 
\end{cor}

\begin{rem} \label{rem.fintype-quasicotilting-cosilting}
 When $\mathcal{G}=\mod R$ is the category of modules over a ring $R$, Theorem~\ref{main_thm_fintype_qcotilt_cosilt} follows from a combination of several results existing in the literature (see \cite[Section~3]{A}). Concretely, in independent work,  Breaz and Pop \cite{BP} and Zhang and Wei \cite{ZW2}, proved that  every cosilting module is pure-injective and, as a consequence, the implication {\rm``(3)$\Rightarrow$(1)''} holds in that case. Zhang and Wei {\rm [op.\ cit.]} also proved the equivalence {\rm ``(2)$\Leftrightarrow$(3)''}. They, and independently Breaz and Zemlicka \cite{BZ}, proved that the cosilting (=quasi-cotilting) torsion pairs in $\mod R$ are exactly those $\t=(\T,\F)$ for which $\F$ is covering. On the other hand, Bazzoni \cite{Bazz} had proved earlier that if $\mathcal{F}$ is definable (equivalently, if  $\mathbf{t}$ is of finite type) then $\mathcal{F}$ is covering. 
\end{rem}

\subsection{Pure-injective objects}\label{pure_subs}

Let $\A$ be an additive category (in what follows $\A$ will usually be --a  subcategory of-- either a Grothendieck or a  triangulated category). Given a set $I$ and an object $Y\in \A$, suppose that the coproduct  $Y^{(I)}$  exists in $\A$, and let $\iota_j\colon Y\to Y^{(I)}$ be the $j$-th inclusion, for all $j\in I$. The {\bf ($I$-)summation map}
\[
s_I\colon Y^{(I)}\longrightarrow Y
\]
is the unique morphism such that $s_I\circ\iota_j=\id_Y$, for each $j\in I$; its uniqueness and existence are ensured by the universal property of the coproduct.

\begin{defi}\label{pure_injective_def}
Let $\mathcal{A}$ be an additive category and  $Y\in\A$. We  say that $Y$ is {\bf pure-injective} when, for any given set $I$, the coproduct $Y^{(I)}$ and the product $Y^I$ exist in $\A$, and the summation map $s_I\colon Y^{(I)}\to Y$ extends through the canonical map $\mu_I\colon Y^{(I)}\to Y^I$:
\[
\xymatrix@C=30pt@R=15pt{
Y^{(I)}\ar[rr]^{s_I}\ar[d]_{\mu_I}&&Y,\\
Y^I\ar@/_+4pt/@{.>}[rru]_{\exists\, \hat s_I}
}
\]
that is, there exists a morphism $\hat{s}_I\colon Y^I\to Y$, such that $s_I=\hat{s}_I\circ\mu_I$. 
\end{defi}

The above definition of pure-injectivity is not the usual one, but it is equivalent to its more common counterpart in all the situations we are interested in, as the following remark points out:
 
\begin{rem}\label{rem_purity_subcat}
\begin{enumerate}[\rm (1)]
\item In the setting of Definition~\ref{pure_injective_def}, if $\mathcal{B}$ is a  subcategory of $\mathcal{A}$ which contains $Y$, $Y^{(I)}$, and $Y^I$, for every set $I$, then $Y$ is pure-injective in $\mathcal{B}$ if, and only if, it is pure-injective in $\mathcal{A}$. 
\item By the famous Jensen-Lenzing characterization of pure-injective modules \cite[Proposition~7.2]{Jensen-Lenzing}, later adapted to locally finitely presented Grothendieck categories \cite[Theorem~3.5.1]{CB}, and even to compactly generated triangulated categories \cite[Theorem~1.8]{Kr_tel_inv}, we know that, when $\mathcal{A}$ is either locally finitely presented Grothendieck or compactly generated triangulated, the notion of pure-injectivity introduced in Definition~\ref{pure_injective_def} coincides with the classical one. 
\end{enumerate}
\end{rem}
 
\subsection{Cosilting objects and cosilting torsion pairs}\label{cosilt_obj_subs}

Let us start introducing cosilting objects in the context of triangulated categories with products:
\begin{propdef}
{Let $\mathcal{D}$ be a triangulated category with products. The following assertions are equivalent for $E\in\D$:
\begin{enumerate}[\rm (1)]
\item the pair $({}^{\perp_{<0}}E,{}^{\perp_{>0}}E)$ is a $t$-structure in $\mathcal{D}$ and $E\in {}^{\perp_{>0}}E$;
\item  $E$ is a cogenerator of $\mathcal{D}$ that satisfies the following two conditions:
\begin{enumerate}[\rm ({2.}1)]
\item the pair $(\mathcal{U}_E,\mathcal{V}_E[1]):=({}^{\perp_{<0}}E, ({}^{\perp_{\leq 0}}E)^\perp )$ is a $t$-structure;
\item the functor $\mathcal{D}(-,E):\mathcal{D}\to\Ab$ vanishes on $\mathcal{V}_E$.
\end{enumerate}
\end{enumerate}
In this situation, the $t$-structures of $(1)$ and $(2.1)$ coincide and  $E$ is called a {\bf cosilting object} in $\mathcal{D}$.}
\end{propdef}
\begin{proof}  
The assertions (1) and (2) are the duals of the definitions of silting objects in triangulated categories with coproducts given in \cite{PV} and \cite{NSZ}, respectively, and it is known that they are equivalent (see \cite[Remark~4.3 and Theorem~4.1]{NSZ}). Finally, the two $t$-structures coincide as they share the aisle.
\end{proof}

In the derived category of a given Grothendieck category, there is a special class of well-studied cosilting objects, called $2$-term cosilting complexes:

\begin{defi} 
Let $\G$ be a Grothendieck category. An object $E$ of $\Der(\G)$ is a {\bf $2$-term cosilting complex} when it is a cosilting object of $\Der(\G)$ which is quasi-isomorphic to a complex of injectives 
\[
\cdots\longrightarrow 0\longrightarrow E^{-1}\longrightarrow E^0\longrightarrow 0\longrightarrow \cdots,
\] 
concentrated in degrees $-1$ and $0$. An object $Q$ of $\G$ is said to be a {\bf cosilting object} when \mbox{$Q\cong H^{-1}(E)$}, for some $2$-term cosilting complex $E$ of $\Der(\G)$.  
\end{defi}

It is important to underline that, although the terminology may suggest the contrary, a cosilting object $Q$ in a Grothendieck category $\G$ is not necessarily a cosilting object of $\Der(\G)$, when considered as a complex $Q[0]$ concentrated in degree zero. In the following lemma we give a more explicit characterization of the  $2$-term cosilting complexes. 

\begin{lema}\label{char_of_2-term_lem}
Let $\G$ be a Grothendieck category and consider a complex of injectives 
\[
E:\hspace*{0.5cm}\cdots\longrightarrow0\longrightarrow E^{-1}\stackrel{\sigma}{\longrightarrow}E^0\longrightarrow 0\longrightarrow\cdots
\]
concentrated in degrees $-1$ and $0$, viewed as an object of $\Der(\G)$. Then, the following are equivalent:
\begin{enumerate}[\rm (1)]
\item $E$ is a ($2$-term) cosilting complex;
\item $\mathcal{F}:=\{F\in\G:\Der(\G)(F[0],E)=0\}=\{F\in\G \text{ s.t. } \sigma_*\colon\G(F,E^{-1})\to\G(F,E^0)\text{ is epic}\}$ is a torsionfree class in $\G$ and $\mathcal{F}=\Cogen(Q)$, for  $Q:=H^{-1}(E)$. 
\end{enumerate}
Furthermore, under these conditions, $\Ext_\G^1(F,Q)=0$ for all $F\in \F$.
\end{lema}
\begin{proof} 
Part of the arguments in this proof are dual to those used in \cite{Hoshino-Kato-Miyachi}; we sketch them here, leaving some of the details to the reader. First of all, let $\mathcal{U}_E:={}^{\perp_{<0}}E$ and $\mathcal{V}_E:={}^{\perp_{\geq0}}E$. Then it is easy to verify that $\Der^{\leq -1}(\G)\subseteq \mathcal{U}_E$ and  $\Der^{> 0}(\G)\subseteq \mathcal{V}_E$. 

\smallskip\noindent
(1)$\Rightarrow$(2). By \cite[Lemma~1.1.2]{Poli}, $\mathcal{U}_E=\Der^{\leq -1}(\G)\ast\mathcal{T}'[0]$ and $\mathcal{V}_E=\mathcal{F}'[0]\ast\Der^{> 0}(\G)$, for a certain torsion pair $(\mathcal{T}',\mathcal{F}')$ in $\G$. 
%The fact that $(\mathcal{U}_E,\mathcal{V}_E[1])$ is a $t$-structure  in $\Der(\G)$ readily gives that $\mathbf{t}_E:=(\mathcal{T},\mathcal{F})$ is a torsion pair in $\G$.
By definition, an object $F\in\G$ is in $\mathcal{F}'$ if and only if $F[0]\in\mathcal{V}_E={}^{\perp_{\geq0}}E$, if and only if $\Der(\G)(F[0],E[i])=0$, for all $i\geq 0$. Since $E$ is a complex of injectives concentrated in degrees $-1$ and $0$, $\Der(\G)(M[0],E[i])=0$, for all $M\in\G$ and $i\geq1$. Hence, $F\in\mathcal{F}'$ if and only if $\Der(\G)(F[0],E)=0$, so $\F=\F'$ is a torsionfree class. 
\\
We next prove that $Q\in\mathcal{F}$, which implies that $\Cogen(Q)\subseteq\mathcal{F}$. Indeed, apply the functor $\Der(\G)(-,E[1])$ to the following triangle: 
\begin{equation}\label{tria_eq_cosilt_lem}
H^0(E)[-1]\longrightarrow  Q[1]\longrightarrow E\longrightarrow H^0(E)[0]
\end{equation}
and bear in mind that $\Der(\G)(H^0(E)[k],E[1])=0$, for $k=-1,\, 0$. We then get an isomorphism $\Der(\G)(Q[1],E[1])\cong\Der(\G)(E,E[1])$. But $E\in {}^{\perp_{>0}}E$, so that $0=\Der(\G)(E,E[1])\cong\Der(\G)(Q[0],E)$. Then, $Q\in\mathcal{F}$, as desired. 
Conversely, if $F\in\mathcal{F}$, apply $\Der(\G)(F[0],-)$ to the triangle \eqref{tria_eq_cosilt_lem}, to get the exact sequence:
\begin{equation}\label{ext_is_0_eq}
0=\Der(\G)(F[0],H^0(E)[-1])\longrightarrow\Der(\G)(F[0],Q[1])\longrightarrow\Der(\G)(F[0],E)=0,
\end{equation}
which implies that $\Ext_\G^1(F,Q)\cong\Der(\G)(F[0],Q[1])=0$. Knowing this, to prove that $\mathcal{F}\subseteq\Cogen(Q)$ it is enough to check that $\G(F,Q)\neq 0$, for all $F\in\mathcal{F}\setminus\{0\}$. 
%Indeed if this is true then, for each $F\in\mathcal{F}$, the canonical morphism $\lambda_F\colon F\to Q^{\G(F,Q)}$ has a kernel that satisfies that $\G(\Ker(\lambda_F),Q)=0$, which then implies that $\Ker(\lambda_F)=0$, and so $F\in\Cogen(Q)$.  
Take then $F\in\mathcal{F}$ and  assume that $0=\G(F,Q)\cong\Der(\G)(F[1], Q[1])$. Apply  the functor $\Der(\G)(F[1],-)$ to the triangle in \eqref{tria_eq_cosilt_lem} to obtain that $\Der(\G)(F[1],E )=0$, which implies that $\Der(\G)(F[i],E )=0$, for all $i\in\mathbb{Z}$ since, by definition of $\mathcal{F}$, we know that $\Der(\G)(F[0],E )=0$. It follows that $F=0$ since $E$ cogenerates $\Der(\G)$. 

\smallskip\noindent
(2)$\Rightarrow$(1). Consider the following triangle  in $\Der(\G)$:
\begin{equation}\label{components_of_E_tria_eq}
E^{-1}[0]\stackrel{\sigma[0]}{\longrightarrow}E^0[0]\longrightarrow E\stackrel{}{\longrightarrow}E^{-1}[1].
\end{equation}
Given $X\in\Der(\G)$, apply $\Der(\G)(X,-)$ to the triangle \eqref{components_of_E_tria_eq} to get the following long exact sequence:
\[
\scalebox{0.95}{$
\begin{split}
\Der(\G)(X,E^{-1}[-1])\stackrel{\sigma[-1]_*}{\longrightarrow}\Der(\G)(X,E^{0}[-1])\to\Der(\G)(X,E[-1])\to\Der(\G)(X,E^{-1}[0])\stackrel{\sigma[0]_*}{\longrightarrow}\Der(\G)(X,E^{0}[0])\to\\
\to\Der(\G)(X,E)\to\Der(\G)(X,E^{-1}[1])\stackrel{\sigma[1]_*}{\longrightarrow}\Der(\G)(X,E^{0}[1])\to\Der(\G)(X,E[1])\to\cdots
\end{split}$
}
\]
Hence, $X\in {}^{\perp_{<0}}E$ if and only if $\sigma[0]_ *$ is a monomorphism and $\sigma[j]_ *$ is an isomorphism for all $j<0$, while $X\in {}^{\perp_{>0}}E$ if and only if $\sigma[1]_ *$ is an epimorphism and $\sigma[j]_ *$ is an isomorphism, for all $j>1$. Now, given $j\in\mathbb{Z}$ there is a natural isomorphism $\Der(\G)(X,E^i[j])\cong\G(H^{-j}(X),E^i)$ (for $i=0,-1$) so $\sigma[j]_ *$ is an epimorphism if and only if $H^{-j}(X)\in\mathcal{F}$ (by definition of $\mathcal{F}$), and it is a monomorphism if and only if $0=\G(H^{-j}(X),\Ker(\sigma ))=\G(H^{-j}(X),Q)$. Since $\mathcal{F}=\Cogen(Q)$, we conclude that $\sigma[j]_*$ is a monomorphism if and only if $H^{-j}(X)\in\mathcal{T}$. As a conclusion, $X\in {}^{\perp_{<0}}E$ (resp.,  $X\in {}^{\perp_{>0}}E$) if, and only if, $H^0(X)\in\mathcal{T}$ and $H^{-j}(X)\in\mathcal{T}\cap\mathcal{F}=0$, for all $j<0$  (resp., $H^{-1}(X)\in\mathcal{F}$ and $H^{-j}(X)=0$, for all $j>0$). Therefore, $({}^{\perp_{<0}}E, {}^{\perp_{>0}}E)=(\mathcal{U}_\mathbf{t},\mathcal{V}_\mathbf{t}[1])$ is  the HRS $t$-structure associated with $\mathbf{t}$. 
Furthermore, $E\in {}^{\perp_{>0}}E$, since $ {}^{\perp_{>0}}E=\mathcal{V}_\mathbf{t}[1]=\mathcal{F}[1]\ast\Der^{\geq 0}(\G)$, while $E\in\Der^{\geq -1}(\G)$ and $H^{-1}(E)=Q\in\mathcal{F}$. 

Finally, the last statement follows by \eqref{ext_is_0_eq}.
\end{proof}

\begin{defi}
 Let $\G$ be a Grothendieck category. A {\bf cosilting torsion pair} in $\G$ is a torsion pair $\mathbf{t}=(\mathcal{T},\mathcal{F})$ such that $\mathcal{F}=\Cogen(Q)$, for some cosilting object $Q$ of $\G$. 
\end{defi}

\subsection{Quasi-cotilting objects and quasi-cotilting torsion pairs}\label{quasicotilt_obj_subs}
Let $\G$ be a Grothendieck category and $Q\in \G$. Recall that, $\underline{\Cogen}(Q)$ is the class of quotients of objects in $\Cogen(Q)$. Clearly, one always has the inclusions 
\[
\Prod(Q)\subseteq\Copres(Q)\subseteq \Cogen(Q)\subseteq \underline \Cogen(Q).
\] 
Recall also that ${}^{\perp_1}Q:=\{X\in \G:\Ext_\G^1(X,Q)=0\}$.
\begin{rem}\label{rem_copres=cogen}
Let $\G$ be a Grothendieck category and $Q\in \G$. Note that, 
\[
\text{ if \ \ $\Prod(Q)\subseteq {}^{\perp_1}Q\cap \underline \Cogen(Q)\subseteq \Cogen (Q)$ \ \ then \ \ $\Cogen(Q)=\Copres(Q)$.} 
\]
For $X\in \Cogen(Q)$, take the exact sequence $0\to X\to \Psi_Q(X)\to C\to 0$, with $\Psi_Q(X)\in \Prod(Q)$. Applying $\G(-,Q)$, we obtain an exact sequence:
\[
\G(\Psi_Q(X),Q)\overset{(*)}{\longrightarrow} \G(X,Q)\to \Ext^1_\G(C,Q)\to \Ext^1_\G(\Psi_Q(X),Q)=0
\]
where the last term is trivial since $\Prod(Q)\subseteq {}^{\perp_1}Q$. Furthermore, the map $(*)$ is surjective by construction, so $\Ext^1_\G(C,Q)=0$, showing that $C\in {}^{\perp_1}Q\cap \underline \Cogen(Q)\subseteq \Cogen (Q)$. Hence, $X\in \Copres(Q)$.
\end{rem}

Finally, recall that a subcategory $\X$ of $\G$ is called {\bf generating} when each object of $\G$ is the epimorphic image of an object in $\X$. {\bf Cogenerating subcategories} are defined dually.

\begin{defi}
An object $Q$ in a Grothendieck category $\G$ is said to be 
%Slightly modifying \cite[Definition~2.3 and 2.6]{C}, an object $Q$ in a Grothendieck category $\G$ is said to be 
\begin{itemize}
\item ($1$-){\bf cotilting} when $\Cogen(Q) = {}^{\perp_1}Q$ and this is a generating subcategory of $\G$;
\item {\bf quasi-cotilting} when $\Cogen(Q)={}^{\perp_1}Q\cap\underline{\Cogen}(Q)$.  
\end{itemize}
In both cases, $\Cogen (Q)=\Copres (Q)$ (see Remark~\ref{rem_copres=cogen}) and this is a torsionfree class in $\G$. A torsion pair of the form $({}^\perp Q,\Cogen(Q))$, for $Q$  (quasi-)cotilting, is called a {\bf \mbox{(quasi-)}cotilting torsion pair}.
\end{defi}
{In \cite[Definition~2.6]{C} the Author only required the equality $\Cogen(Q)={}^{\perp_1}Q$ for an object $Q$ to be (1-)cotilting but, in the setting of that paper, $\G$ is assumed to have enough projectives, in which case it is automatic  that  ${}^{\perp_1}Q$ is  generating.}

\subsection{Proofs of the main results} 
In this  subsection we prove the results announced at the beginning of this section and we deduce some consequences that will be useful later on. We start with the following technical lemma, but let us first recall that, given $M\in\G$ and $n\in\Z$, the {\bf $n$-th disk complex of $M$} is the following object of $\Ch(\G)$:
\[
D^n(M):\qquad\xymatrix{\cdots \ar[r]& 0\ar[r]& M\ar[r]^-{\id_{M}}&M\ar[r]& 0\ar[r]& \cdots
}
\] 
concentrated in degrees $n$ and $n+1$. Any such complex is contractible, and hence it is trivial when considered as  an object in $\K(\G)$.

\begin{lema}\label{lem_fintype_qcotilt_cosilt}
Let $\G$ be a Grothendieck category and $Y\in\Ch^+(\Inj\G)$ a bounded-below complex of injectives. If, for a given $k\in\mathbb{Z}$, $\Der(\G)(-,Y[k])$ vanishes on $\G[0]$, then we have an isomorphism 
\[
Y\cong Y^-\oplus Y^+\qquad\text{ in $\K(\Inj\mathcal{G})$,}
\] 
where $Y^{-}\in\Ch^{<k}(\Inj\G)$, $Y^{+}\in\Ch^{>k}(\Inj\G)$, and such that $(Y^-)^{k-1}$ is a {direct} summand of $Y^{k-1}$. In particular,  if $\Der(\G)(-,Y[j])$ vanishes on $\G[0]$, for all $j\geq k$, then $Y\in\mathcal{K}^{<k}(\Inj\G)$.
\end{lema}
\begin{proof}
In what follows, for all $n\in \Z$, $d^{n}\colon Y^n\to Y^{n+1}$ denotes the $n$-th differential of $Y$, while
\[
Z^n:=Z^n(Y)=\Ker(d^n)\qquad\text{and}\qquad B^n:=B^n(Y)=\Im(d^{n-1})
\]
are the objects of $n$-cocycles and coboundaries of $Y$, respectively, so that $H^n(Y)=Z^n/B^{n}$. We know by \cite[Lemma~5.9]{NSZ} that $H^k(Y)=0$. The inclusion $Z^k\to Y^k$ induces a map $Z^k[0]\to Y[k]$ in $\Ch(\G)$, which is null-homotopic since $\K(\G)(Z^k[0],Y[k])  \cong \Der(\G)(Z^k[0],Y[k])=0$ (in fact, being $Y$ a bounded below complex of injectives, it is homotopically injective). It follows that the induced epimorphism $Y^{k-1}\to B^k=Z^k$ is a retraction. In particular, $Z^k$ is injective in $\G$ and we have decompositions $Y^{k-1}\cong Z^{k-1}\oplus Z^k$ and $Y^k\cong B^{k+1}\oplus Z^k$. Hence, also $B^{k+1}$ is injective, giving us a decomposition $Y^{k+1}\cong B^{k+1}\oplus \overline{Y^{k+1}}$. The differentials $d^{k-1}$ and $d^k$ can then be written in matricial form as follows:
\[
d^{k-1}=\left[\begin{smallarray}{cc} 0 & 0\\ 0 & \id_{Z^k}\end{smallarray}\right]
\colon Z^{k-1}\oplus Z^k\longrightarrow B^{k+1}\oplus Z^k
\quad \text{and}\quad
d^k=\left[\begin{smallarray}{cc}\id_{B^{k+1}} & 0\\ 0 & 0 \end{smallarray}\right]
\colon B^{k+1}\oplus Z^k\longrightarrow B^{k+1}\oplus\overline{Y^{k+1}}.
%d^k=\begin{pmatrix}\bar{d} & 0\end{pmatrix}
%\colon C^k\oplus Z^k\longrightarrow Y^{k+1},
\] 
%where $\bar{d}$ is a monomorphism. 
%Therefore, we obtain a decomposition $Y^{k+1}\cong\tilde{Y}^k\oplus\tilde{Y}^{k+1}$ that allows us to depict $d^k$ as 
%\[
%d^k=\begin{pmatrix} \id_{\tilde{Y}^k} & 0\\ 0 & 0 \end{pmatrix}
%\colon \tilde{Y}^k\oplus Z^k\longrightarrow\tilde{Y}^k\oplus\tilde{Y}^{k+1}.
%\]
This shows that
$Y\cong Y^-\oplus D^{k-1}(Z^k)\oplus D^k(B^{k+1})\oplus Y^+$  in $\Ch(\Inj\G)$,
with $Y^-\in\Ch^{<k}(\Inj\G)$ and  $Y^+\in\Ch^{>k}(\Inj\G)$. The first part of the statement is then proved, as disk complexes are contractible. 

{As a consequence, given $Z\in \K^{\geq k}(\Inj\G)$ such that $\Der(\G)(-,Z[k])$ vanishes on $\G[0]$, there is a some $Z'\in\Ch^{\geq k}(\Inj \G)$ such that $Z\cong Z'$ in $\K(\Inj \G)$, and $(Z')^{-}\in \K^{\geq k}(\Inj \G)\cap \K^{<k}(\Inj \G)=0$. Hence, $Z\in \K^{>k}(\Inj \G)$.}

{On the other hand, if $Y\in \Ch^{+}(\Inj\G)$ and $\Der({\G})(-,Y[j])$ vanishes on $\G[0]$, for all $j\geq k$, we apply inductively the conclusion of the last paragraph to $Z=Y^{+}$ to get that $Z\in \K^{>j}(\Inj\G)$, for all $j\geq k$, and so $Z=0$ in $\K(\Inj \G)$. Hence, $Y$ is isomorphic to $Y^{-}$ in $\K(\Inj \G)$, so that $Y\in \K^{< k}(\Inj \G)$.}
\end{proof}

We are now ready for the proof of the main result of this section:
 
\begin{proof}[Proof of Theorem~\ref{main_thm_fintype_qcotilt_cosilt}]
(1)$\Rightarrow$(2). By Proposition~\ref{main_result_PS2}, we know that assertion (1) holds if, and only if, the heart $\mathcal{H}_\mathbf{t}$ of the associated HRS $t$-structure in $\Der(\G)$ is Grothendieck. By \cite[Proposition~3.8]{PS4}, if $Y$ is an injective cogenerator of $\mathcal{H}_\mathbf{t}$, then $Q:=H^{-1}(Y)$ is quasi-cotilting in $\G$ and $\mathcal{F}=\Cogen(Q)=\Copres(Q)$. 
 
 \smallskip\noindent
 (2)$\Rightarrow$(1). Fix a quasi-cotilting object $Q$ such that $\mathcal{F}=\Cogen(Q)$ and let $\underline{\mathcal{F}}:=\underline{\Cogen}(Q)$. In the proof of \cite[Theorem~4.18]{PS4} it is shown that $\underline{\mathcal{F}}$ is coreflective in $\G$, that $\underline \F$ is a Grothendieck category,  and that $Q$ is a (1-)cotilting object in $\underline{\mathcal{F}}$. Although our hypotheses are not exactly the same, the arguments of \cite{PS4} also apply here. Then, \cite[Theorem~3.9]{Coupek-Stovicek} says that $\mathcal{F}$ is closed under direct limits in $\underline{\mathcal{F}}$, which are computed as in $\G$, therefore $\mathbf{t}$ is a torsion pair of finite type in $\G$.

\smallskip\noindent
(1,2)$\Rightarrow$(3). The associated HRS $t$-structure $\tau_{\mathbf{t}}=(\mathcal{U}_\mathbf{t},\mathcal{V}_\mathbf{t}[1])$ in $\Der(\G)$ is smashing, (left and right) non-degenerate and its heart is a Grothendieck category. Moreover, due to \cite[Proposition~5.1]{AJSo} and \cite[Theorem~7.2]{Porta}, we know that $\Der(\G)$ is a well-generated triangulated category, which implies that it satisfies Brown's Representability Theorem (see \cite[Section~9]{N}). It follows from \cite{Saorin-Stovicek} that there is a cosilting pure-injective object $E\in\Der(\G)$ such that $\tau_\mathbf{t}=({}^{\perp_{<0}}E,{}^{\perp_{>0}}E)$ is the associated $t$-structure. Let us show that $E$ is a $2$-term cosilting complex. Indeed, we have that $E\in\mathcal{V}_\mathbf{t}[1]\subseteq\Der^{\geq -1}(\G)$ and so $E$ is quasi-isomorphic to a complex in $\Ch^{\geq-1}(\Inj\G)$, with $H^{-1}(E)\in\mathcal{F}$. On the other hand, $\G[-i]\subseteq\Der^{\geq 0}(\G)\subseteq\mathcal{V}_\mathbf{t}[1]= {}^{\perp_{>0}}E$, so $0=\Der(\G)(\G[-i],E[j])\cong\Der(\G)({\G[0]},E[i+j])$, for all integers $i\geq 0$ and $j>0$. By Lemma~\ref{lem_fintype_qcotilt_cosilt}, $E$ is isomorphic in $\K(\Inj\G)$ to a complex of injectives concentrated in degrees $-1$ and $0$, so $E$ is a $2$-term cosilting complex. For simplicity, in the rest of the proof we assume that $E=(\cdots \to  0\to E^{-1}\stackrel{\sigma}{\to}E^0\to 0 \to \cdots)$, with $E^i$ injective for $i=-1,\ 0$.
%
%\[
%\cdots\to 0\rightarrow E^{-1}\rightarrow\tilde{E}^0\rightarrow 0\to \cdots,
%\] 
%where $\tilde{E}^0$ is a direct summand of $E^0$. Then, $E$ is a $2$-term cosilting complex, as desired. 
%
%
%
%Hence, the proof reduces to check that $E$ is a $2$-term cosilting complex. 
Hence, $Q:=H^{-1}(E)$ is a cosilting object of $\G$ such that $\mathcal{F}=\Cogen(Q)$. Moreover,  {since}  $E$ is pure-injective in $\Der(\G)$, there is commutative diagram as follows (see Definition~\ref{pure_injective_def}):
\begin{equation}\label{pure_inj_dia_proof}
\xymatrix@C=30pt@R=10pt{
E^{(I)}\ar[rr]^{s^E_I}\ar[d]_{\mu^E_I}&&E.\\
E^I\ar@/_+4pt/@{.>}[rru]_{\exists\, \hat s^E_I}
}
\end{equation}
% if we fix the canonical map $\mu_I^E\colon E^{(I)}\to E^I$ from the coproduct to the product in $\Der(\G)$, and we fix a morphism $\hat{s}_I^E\colon E^I\to E$ such that $s_I^E=\hat{s}_I^E\circ\mu_I^E\colon E^{(I)}\to E$ is the summation map, then 
{As products are left exact in $\G$, we have that $H^{-1}(E^I)=\Ker(\sigma^{I})\cong \Ker(\sigma)^{I}\cong H^{-1}(E)^I$. Similarly,} being coproducts exact in $\G$, also $H^{-1}(E^{(I)})\cong H^{-1}(E)^{(I)}$. It follows that $H^{-1}(\mu_I^E)$ gets identified with the canonical map $Q^{(I)}\to Q^I$ and $H^{-1}(s_I^E)$ identifies with the summation map $Q^{(I)}\to Q$. The decomposition $H^{-1}(s_I^E)=H^{-1}(\hat{s}_I^E)\circ H^{-1}(\mu_I^E)$ tells us that $Q$ is pure-injective in $\G$.

 \smallskip\noindent
(3)$\Rightarrow$(2). Let $E$ be a $2$-term cosilting complex 
\[
E:\hspace*{0.5cm}\cdots \longrightarrow 0\longrightarrow E^{-1}\stackrel{\sigma}{\longrightarrow}E^0\longrightarrow 0\longrightarrow\cdots
\] 
such that $Q=H^{-1}(E)$ and $\Cogen(Q)=\mathcal{F}=\{F\in\G:\Der(\G)(F[0],E)=0\}$. By Lemma~\ref{char_of_2-term_lem}, $\mathcal{F}\subseteq {}^{\perp_1}Q$ and so $\mathcal{F}\subseteq  {}^{\perp_1}Q\cap\underline{\Cogen}(Q)=  {}^{\perp_1}Q\cap\underline{\mathcal{F}}$. We just need to prove the converse inclusion: take $M\in  {}^{\perp_1}Q\cap\underline{\mathcal{F}}$  and fix a short exact sequence $0\rightarrow F'\stackrel{u}{\longrightarrow}F\to M\rightarrow 0$, with $F\in \F$. Applying $\Der(\G)(-,E)$ to the corresponding triangle in $\Der(\G)$, we get an exact sequence
\[
\Der(\G)(F[1],E)\stackrel{u[1]^*}{\longrightarrow}\Der(\G)(F'[1],E)\longrightarrow\Der(\G)(M[0],E)\longrightarrow\Der(\G)(F[0],E)=0.
\]
The natural isomorphism $\Der(\G)(-[1],E)_{\restriction \G}\cong\G(-,Q)$ of functors $\G^{\op}\to\Ab$ gives us an exact sequence 
\[
\G(F,Q)\stackrel{u^*}{\longrightarrow}\G(F',Q)\longrightarrow\Der(\G)(M[0],E)\rightarrow 0,
\] 
where $u^*$ is surjective since $\Ext_\G^1(M,Q)=0$. It follows that $\Der(\G)(M[0],E)=0$, and so $M\in\mathcal{F}$. 
\end{proof}

Now the proof of Corollary~\ref{main_cor_fintype_qcotilt_cosilt} consists in a closer analysis of the methods used in the above proof. Before proceeding, let us just recall that, given a Grothendieck category $\G$, two objects $X,\, Y\in \G$ are said to be {\bf $\Prod$-equivalent} if and only if $\Prod(X)=\Prod(Y)$. Clearly, if $X$ and $Y$ are $\Prod$-equivalent, then the former is quasi-cotilting if and only if so is the latter.

\begin{proof}[Proof of Corollary~\ref{main_cor_fintype_qcotilt_cosilt}]
 In the proof of the implication ``(3)$\Rightarrow$(2)'' in Theorem~\ref{main_thm_fintype_qcotilt_cosilt} we have verified the ``if'' part of the assertion, since any object which is $\Prod$-equivalent to a quasi-cotilting object is also quasi-cotilting. On the other hand, the proof of the implication ``(1,2)$\Rightarrow$(3)'' shows that, if $Q$ is a quasi-cotilting object and $\mathbf{t}=(\mathcal{T},\mathcal{F})$ is the associated torsion pair, then it is also the cosilting torsion pair associated with a pure-injective cosilting object $Q'$. Hence, 
 \[
 {}^{\perp_{ {1}} }Q'\cap\underline{\mathcal{F}}=\Cogen(Q')=\Copres(Q')=\mathcal{F}=\Cogen(Q)=\Copres(Q)={}^{\perp_1}Q\cap\underline{\mathcal{F}}.
 \] 
 Thus, $Q\in \Prod(Q)=\Prod(Q')$.  Moreover, by \cite[Proposition~3.4]{Coupek-Stovicek}, the class of pure-injective objects is closed under products in $\G$, and it is clearly closed under  direct summands, so $Q$ is pure-injective. 
\end{proof}

We end the section with two results that will be useful later on.

\begin{prop} \label{prop.injcogenerator-from-cosilting}
Let $\mathcal{G}$ be a Grothendieck category, take a $2$-term cosilting complex $E$,
\[
E:\hspace*{0.5cm}\cdots\longrightarrow0\longrightarrow E^{-1}\stackrel{\sigma}{\longrightarrow}E^0\longrightarrow 0\longrightarrow\cdots
\]
let $\mathbf{t}=(\mathcal{T},\mathcal{F})$ be the associated (finite type) torsion pair in $\mathcal{G}$, and consider the canonical epimorphism $p\colon E\to (1:t)(H^0(E))[0]$ in $\Ch(\mathcal{G})$. Then, $W:=\Ker(p)$ is an injective cogenerator of the heart $\mathcal{H}_\mathbf{t}$ and, furthermore, $H^{-1}(W)\cong H^{-1}(E)$ and $H^0(W)\cong t(H^0(E))$. 
\end{prop}
\begin{proof}
Consider the triangle in $\Der (\G)$ associated with the following exact sequence of complexes
\[
0\longrightarrow W\longrightarrow E\stackrel{p}{\longrightarrow}(1:t)(H^0(E))[0]\longrightarrow 0.
\] 
Analyzing the associated long exact sequence of cohomologies, one sees that $H^{-1}(W)\cong H^{-1}(E)$, \mbox{$H^0(W)\cong t(H^0(E))$} and $H^k(W)=0$, for all integers $k\neq -1,0$. It then follows that $W\in\mathcal{H}_\mathbf{t}$.
\\
For the injectivity of $W$ in $\mathcal{H}_\mathbf{t}$ we just need to check that $0=\Ext_{\mathcal{H}_\mathbf{t}}^1(F[1],W)\cong\Der (\G)(F[0],W)$ and $0=\Ext_{\mathcal{H}_\mathbf{t}}^1(T[0],W)\cong\Der(\G) (T[0],W[1])$, for all $F\in\F$ and all $T\in\T$. This follows immediately from the following induced exact sequences of Abelian groups (see Lemma~\ref{char_of_2-term_lem}):
\[
\begin{aligned}
0=\Der(\G)(F[0],(1:t)(H^0(E)))[-1])\longrightarrow &\Der(\G)(F[0],W)\longrightarrow\Der(\G)(F[0],E)=0 \\ 0=\Der(\G)(T[0],(1:t)(H^0(E)){[0])}\longrightarrow &\Der(\G)(T[0],W[1])\longrightarrow\Der(\G)(T[0],E[1])=0.
\end{aligned}
\]
{The last equality to zero is due to the fact that $\Der(\G)(T[0],E[1])\cong \K(\G)(T[0],E[1])=0$, since {$E[1]$} is a complex of injectives concentrated in degrees {-2 and -1.}} It remains to prove that, if  $M\in\mathcal{H}_\mathbf{t}$ and $\mathcal{H}_\mathbf{t}(M,E)=0$, then $M=0$. Consider the isomorphism 
\[
\mathcal{H}_\mathbf{t}(M,W)=\Der(\G)(M,W)\stackrel{\cong}{\longrightarrow}\Der(\G)(M,E),
\] 
which is true since  $(1:t)(H^0(E))[k]\in\V_\mathbf{t}=\U_\mathbf{t}^\perp$, for $k=-1,0$. Then, if $\Der(\G)(M,W)=0$, we get that $M\in {}^\perp E$, and so $\Der(\G)(M,E[k])=0$, for all $k\in\mathbb{Z}$, because $M\in\mathcal{H}_\mathbf{t}=\U_\mathbf{t}\cap\V_\mathbf{t}[1]={}^{\perp_{<0}}E\cap {}^{\perp_{>0}}E$. It follows that $M=0$, since $E$ cogenerates $\Der(\G)$. 
\end{proof}

For the last result of this section, it is  convenient  to view $\Der(\G)$ as the base of the  strong and stable derivator $\mathbb{D}_\G\colon \mathbf{Cat}^{\op}\to \mathbf{CAT}$ such that $\mathbb{D}_\G(I)=\Der (\G^I)\cong \Ho(\Ch(\mathcal{G}^I))$ for any small category $I$. The reader is referred to \cite{SSV} for all the needed terminology.

\begin{cor}\label{coro_derivators}
In the same setting of Proposition~\ref{prop.injcogenerator-from-cosilting}, take a direct system $(W_\lambda )_{\Lambda}$ in $\H_{\mathbf{t}}$, of objects in $\Prod_{\H_\mathbf{t}}(W)$. Then, there is a direct system of short exact sequences in $\Ch(\G)$,
\[
(0\longrightarrow E^{I_\lambda}\longrightarrow Y_\lambda\longrightarrow W'_\lambda [1]\longrightarrow 0)_{\Lambda},
\] 
with each $Y_\lambda\in\V_\mathbf{t}[1]= {}^{\perp_{>0}}E$ and such that  {the direct system}  $(W'_\lambda)_{\Lambda}$ in $\Ch(\G)$  is mapped, up to isomorphism,  onto $(W_\lambda )_{\Lambda}$ by the localization functor $q\colon \Ch(\mathcal{G})\to\Der(\mathcal{G})$.
\end{cor}
\begin{proof}
By \cite[Theorem~A]{SSV}, $(W_\lambda )_{\Lambda}$ can be lifted to a coherent diagram in $\mathbb{D}_\mathcal{G}(\Lambda )\cong\Ho(\Ch(\mathcal{G}^\Lambda))$. This means that we have some direct system $(W'_\lambda)_{\Lambda}$ in $\Ch(\G)$ that is mapped, up to isomorphism,  onto $(W_\lambda )_{\Lambda}$ by the localization functor $q\colon \Ch(\mathcal{G})\to\Der(\mathcal{G})$. 
 We now put $I_\lambda: =\Ch(\G)(W'_\lambda ,E)$, for each $\lambda\in\Lambda$, and consider the canonical morphism $u_\lambda\colon W'_\lambda\to E^{I_\lambda}$. Since $E$ is homotopically injective, the morphism $u_\lambda$ is a $\Prod (E)$-preenvelope both in $\Ch(\G)$ and in $\Der (\G)$. Taking the classical cone  of a chain map in $\Ch(\G)$, we get a direct system $(0\rightarrow E^{I_\lambda}\to\mathrm{cone}(u_\lambda )\to W'_\lambda [1]\rightarrow 0)_{\Lambda}$ of short exact sequences in $\Ch(\G)$. We claim that $Y_\lambda :=\mathrm{cone}(u_\lambda)\in\V_\mathbf{t}[1]= {}^{\perp_{>0}}E$. We clearly have that $\Der(\G)(-,E[k])$ vanishes on $E^{I_\lambda}$ and $W'_\lambda [1]$, for $k>1$. So we only need to prove that $\Der(\G)(Y_\lambda ,E[1])=0$. This follows from the fact that, $u_\lambda$ being a $\Prod(E)$-preenvelope,  we have the following exact sequence:
 \[
 \Der(\G)(E^{I_\lambda},E)\stackrel{u_\lambda^*}{\longrightarrow}\Der(\G)(W'_\lambda,E)\stackrel{0}{\longrightarrow}\Der(\G)(Y_\lambda[-1],E)\longrightarrow\Der(\G)(E^{I_\lambda}[-1],E)=0.\qedhere
 \] 
\end{proof}

\section{Finitely presented objects in the heart}\label{Section_fp}

In this section we start with a Grothendieck category $\G$ and a torsion pair of finite type $\t=(\T,\F)$  in $\G$, and we give a general characterization of the finitely presented objects in the heart $\Ht$ of the HRS $t$-structure associated with $\t$. We then deduce explicit characterizations, under fairly general hypotheses, for the finitely presented stalk complexes in the heart, deducing a complete description $\fp(\Ht)$ when $\G$ is good enough. In the last two subsections we study $\fp_2(\Ht)$ (whose understanding is essential in determining whether $\Ht$ is locally coherent) and the category $\underline \F$, which helps us to properly interpret the results of the first half of the section {(see Remark~\ref{interpretation_of_results})}.

\subsection{General characterization of finitely presented objects}\label{subs5.1}

 Recall that Proposition~\ref{prop.fpobjects-wrt-torsionpairs} gives a criterion for the finite presentability of an object, in terms of a given torsion pair of finite type. The following corollary specializes this criterion to the tilted torsion pair $(\F[1],\T[0])$ in the heart $\Ht$.

\begin{cor} \label{cor.fpobjects-in-heart}
Let $\G$ be a Grothendieck category and  $\t=(\T,\F)$  a torsion pair of finite type in $\G$. Then, $\H_\t$ is a Grothendieck category and $\bar{\mathbf{t}}:=(\mathcal{F}[1],\mathcal{T}[0])$ is a torsion pair of finite type in $\H_\t$. Furthermore, the following assertions are equivalent for an object $M\in\H_\t$:
\begin{enumerate}[\rm (1)]
\item $M$ is finitely presented in $\H_\t$;
\item $M$ satisfies the following conditions:
\begin{enumerate}[\rm ({2.}1)]
\item $\Der(\G )(M,(-)[1])_{\restriction \mathcal{F}}:\mathcal{F}\to\Ab$ preserves direct limits;
\item $H^0(M)\in\fp(\G)$;
\item ${\varinjlim}_I \Der(\G )(M,F_i[2])\to\Der(\G )(M,({\varinjlim}_I F_i)[2])$ is a monomorphism, for all $(F_i)_{I}\subseteq \mathcal{F}$ directed.
\end{enumerate}
\end{enumerate}
\end{cor}
\begin{proof}
We have already mentioned that $\H_\t$ is Grothendieck and that $\bar{\mathbf{t}}=(\mathcal{F}[1],\mathcal{T}[0])$ is a torsion pair, which is of finite type by \cite[Proposition~4.2]{PS1}. In this setting, conditions (2.1), (2.2) and (2.3) are just a reformulation of the corresponding conditions in Proposition~\ref{prop.fpobjects-wrt-torsionpairs}, via the  natural isomorphisms: 
\[
\mathcal{T}(H^0(M),-)\cong\Ht(M,(-)[0])_{\restriction\mathcal{T}}\colon\mathcal{T}\longrightarrow\Ab,
\qquad
\H_\t(M,(-)[1])_{\restriction \mathcal{F}}=\Der(\G )(M,(-)[1])_{\restriction\mathcal{F}}\colon\mathcal{F}\longrightarrow\Ab
\] 
and $\Ext_{\H_\t}^1(M,(-)[1])_{\restriction\mathcal{F}}\cong\Der(\G )(M,(-)[2])_{\restriction\mathcal{F}}\colon\mathcal{F}\to\Ab$ (see \cite[R\'emarque~3.1.17]{BBD}), and Lemma~\ref{fp_torsion_lema}.
\end{proof}

\subsection{Finitely presented objects in $\F[1]$}\label{subs5.2}
In this subsection we apply Corollary~\ref{cor.fpobjects-in-heart} to characterize the stalk complexes concentrated in degree $-1$ that are finitely presented in $\H_\t$.

\begin{defi}\label{def_f0}
Let $\G$  be a Grothendieck category and $\t=(\T,\F)$  a torsion pair in $\G$. Define $\mathcal{F}_0$ as the class of all $F\in\F$ such that $\G(F,-)\colon \mathcal{\G}\to\Ab$ preserves direct limits of objects in $\mathcal{F}$. 
\end{defi}

When the torsion pair $\t$ is of finite type, the class $\F_0$ is particularly useful to describe the finitely presented objects in the heart, in fact, as shown in the following corollary, $\fp(\H_\t)\cap \F[1]=\F_0[1]$.

\begin{cor} \label{cor.F0}
Let $\G$ be a Grothendieck category and  $\t=(\T,\F)$  a torsion pair of finite type in $\G$.
For an object $F\in\mathcal{F}$, consider {the} following conditions:
\begin{enumerate}[\rm (1)]
\item $F$ is isomorphic to a {direct} summand of $(1:t)(X)$, for some $X\in\fp(\G)$;
\item $F[1]$ is finitely presented in $\mathcal{H}_\mathbf{t}$;
\item $F\in\mathcal{F}_0$.
\end{enumerate}
Then, the implications {\rm``(1)$\Rightarrow$(2)$\Leftrightarrow$(3)''} hold true. Furthermore, if  $\G$ is locally finitely presented, then $\F=\varinjlim\F_0$ and the assertions are all equivalent.
\end{cor}
\begin{proof}
By Corollary~\ref{cor.fpobjects-in-heart}, $\bar\t=(\mathcal{F}[1],\mathcal{T}[0])$ is a torsion pair of finite type in $\H_\t$, so the equivalence ``(2)$\Leftrightarrow$(3)'' follows by Proposition~\ref{prop.fpobjects-wrt-torsionpairs}, Lemma~\ref{fp_torsion_lema}, and Lemma~\ref{rem. stalks}.

\smallskip\noindent
(1)$\Rightarrow$({3}). Given $X\in \G$, there is a natural isomorphism  $\G (X,-)_{\restriction\mathcal{F}}\cong\mathcal{F}((1:t)(X),-)\colon \mathcal{F}\to\Ab$, where the first of these  functors preserves direct limits when $X\in\fp(\G)$. 

\smallskip\noindent
(2,\,3)$\Rightarrow$(1), {\em if $\G$ is locally finitely presented}. 
%*************By the natural isomorphism ${\G}(F,H^{-1}(-))\cong{\H_\t}(F[1],-)$ of functors $\H_\t\to\Ab$, the object $F[1]$ belongs to $\fp(\H_\t)$ if, and only if, the functor  ${\G}(F,H^{-1}(-))\colon\H_\t\to\Ab$ preserves direct limits. Since $\t$ is of finite type, $H^{-1}\colon\H_\t\to\G$ preserves direct limits (see \cite[Theorem~4.8]{PS1}), so that $F[1]\in\fp(\H_\t)$ if, and only if, ${\G}(F,-)\colon\G\to\Ab$ preserves direct limits of objects in $\F$. 
%
%Given $M\in\fp(\G)$ there is a natural isomorphism ${\G}(M,-)_{\restriction \F}\cong{\F}((1:t)(M),-)$ of functors $\F\to\Ab$. From this and the previous paragraph it follows that $(1:t)(M)[1]\in\fp(\H_\t)$. As a consequence,  if $F$ is a direct summand of $(1:t)(M)$, then $F[1]\in\fp(\H_\t)$.  
%
Let $F\in \F$ and consider $F[1]\in\H_\t$; express $F$ as a direct limit $F={\varinjlim}_I M_i$, where $(M_i)_{I}\subseteq \fp(\G)$. Since $\t$ is of finite type, $F\cong{\varinjlim}_I (1:t)(M_i)$, showing that $\F\subseteq\varinjlim \F_0$. Furthermore, whenever $F\in \F_0$, the fact that ${\G}(F,-)$ preserves direct limits of objects in $\F$, implies that there exists $j\in I$ such that the isomorphism $F{\to}{\varinjlim}_I (1:t)(M_i)$ factors as: 
\[
\xymatrix@R=-3pt{
F\ar@/_+4pt/[rrd]_-u\ar[rrrr]|-{\cong}&&&&{\varinjlim}_I (1:t)(M_i),\\
&&(1:t)(M_j)\ar@/_+4pt/[rru]_-{\iota_{j}}
}
\] 
with $\iota_j$ the canonical map the the colimit. Then, $u$ is a section and $F$  a  summand of $(1:t)(M_j)$.
\end{proof}

\begin{lema}\label{tech_lem_dir_syst_F}
Let $\mathcal{G}$ be a locally finitely  presented Grothendieck category and $\mathbf{t}=(\mathcal{T},\mathcal{F})$ a torsion pair of finite type in $\mathcal{G}$. Given $F\in\mathcal{F}_0$ and $(F_i)_{ I}\subseteq \mathcal{F}$ directed, the following  canonical map is a monomorphism 
 \begin{equation}\label{general_tech_lem_statement_eq}
 {\varinjlim}_I\Ext_\mathcal{G}^1(F,F_i)\longrightarrow\Ext_\mathcal{G}^1(F,{\varinjlim}_I F_i).
 \end{equation}
 \end{lema}
\begin{proof}
Note that the morphism in \eqref{general_tech_lem_statement_eq} can be identified with the canonical monomorphism ${\varinjlim}_I\Ext_{\H_\mathbf{t}}^1(F[1],F_i[1])\to\Ext_{\H_\mathbf{t}}^1(F[1],{\varinjlim}_I^{(\H_\mathbf{t})}(F_i[1]))$
(see Proposition~\ref{prop_mono} {and Corollary~\ref{cor.F0}).}
\end{proof}

\subsection{Finitely presented objects in $\T[0]$}\label{subs5.3}
Our next goal is to characterize the finitely presented stalk complexes concentrated in degree $0$, paralleling the results about $\F[1]\cap\fp(\H_\t)$ from Subsection~\ref{subs5.2}.

\begin{defi}\label{def_t0}
Let $\T_0$ be the class of all objects $T\in \fp(\G)\cap \T$ such that, for each direct system $(F_i)_{I}\subseteq \mathcal{F}$, the following canonical map is an isomorphism for $k=1$ and a monomorphism for $k=2$
\[
{\varinjlim}_I\Ext_\G ^k(T,F_i)\longrightarrow\Ext_\G ^k(T,{\varinjlim}_I F_i).
\] 
\end{defi}

Our next result, whose proof is a consequence of  Corollary~\ref{cor.fpobjects-in-heart}, gives  the desired characterization.

%When the torsion pair $\t$ is of finite type, the class $\T_0$ can be used to completely describe the stalk complexes concentrated in degree $0$ that are finitely presented objects in the heart $\H_\t$. We omit the proof of the following corollary, as it is an easy consequence of Corollary~\ref{cor.fpobjects-in-heart}:

\begin{cor} \label{cor.fp-torsion-stalks}
Let $\G$ be a Grothendieck category and $\t=(\T,\F)$ a torsion pair of finite type in $\G$.
The following assertions are equivalent for an object $T\in\mathcal{T}$:
\begin{enumerate}[\rm (1)]
\item $T[0]$ is a finitely presented object in $\H_\mathbf{t}$;
\item  $T\in \T_0$.
\end{enumerate}
\end{cor}
%\begin{proof}
%The double implication $(1)\Longleftrightarrow (2)$ follows by applying  to $M=T[0]$. 
%\end{proof}

In the following lemma we give a more explicit description of the class $\T_0$ in a particular case. Indeed, if the torsion pair $\t=(\T,\F)$ is of finite type, $\mathcal{H}_\mathbf{t}$ is locally finitely presented and the induced torsion pair  $\bar{\mathbf{t}}:=(\mathcal{F}[1],\mathcal{T}[0])$ restricts to $\fp(\mathcal{H}_\mathbf{t})$, then $\T_0=\T\cap\fp(\G)$.

\begin{lema} \label{lem.subcategories of Tcapfp}
Let $\G$ be a Grothendieck category and  $\mathbf{t}=(\mathcal{T},\mathcal{F})$ a torsion pair of finite type in $\G$. Then, 
\[
\mathcal{T}_0\subseteq H^0(\fp(\mathcal{H}_\mathbf{t}))\subseteq\add(H^0(\fp(\mathcal{H}_\mathbf{t})))\subseteq\mathcal{T}\cap\fp(\G).
\] 
Moreover, the following assertions hold true:
\begin{enumerate}[\rm (1)]
\item if $\bar{\mathbf{t}}:=(\mathcal{F}[1],\mathcal{T}[0])$ restricts to $\fp(\mathcal{H}_\mathbf{t})$, then $\mathcal{T}_0=H^0(\fp(\mathcal{H}_\mathbf{t}))=\add(H^0(\fp(\mathcal{H}_\mathbf{t})))$;
\item if $\mathcal{H}_\mathbf{t}$ is locally finitely presented, then $\add(H^0(\fp(\mathcal{H}_\mathbf{t})))=\mathcal{T}\cap\fp(\G)$. 
\end{enumerate}
\end{lema}
\begin{proof}
By Corollary~\ref{cor.fpobjects-in-heart}, we know that $H^0(M)\in\mathcal{T}\cap\fp(\G)$, for all $M\in\fp(\mathcal{H}_\mathbf{t})$, so the inclusion $\add(H^0(\fp(\mathcal{H}_\mathbf{t})))\subseteq\mathcal{T}\cap\fp(\G)$ follows. By Corollary~\ref{cor.fp-torsion-stalks}, we also have that $\mathcal{T}_0\subseteq H^0(\fp(\mathcal{H}_\mathbf{t}))$ since $T=H^0(T[0])${, for all $T\in \T_0$}.  We can now proceed with the proof of assertions (1) and (2):

\smallskip\noindent
(1). Using Corollary~\ref{cor.fp-torsion-stalks}, we get that $\bar{\mathbf{t}}$ restricts to $\fp(\mathcal{H}_\mathbf{t})$ if, and only if, $H^{-1}(M)[1]\in\fp(\mathcal{H}_\mathbf{t})$ and $H^0(M)\in\mathcal{T}_0$, for all $M\in\fp(\mathcal{H}_\mathbf{t})$. Then the inclusion $H^0(\fp(\mathcal{H}_\mathbf{t}))\subseteq\mathcal{T}_0$ holds, which implies the equality $\add(H^0(\fp(\mathcal{H}_\mathbf{t})))=\mathcal{T}_0$, since $\mathcal{T}_0$ is closed under taking direct summands. 

\smallskip\noindent
(2).  If $\mathcal{H}_\mathbf{t}$ is locally finitely presented and $T\in\mathcal{T}\cap\fp(\G)$, there is a direct system $(M_i)_{ I}$ in $\fp(\mathcal{H}_\mathbf{t})$ such that $T[0]\cong{\varinjlim}_I^{(\mathcal{H}_{\t})} M_i$. The functor $H^0_{\restriction\mathcal{H}_\mathbf{t}}\colon\mathcal{H}_\mathbf{t}\to\G$ is right exact and it preserves coproducts (so also direct limits) and $T=H^0(T[0])\cong{\varinjlim}_I H^0(M_i)$. Since $T\in\fp(\G)$,  there exists $j\in I$ such that the canonical map $\iota_j\colon H^0(M_j)\to T$ to the direct limit is a retraction. Hence, $T\in\add(H^0(\fp(\mathcal{H}_\mathbf{t})))$.
\end{proof}

\subsection{An explicit description of $\fp(\Ht)$, for good enough $\G$}
As we announced in the introduction to this section, the characterizations of finitely presented stalk complexes in the previous subsections can be used to give a very explicit description of $\fp(\Ht)$, under suitable assumptions on $\G$:

\begin{prop}\label{char_of_fp_under_dot+restricts}
Let $\G$ be a Grothendieck category, $\mathbf{t}=(\mathcal{T},\mathcal{F})$ a torsion pair of finite type in $\G$, and suppose that the following condition is satisfied:
\begin{enumerate}[\rm ($\dag$)]
\item  $\Ext_\G ^k(T,-)\colon \G\to \Ab$ preserves direct limits of objects in $\mathcal{F}$, for all $T\in\mathcal{T}\cap\fp(\G)$ and $k=1,\,2$.
\end{enumerate}
Then, an object $M\in\mathcal{H}_\mathbf{t}$ is in $\fp(\mathcal{H}_\mathbf{t})$ if, and only if, $H^{-1}(M)\in\mathcal{F}_0$ and $H^0(M)\in\fp(\G)$, that is,
\[
\fp(\H_\t)=\F_0[1]*\T_0[0]=\F_0[1]*(\fp(\G)\cap \T)[0].
\]
As a consequence, the induced torsion pair $\bar{\mathbf{t}}=(\mathcal{F}[1],\mathcal{T}[0])$ in $\mathcal{H}_\mathbf{t}$ restricts to $\fp(\mathcal{H}_\mathbf{t})$. 
\end{prop}

Before proceeding with the proof of the above proposition, let us comment on condition ($\dag$): by definition of $\T_0$ one always has that $\T_0\subseteq \fp(\G)\cap \T$; condition ($\dag$) implies (but it is stronger than) the converse inclusion, so that $\T_0=\fp(\G)\cap \T$. Note also that any $T\in \fp_3(\G)\cap \T$ satisfies condition ($\dag$) so, if $\G$ is locally coherent, condition ($\dag$) is always satisfied.

\begin{proof}[Proof of Proposition~\ref{char_of_fp_under_dot+restricts}]
Condition ($\dag$) guarantees that $(\mathcal{T}\cap\fp(\G))[0]\subseteq\fp(\Ht)$ (see Corollary~\ref{cor.fp-torsion-stalks}), while $\F_0[1]\subseteq \fp(\Ht)$ by Corollary~\ref{cor.F0}. Thus, $\F_0[1]*(\mathcal{T}\cap\fp(\G))[0]\subseteq \fp(\Ht)*\fp(\Ht)=\fp(\Ht)$. Conversely, if $M\in\fp(\Ht)$ then $H^0(M)[0]\in\fp(\Ht)$, by ($\dag$) and Corollaries~\ref{cor.fpobjects-in-heart} and \ref{cor.fp-torsion-stalks}. Consider now the following short exact sequence in $\H_\t$
\[
0\longrightarrow H^{-1}(M)[1]\longrightarrow M\longrightarrow H^0(M)[0]\longrightarrow 0
\] 
and a direct system $(F_i)_{I}\subseteq\mathcal{F}$. Apply the functors $\Ht(-,{\varinjlim}_I (F_i[1]))\cong\Ht(-,({\varinjlim}_IF_i)[1])$ and ${\varinjlim}_I\Ht(-,F_i[1])$ to the above sequence to get the following commutative diagram with exact rows: 
\[
\scalebox{0.835}{
\xymatrix@C=13pt{
{\varinjlim}_I\Ext_\G^1(H^0(M),F_i)\ar[r]\ar[d]^{f_1}&{\varinjlim}_I\Ht(M,F_i[1])\ar[r]\ar[d]^{f_2} &{\varinjlim}_I\G(H^{-1}(M),F_i)\ar[r]\ar[d]^{f_3}&{\varinjlim}_I\Ext_\G^2(H^0(M),F_i)\ar[r]\ar[d]^{f_4}&{\varinjlim}_I\Ext_{\Ht}^1(M,F_i[1])\ar[d]^{f_5}\\
\Ext_\G^1(H^0(M),{\varinjlim}_I F_i)\ar[r]&\Ht(M,{\varinjlim}_I (F_i[1]))\ar[r]&\G(H^{-1}(M),{\varinjlim}_I F_i)\ar[r]&\Ext_\G^2(H^0(M),{\varinjlim}_I F_i)\ar[r]&\Ext_{\Ht}^1(M,{\varinjlim}_I (F_i[1])) 
}
}
\]
By the condition ($\dag$) and Corollary~\ref{cor.fpobjects-in-heart}, we know that $f_1$ and $f_4$ are isomorphisms. Furthermore, using that  $M\in\fp(\Ht)$ and Proposition~\ref{prop_mono}, we get that  $f_2$ is an isomorphism and $f_5$ is a monomorphism. The Five Lemma then implies that $f_3$ is an isomorphism, that is, $H^{-1}(M)\in\mathcal{F}_0$. Finally, since $\mathcal{F}[1]\cap\fp(\Ht)=\mathcal{F}_0[1]$, we also get that $\bar{\mathbf{t}}$ restricts to $\fp(\Ht)$.  
\end{proof}

\subsection{Finitely $2$-presented objects in the heart} \label{subsection.finitely-2-presented}
Let $\G$ be a Grothendieck category and  $\t=(\T,\F)$ a torsion pair of finite type in $\G$. 
By Theorem~\ref{main_thm_fintype_qcotilt_cosilt}, we can fix a $2$-term cosilting complex 
\[
E:\qquad\cdots \longrightarrow 0\longrightarrow E^{-1}\stackrel{\sigma}{\longrightarrow}E^0\longrightarrow 0 \longrightarrow \cdots,
\] 
such that $\mathbf{t}$ is the associated torsion pair. Use Proposition~\ref{prop.injcogenerator-from-cosilting} to get a short exact sequence 
\[
0\longrightarrow W\longrightarrow E\stackrel{p}{\longrightarrow} (1:t)(H^0(E))[0]\longrightarrow 0\qquad \text{in $\Ch(\G)$,}
\] 
with $W$ an injective cogenerator of $\H_{\mathbf{t}}$, {and $H^{-1}(W)=H^{-1}(E)$, so the following sequence is exact:
\[
0\longrightarrow H^{-1}(E)[1]\longrightarrow W\longrightarrow {R}\longrightarrow 0\qquad \text{in $\Ch(\G)$,}
\]
with $R$ concentrated in degrees $-1$ and $0$, and such that $R\cong H^{0}(W)[0]\cong t(H^{0}(E))[0]$ in $\Der(\G)$. Both sequences are functorial in $E$ and the second one,} when viewed  in $\H_\mathbf{t}$, is the (exact) torsion sequence {$0\to H^{-1}(W)[1]\to W\to H^0(W)[0]\to 0$} of  $W$ with respect to $\bar{\mathbf{t}}=(\F[1],\T[0])$. Hence, given a direct system $(E_\lambda)_{\lambda\in\Lambda}\subseteq \Prod_{\Ch(\G)}(E)$, we have two direct systems of short exact sequences in $\Ch(\G)$:
\[
(0\longrightarrow W_\lambda\longrightarrow E_\lambda\stackrel{p_\lambda}{\longrightarrow} (1:t)(H^0(E_\lambda ))[0]\longrightarrow 0)_{\Lambda}\ \ 
\text{and}
\ \ 
(0\longrightarrow H^{-1}(E_\lambda)[1]\longrightarrow W_\lambda\longrightarrow{R_\lambda}\longrightarrow 0)_{\Lambda},
\] 
such that, viewed in $\Der(\G)$,  $W_\lambda\in\Inj\H_\mathbf{t}$. Given $X\in\Der(\G)$, apply $\Der(\G)(X,-)$ to the direct system on the right-hand side above to get the following natural morphism of Abelian groups:
\begin{equation}\label{def_xi_eq}
\Der(\G)(X,({\varinjlim}_\Lambda t(H^0(E_\lambda)))[1]){\cong \Der(\G)(X,({\varinjlim}_\Lambda R_{\lambda})[1])} \overset{\xi_X}\longrightarrow \Der(\G)(X,({\varinjlim}_\Lambda H^{-1}(E_\lambda))[3]),%\qquad \text{in $\Ab$}.
\end{equation}
{by  the exactness of direct limits in $\Ch(\G)$ and since $t(H^{0}(E_{\lambda}))=H^0(R_{\lambda}),$ for all $\lambda \in \Lambda$. }

\begin{rem}\label{rem_stalk_for_map_ext_3}
When $X=G[k]$ for some $G\in\G$ and $k\in \{0,1\}$, we can describe the morphism $\xi_X$ quite explicitly: given $[\hat{\epsilon}]\in\Ext_\G^2({\varinjlim}_{\Lambda} t(H^0(E_\lambda)),{\varinjlim}_{\Lambda} H^{-1}(E_\lambda))$, represented by the exact sequence below
\[
\xymatrix@C=18pt{
\hat{\epsilon}:\hspace*{0.5cm} 0\ar[r]&{\varinjlim}_{\Lambda} H^{-1}(E_\lambda)\ar[r]& {\varinjlim} _{\Lambda}E_\lambda^{-1}\ar[rr]^-{{\varinjlim}_{\Lambda}\tilde{\sigma}_\lambda}&& {\varinjlim}_{\Lambda}  W_\lambda^0\ar[rr]^-{{\varinjlim}_{\Lambda}\tilde{p}_\lambda}&&{\varinjlim}_{\Lambda} t(H^0(E_\lambda))\ar[r]& 0,
}
\] 
define a map $\Ext_\G^{1-k}(G,{\varinjlim}_\Lambda t(H^0(E_\lambda)))\to \Ext_\G^{3-k}(G,{\varinjlim}_\Lambda H^{-1}(E_\lambda))$ such that $\delta\mapsto\hat{\epsilon}\cdot\delta$,  where  ``\hspace{2.1pt}$\cdot$'' is the Yoneda product (see {\cite[Chapter~III]{MacLane}}), then:
\[
\xymatrix@C=10pt@R=5pt{
\Ext_\G^{1-k}(G,{\varinjlim}_\Lambda t(H^0(E_\lambda)))\ar@{=}[d]\ar[rr]&&\Ext_\G^{3-k}(G,{\varinjlim}_\Lambda H^{-1}(E_\lambda))\ar@{=}[d]\\
\xi_{G[k]}\colon \Der(\G)(G[k],({\varinjlim}_\Lambda t(H^0(E_\lambda)))[1])\ar[rr]&& \Der(\G)(G[k],({\varinjlim}_\Lambda H^{-1}(E_\lambda))[3])
}
\]
\end{rem}

\begin{prop} \label{prop.objects in fp2(H)}
Let $\G$ be a Grothendieck category and $\t=(\T,\F)$ a torsion pair of finite type in $\G$, induced by the $2$-term cosilting complex 
\[
E:\qquad\cdots \longrightarrow 0\longrightarrow E^{-1}\stackrel{\sigma}{\longrightarrow}E^0\longrightarrow 0 \longrightarrow \cdots.
\] 
The following statements are equivalent for $M\in\fp(\H_\mathbf{t})$:
 \begin{enumerate}[\rm (1)]
 \item $M\in\fp_2(\H_\mathbf{t})$;
 \item the following assertions hold true:
 \begin{enumerate}[\rm (2.1)]
 \item the functor $\Der(\G)(M,-[2])_{\restriction \F}\colon\F\to\Ab$ preserves direct limits;
 \item  for any direct system $(E_\lambda)_{\Lambda}\subseteq\Prod_{\Ch(\G)}(E)$, the following map (see \eqref{def_xi_eq}) is monic:
 \[
 \xi_M\colon\Der(\G)(M,({\varinjlim}_\Lambda t(H^0(E_\lambda)))[1])\longrightarrow \Der(\G)(M,({\varinjlim}_\Lambda H^{-1}(E_\lambda))[3]).
 \]
 \end{enumerate}
\end{enumerate}
\end{prop}
\begin{proof}
As discussed at the beginning of this subsection, for any direct system $(E_\lambda )_{\Lambda}$ as in $(2.2)$, we have two  direct systems of short exact sequences  in $\Ch(\G)$:
\[
(0\longrightarrow W_\lambda\longrightarrow E_\lambda\stackrel{p_\lambda}{\longrightarrow} (1:t)(H^0(E_\lambda ))[0]\longrightarrow 0)_{\Lambda}\ \ 
\text{and}
\ \ 
(0\longrightarrow H^{-1}(E_\lambda)[1]\longrightarrow W_\lambda\longrightarrow {R_\lambda}\longrightarrow 0)_{\Lambda},
\]  
 where $R_\lambda$ is a complex concentrated in degrees $-1$ and $0$ and isomorphic to $t(H^{0}(E_\lambda))[0]$ in $\Der(\G)$.
The direct limit of the latter gives the following triangle in $\Der(\G)$:
\begin{equation}\label{to the triangle}
\xymatrix{
(\varinjlim_{\Lambda} H^{-1}(E_\lambda))[1]\longrightarrow \varinjlim_{\Lambda} W_\lambda\longrightarrow (\varinjlim_{\Lambda} t(H^0(E_\lambda)))[0]\longrightarrow  (\varinjlim_{\Lambda} H^{-1}(E_\lambda))[2].}
\end{equation}

\noindent
(1)$\Rightarrow$(2). By  \cite[R\'emarque~3.1.17]{BBD}, the two functors $\Der(\G)(M,-[2])_{\restriction \F}\cong\Ext_{\H_\mathbf{t}}^1(M,-[1])_{\restriction  \F}\colon\F\to\Ab$ are naturally isomorphic and they preserves direct limits because $({\varinjlim}^{(\G)}_{\Lambda} F_\lambda)[1]\cong{\varinjlim}^{(\H_\mathbf{t})}_{\Lambda}(F_\lambda[1])$, for each direct system $(F_\lambda )_{\Lambda}$ in $\F$, by Lemma~\ref{rem. stalks}.\\
On the other hand, let $(E_\lambda)_{\Lambda}$ be  a direct system in $\Prod_{\Ch(\G)}(E)$, and consider the direct systems of exact sequences in $\Ch(\G)$ {constructed at the beginning of the proof}. {In particular, $(W_\lambda)_{\Lambda}\subseteq \Inj \Ht$ and ${\varinjlim}^{(\Ch(\G))}_{\Lambda} W_\lambda\cong{\varinjlim}^{(\H_{\mathbf{t}})}_{\Lambda}W_\lambda$, by \cite[Lemma~4.4]{PS1}. Since $M\in \fp_2(\Ht)$, Lemma~\ref{lema_BP} implies that:} 
\begin{equation}\label{eq_vanish_ext_prop_3}
0=\Ext_{\H_\mathbf{t}}^1(M,{\varinjlim}_{\Lambda} W_\lambda)\cong\Der(\G)(M,({\varinjlim}_{\Lambda} W_\lambda)[1]).
\end{equation}
Applying now the cohomological functor $\Der(\G)(M,-)$ to the triangle \eqref{to the triangle}
%\[
%({\varinjlim}_{\Lambda} W_\lambda)[1]\to ({\varinjlim}_{\Lambda} t(H^0(E_\lambda)))[1]\to ({\varinjlim}_{\Lambda} H^{-1}(E_\lambda ))[3]\to({\varinjlim}_{\Lambda} W_\lambda)[2],
%\] 
and analyzing the induced long exact sequence, we deduce that $\xi_M$ is a monomorphism by \eqref{eq_vanish_ext_prop_3}. 

\smallskip\noindent
(2)$\Rightarrow$(1).  By Lemma~\ref{lema_BP}, condition (1) is equivalent to say that the following functor 
\[
\Ext_{\H_\mathbf{t}}^1(M,-)\cong\Der (\G)(M,-[1])_{\restriction\H_\mathbf{t}}\colon \H_\mathbf{t}\longrightarrow\Ab
\] 
vanishes on colimits of direct systems in $\Prod_{\H_\mathbf{t}}(W)$, for $W$ as in Proposition~\ref{prop.injcogenerator-from-cosilting}. For such a direct system $(W_\lambda)_{\Lambda}$, Corollary~\ref{coro_derivators} gives a direct system of exact sequences  $(0\rightarrow E^{I_\lambda}\to Y_\lambda\to W'_\lambda [1]\rightarrow 0)_{\Lambda}$ in $\Ch(\G)$, with $Y_\lambda\in\V_\mathbf{t}[1]= {}^{\perp_{>0}}E$  for all $\lambda\in \Lambda$, and such that $(W'_\lambda)_{\Lambda}$ is isomorphic to $(W_\lambda )_{\Lambda}$ in $\Der(\G)$.
 Applying the cohomological functor $H_\mathbf{t}^0$, % to the triangle $Y_\lambda [-1]\to W'_\lambda\stackrel{u_\lambda}{\longrightarrow}E^{I_\lambda}\stackrel{}{\to}Y_\lambda [0]$, to 
 one sees that $H_\mathbf{t}^0(u_\lambda )\colon W'_\lambda\cong H_\mathbf{t}^0(W'_\lambda)\to H_\mathbf{t}^0(E^{I_\lambda})$ is a (necessarily split) monomorphism in $\mathcal{H}_\mathbf{t}$, for all \mbox{$\lambda\in\Lambda$,} {where each $u_\lambda$ is constructed as in the proof of Corollary~\ref{coro_derivators}}.  Thus, the induced monomorphism 
 \[
 {\varinjlim}_\Lambda^{(\H_\mathbf{t})} W'_\lambda\longrightarrow{\varinjlim}_\Lambda^{(\H_\mathbf{t})}H_\mathbf{t}^0(E^{I_\lambda})
 \] 
 is a direct limit of split monomorphisms, and it is then kept monic by  $\Ext_{\H_\mathbf{t}}^1(M,-)$, as $M\in\fp(\H_\mathbf{t})$. Up to replacing $(W_\lambda)_{\Lambda}$ by  $(W'_\lambda)_{\Lambda}$ and, if needed, replacing the latter  by $(H_\mathbf{t}^0(E^{I_\lambda}))_{\Lambda}$, we can assume that there is a direct system $(E_\lambda =E^{I_\lambda})_{\Lambda}$ in $\Prod_{\Ch(\G)}(E)$, such that $W_\lambda \cong\Ker(\pi_\lambda)$ is the kernel of the  the canonical projection $\pi_\lambda \colon E_\lambda\twoheadrightarrow (1:t)(H^0(E_\lambda))[0]$, for all $\lambda\in\Lambda$.  
Consider now the following commutative diagram and note that {the vertical arrow on the left is an isomorphism since $M\in\fp(\H_\t)$, while the one on the right is an isomorphism by (2.1) and \cite[R\'emarque~3.1.17]{BBD}:}%\textcolor{blue}{\bf (Rev.\,17)} \textcolor{red}{isomorphisms (see Lemma~\ref{rem. stalks} and \cite{BBD} for the rightmost vertical morphism)}  
\[
\xymatrix@R=15pt{
{\varinjlim}_\Lambda\H_{\mathbf{t}}(M,t(H^0(E_\lambda))[0])\ar[d]|(.45){\cong}\ar@{->>}[r]&{\varinjlim}_{\Lambda}\Ext_{\H_{\mathbf{t}}}^1(M,H^{-1}(E_\lambda)[1])\ar[d]|(.45){\cong}\\
\H_{\mathbf{t}}(M,({\varinjlim}_{\Lambda} t(H^{0}(E_\lambda))[0])\ar@{->>}[r]&\Ext_{\H_{\mathbf{t}}}^1(M,({\varinjlim}_{\Lambda} H^{-1}(E_\lambda))[1]).
}
\]
Furthermore, the upper {horizontal arrow} is an epimorphism since $\Ext_{\H_\mathbf{t}}^1(M,W_\lambda)=0$, for all $\lambda\in\Lambda$, so that the lower {horizontal} arrow is an epimorphism too. In fact, this last map is {isomorphic} to the canonical morphism $\Der(\G)(M,({\varinjlim}_{\Lambda} t(H^{0}(E_\lambda))[0])\twoheadrightarrow\Der(\G)(M,({\varinjlim}_{\Lambda} H^{-1}(E_\lambda))[2])$ so that, when we apply $\Der(\G)(M,-)$ to the triangle \eqref{to the triangle}
%\[
%{\varinjlim}_{\Lambda} W_\lambda\longrightarrow ({\varinjlim}_{\Lambda} t(H^0(E_\lambda)))[0]\longrightarrow ({\varinjlim}_{\Lambda} H^{-1}(E_\lambda))[2]{\longrightarrow}({\varinjlim}_{\Lambda} W_\lambda)[1],
%\] 
we find the following exact sequence (with a $0$ on the left):
\[
\scalebox{0.97}{
\xymatrix{0\longrightarrow \Der(\G)(M,({\varinjlim}_{\Lambda} W_\lambda)[1])\longrightarrow\Der(\G)(M,({\varinjlim}_{\Lambda} t(H^0(E_\lambda))[1])\stackrel{\xi_M}{\longrightarrow}\Der(\G)(M,({\varinjlim}_{\Lambda} (1:t)(H^{-1}(E_\lambda))[3]). }
}
\] 
By assumption, $\xi_M$ is monic, so  $\Der(\G)(M,({\varinjlim}_{\Lambda} W_\lambda)[1])=0$ and thus $\Ext_{\H_\mathbf{t}}^1(M,{\varinjlim}_{\Lambda}^{(\H_\mathbf{t})}W_\lambda)=0$, since ${\varinjlim}_{\Lambda} W_\lambda \cong{\varinjlim}_{\Lambda}^{(\H_\mathbf{t})}W_\lambda$ (by  \cite[Proposition~4.2]{PS1}), as desired.
\end{proof}

\subsection{The category $\underline\F$}

Let $\G$ be a Grothendieck category and  $\mathbf{t}=(\T,\F)$ a torsion pair in $\G$. Then, 
\begin{itemize}
\item $\underline{\F}$ denotes the  subcategory of $\G$ of all quotients of objects in $\F$.
\end{itemize}

\begin{ejem}\label{example_modules_underline}
{Suppose that,} in the above setting, $\G=\mod\A$ is the category of modules over a small preadditive category $\A$. {The torsion ideal $t(\A)$ of $\A$ (with respect to $\mathbf{t}$) is defined as
\[
t(\A)(a, b) := t(H_b)(a),\qquad\text{for all $a,\, b \in \A$},
\]
where $H_b:=\A(-,b)\colon \A^{\op} \to \Ab$ (see \cite[Section~3]{PSV2}). In this situation,} $\underline{\F}$ is just the  category $\mod{(\A/t(\A))}$, {viewed as a full subcategory of $\mod{\A}$ in the obvious way.}% where $t(\A)$ is the torsion ideal of $\A$ (see \cite[Section~3]{PSV2} for the definition of this ideal).
\end{ejem}

\begin{lema}\label{F0generates_underlineF}
Let $\G$ be a Grothendieck category and  $\mathbf{t}=(\T,\F)$ a torsion pair of finite type. Then, 
\begin{enumerate}[\rm (1)]
\item $\underline \F$ is coreflective in $\G$; 
\item $\underline \F$ is a Grothendieck category.
\end{enumerate} 
Furthermore, if $\G$ is locally finitely presented, then
\begin{enumerate}[\rm (3)]
\item $\F_0$ is a skeletally small class of generators of $\underline{\F}$.
\end{enumerate}
\end{lema}
\begin{proof}
(1) and (2) follow by \cite[Proposition~4.2]{PSV} (and they are also verified inside the proof of \cite[Theorem~4.18]{PS4}). Furthermore, {when $\G$ is locally finitely presented,} $\F_0$ is skeletally small by Corollary~\ref{cor.F0}, which also shows that $\F=\varinjlim\F_0$, so $\Gen(\F_0)=\Gen(\F)=\underline\F$.
\end{proof}

\begin{lema}\label{useful_lema_segunda_parte}
Let $\mathcal{G}$ be a locally finitely  presented Grothendieck category, $Q$  a quasi-cotilting object, and $\mathcal{F}=\Cogen(Q)={}^{\perp_1}Q\cap\underline{\Cogen}(Q)$. Suppose that there is a set $\mathcal X\subseteq\mathcal{F}_0\subseteq\pres_2(\mathcal{X})\subseteq \G$ such that 
\begin{enumerate}[\rm (\textsc{i})]
\item $\Ext_\mathcal{G}^1(X,{\varinjlim}_I Q_i)=0$, for all $X\in\mathcal{X}$ and all direct systems $(Q_i)_{I}\subseteq\Prod(Q)$.
\end{enumerate} 
Then, the following statements hold true:
 \begin{enumerate}[\rm (1)]
\item the functor $\Ext_\mathcal{G}^1(F,-)_{\restriction \mathcal{F}}\colon\mathcal{F}\to\Ab$ preserves direct limits, for all $F\in\mathcal{F}_0$;
\item $\F_0=\F\cap\fp(\underline{\F})$; 
\item $\underline{\F}$ is locally finitely presented. 
\end{enumerate} 
\end{lema}
\begin{proof}
Although $\F$ might fail to be a Grothendieck category, we still have that $\Ext_\G^1(-,Q')_{\restriction\F}=0$, for all $Q'\in\Prod(Q)$. With a little abuse of notation, we define: 
 \[
\fp_1(\F):=\F_0\quad \text{and}\quad \fp_2(\F):=\{F\in\F_0:\Ext_\G^1(F,-)_{\restriction\F}\text{ preserves direct limits}\}.
\]

\noindent
%(1). \textcolor{blue}{\bf (Rev.\,18) }\textcolor{red}{Note that $Q$ is 1-cotilting object of $\underline{\F}$ and its associated torsion pair coincide with the pair $\mathbf{t}':=(\T\cap \underline{\F},\F)$. In this case $Q[1]$ is an injective cogenerator of $\mathcal{H}_{\mathbf{t}'}$. On the other hand, Corollary~\ref{cor.F0} says that $\X[1]\subseteq \fp(\mathcal{H}_{\mathbf{t}'})$ and using the hypothesis on $Q$ we deduce $\X[1]\subseteq \fp_2(\mathcal{H}_{\mathbf{t}'})$ (see \cite[Theorem~B.1]{BP2}). In particular } $\X\subseteq\fp_2(\F)$. Similarly, one can adapt the proof of Lemma~\ref{lem.n-fp subcategories} to see that any monomorphism $u\colon F\rightarrowtail F'$, with $F\in\fp_1(\F)=\F_0$, $F'\in\fp_2(\F)$ and $\Coker(u)\in\F$, has the property that $\Coker(u)\in\fp_2(\F)$. Our hypotheses on $\X$ then imply that $\F_0=\fp_2(\F)$, since  {$\F_0\subseteq \pres_2(\X)$}.
(1). Exploiting that $\F=\Copres(Q)$ (see Remark~\ref{rem_copres=cogen}), one can use 
 the proof of \cite[Theorem~B.1]{BP2} to show that $\X\subseteq\fp_2(\F)$. {Indeed, given $F\in \F$, the canonical inclusion $\iota_F\colon F \to \Psi_Q(F)=Q^{\G(F,Q)}$ is a $\Prod(Q)$-preenvelope, so $\Coker(\iota_F)\in \F={}^{\perp_1}Q\cap \underline{\F}$. Thus, replacing $\G$, $I$ and $\Ext^{1}_{\G}(=,-)$ by $\F$, $Q$, and $\Ext^{1}_{\G}(=,-)_{\restriction\F^{\op}\times \F}$}, the proof of \cite[Theorem~B.1]{BP2} applies literally. Similarly, one may adapt  the proof of Lemma~\ref{lem.n-fp subcategories} to see that, for any monomorphism $u\colon F\to F'$ such that $F\in\fp_1(\F)=\F_0$, $F'\in\fp_2(\F)$ and $\Coker(u)\in\F$, one has $\Coker(u)\in\fp_2(\F)$. Thus, $\F_0=\fp_2(\F)$, as  {$\F_0\subseteq \pres_2(\X)$ and $\X\subseteq\fp_2(\F)$}.

\smallskip\noindent
(2). Let $\Phi_{\mathcal X}\colon \underline{\F}\to \underline{\F}$ and  $\pi \colon \Phi_{\mathcal X}\Rightarrow  \id_{\underline{\F}}$ be the functor and the epimorphic natural transformation built in \eqref{def_Phi_Psi} (see Subsection \ref{subsectionEqII}). Then, for a direct system  $(M_i)_{I}\subseteq\underline{\F}$, we get the following short exact sequences:
\[
(0\rightarrow K_i:=\Ker(\pi_{M_i})\to\Phi_\mathcal{X}(M_i)\stackrel{\pi_{M_i}\ \ }{\longrightarrow}M_i\rightarrow 0)_{I},\quad 0\rightarrow {\varinjlim}_I K_i\longrightarrow{\varinjlim}_I \Phi_\mathcal{X}(M_i){\longrightarrow}{\varinjlim}_I M_i\rightarrow 0.
\]
For each $j\in I$, both $K_j$ and $\Phi_\mathcal{X}(M_j)$ are in $\mathcal{F}$. Thus, also ${\varinjlim}_I K_i\in \mathcal{F}$ and ${\varinjlim}_I\Phi_{X}(M_i)\in \mathcal{F}$. Moreover,  for each $F\in\mathcal{F}_0$, we obtain the following commutative diagram with exact rows:
\[
\scalebox{0.9}{
\xymatrix@C=18pt{
{\varinjlim}_I\underline{\F}(F,K_i)\ar@{^(->}[r]\ar[d]^{f_1}&{\varinjlim}_I\underline{\F}(F,\Phi_\mathcal{X}(M_i))\ar[r]\ar[d]^{f_2}&{\varinjlim}_I\underline{\F}(F,M_i)\ar[r]\ar[d]^{f_3}&{\varinjlim}_I\Ext_{\underline{\F}}^1(F,K_i)\ar[r]\ar[d]^{f_4}&{\varinjlim}_I\Ext_{\underline{\F}}^1(F,\Phi_\mathcal{X}(M_i))\ar[d]^{f_5}\\
\underline{\F}(F,{\varinjlim}_I K_i)\ar@{^(->}[r]&\underline{\F}(F,{\varinjlim}_I\Phi_\mathcal{X}(M_i))\ar[r]&\underline{\F}(F,{\varinjlim}_I M_i)\ar[r]&\Ext_{\underline{\F}}^1(F,{\varinjlim}_I K_i)\ar[r]&\Ext_{\underline{\F}}^1(F,{\varinjlim}_I\Phi_\mathcal{X}(M_i)).
}}
\]
The maps $f_1$ and $f_2$ are isomorphisms by definition of $\F_0$, while $f_4$ and $f_5$ are isomorphisms by part (1), forcing also $f_3$ to be an isomorphism, by the Five Lemma. Thus, $F\in \fp(\underline{\F})$, proving the inclusion $\mathcal{F}_0\subseteq\mathcal{F}\cap\fp(\underline{\F})$, while the converse inclusion follows by Corollary~\ref{cor.F0}.

\smallskip\noindent 
(3) follows by part (2) and Lemma~\ref{F0generates_underlineF}(3).
\end{proof}

Let us remark that, if $\H_\mathbf{t}$ is locally coherent, the hypotheses of Lemma~\ref{useful_lema_segunda_parte}, as well as those of the following Lemma~\ref{useful_lema_terza_parte}, are automatically satisfied (see the forthcoming Lemma~\ref{rem_tech_lemma}).

\begin{lema}\label{useful_lema_terza_parte}
Let $\mathcal{G}$ be a locally finitely  presented Grothendieck category, $Q$  a quasi-cotilting object, and $\mathcal{F}=\Cogen(Q)={}^{\perp_1}Q\cap\underline{\Cogen}(Q)$. Suppose that there is a set $\mathcal X\subseteq\mathcal{F}_0\subseteq \mathrm{gen}(\X)\subseteq \G$ such that:
\begin{enumerate}[\sc (i)]
\item $\Ext_\mathcal{G}^1(X,{\varinjlim}_I Q_i)=0$, for all $X\in\mathcal{X}$ and all direct systems $(Q_i)_{I}\subseteq\Prod(Q)$;
\item $\Ker(f)\in \F_0$ for any epimorphism $f\colon X\twoheadrightarrow F$, with $F\in \F_0$ and $X\in \mathrm{sum}(\X)$.
\end{enumerate}
Then, the following statements hold true:
 \begin{enumerate}[\rm (1)]
 \item $\F_0\subseteq \pres_2(\X)$;
  \item $\Ext_{\underline{\F}}^n(F,{\varinjlim}_I Q_i)=0$, for all $n>0$, $F\in\F_0$ and each direct system $(Q_i)_{I}$ in $\Prod(Q)$;
 \item $\F_0=\F\cap\fp_\infty(\underline{\F})$.
 \end{enumerate}
 \end{lema}
 \begin{proof}
 (1) follows easily by {\sc (ii)} and the condition $\mathcal{F}_0\subseteq \mathrm{gen}(\X)$.
 
 \smallskip\noindent
(2). Let $F\in \F_0$, $F'\in\F$, $n> 0$ and let $\epsilon$ be an $(n+1)$-fold extension of $F$ by $F'$ in $\underline \F$:%\footnote{(Rev.\,19)} 
 \[
\epsilon\ := \quad (\ 0\longrightarrow F'\stackrel{}\longrightarrow K_n\stackrel{}\longrightarrow \cdots\stackrel{}\longrightarrow K_1\stackrel{}{\longrightarrow}K_0\stackrel{}\longrightarrow F\longrightarrow  0\ )\ \ \text{i.e.}\ \ \   [\epsilon]\in \Ext^{n+1}_{\underline \F}(F,F').
 \]   
{Let $\epsilon_{-1}:=\epsilon$, $K_0':=K_0$ and construct, by induction on $i\geq -1$, a sequence $(\epsilon_i)_{i=-1}^n$ of representatives of the same $[\epsilon]$, where $\epsilon_{i+1}:=(0\to F'\to \cdots\to K_{i+1}'\to F_i\to \dots \to F\to 0)$ is a pullback of $\epsilon_{i}$ along an epimorphism $\pi_i\colon F_{i}\to K'_{i}$ with $F_i\in \F$. In particular,
\[
\epsilon_n\ := \quad (\  0\longrightarrow F'\longrightarrow F_n\stackrel{}{\longrightarrow} F_{n-1}\longrightarrow \cdots\stackrel{}{\longrightarrow}F_0\stackrel{\pi}{\longrightarrow} F\longrightarrow 0\ ),\qquad \text{with $F_0,\dots,F_n\in\F$.}
 \]}
%Take an epimorphism $p_0\colon F_0 \to X_0\in\underline\F$, with $F_0\in \F$; the pullback of $\epsilon$ along $p_0$, is a new $\epsilon_0\in [\epsilon]$:
% \[
%\epsilon_0\ := \quad (\ 0\rightarrow F'\longrightarrow X_n\to \cdots\to X_2\longrightarrow {X'}_1\stackrel{d_1'}{\longrightarrow}F_0\longrightarrow F\rightarrow 0\ ),\qquad \text{with  $F_0\in \F$.} 
% \] 
%
Take an epimorphism $q\colon \coprod_{ I}X_i\twoheadrightarrow F_0$, with $(X_i)_I\subseteq  \X$;  {then, $F=\Im(\pi \circ q)=\sum_J \Im( \pi \circ q \circ \iota_J),$ where $J$ ranges over the finite subsets of $I$, with $\iota_J\colon\coprod_J X_i \to \coprod_IX_i$  the canonical section. By Proposition~\ref{prop.fg-objects}, as $F\in \fp(\underline{\F})\subseteq \fg(\underline{\F})$, there is $J\subseteq I$ finite such that $\Im(\pi\circ q \circ \iota_J)=F$, i.e., $\rho:=\pi \circ q \circ \iota_J$ is epic.} % Since $F\in \fp(\underline \F)$, there is a finite subset $J\subseteq I$ such that the following composition is an epimorphism:
 %\[
%\xymatrix{
% \rho\colon \coprod_{ J}X_i\ar@{^(->}[r]^-{\iota_J}&\coprod_{ I}X_i\ar[r]^-{q}&F_0\ar[r]^-{\pi}&F.
 %}
 %\]  
Taking a pullback of $\epsilon_n$ along $q\circ\iota_J$, we finally obtain the following representative of $[\epsilon]$:
\[
\tilde\epsilon\ := \ \  (\  0\to F'\to F_n\to \cdots \to F_2\to\tilde{F}_1\stackrel{\tilde{d}}{\longrightarrow}\tilde{F}_0\stackrel{\rho}{\longrightarrow} F\to 0\ ),\ \  \text{with $\tilde F_0\in \mathrm{sum}(\X)$, $\tilde F_1, F_2,\dots, F_n\in \F$,}
\] 
so $[\epsilon]\in\Ext_{\underline{\F}}^n(\Ker(\rho),F')\cdot\Ext_{\underline{\F}}^1(F,\Ker(\rho))$ (where ``$\cdot$'' is the Yoneda product) and $\Ker(\rho)\in \F_0$ by {\sc (ii)}. Condition (2) follows by induction on $n$: the case $n=1$ is covered by Lemma~\ref{useful_lema_segunda_parte}. When $n>1$, take a direct system $(Q_i)_I\subseteq \Prod(Q)$, let $F':=\varinjlim_IQ_i$, and $[\epsilon]\in\Ext_{\underline{\F}}^n(F,F')$. By the above argument, there is $K\in \F_0$ such that $[\epsilon]\in\Ext_{\underline{\F}}^{n-1}(K,F')\cdot\Ext_{\underline{\F}}^1(F,K)$ and, by inductive hypothesis, $\Ext_{\underline{\F}}^{n-1}(K,F')=0$.

\smallskip\noindent
(3). Fix an object $F\in\F_0$ and let us show that $\Ext_{\underline{\F}}^k(F,-)$ preserves direct limits of objects in $\F$, for all $k>0$. We proceed by induction on $k$, the case $k=1$ being covered by Lemma~\ref{useful_lema_segunda_parte}. Suppose  that $k>1$, and consider a direct system $(F_i)_{I}$ in $\F$. There is then a direct system of  exact sequences $(0\rightarrow F_i\to \Psi_Q(F_i)\to F'_i\rightarrow 0)_{ I}$  in $\underline{\F}$, with $F'_i\in\F$ for all $i\in I$. Using our hypotheses, we deduce that  $\Ext_{\underline{\F}}^{k-1}(F,{\varinjlim}_I F'_i)\cong\Ext_{\underline{\F}}^{k}(F,{\varinjlim}_I F_i)$ and so, by inductive hypothesis, $\Ext_{\underline{\F}}^{k}(F,-)$ preserves direct limits of objects in $\F$. Similarly, for a direct system $(M_i)_I$ in $\underline \F$, take the direct system of exact sequences $(0\to F_i\to \Phi_\X(M_i)\to M_i\to 0)_{I}$, where $F_i,\ \Phi_\X(M_i)\in \F$. By the first part of the proof,  the canonical maps $\varinjlim_I\Ext_{\underline{\F}}^{k}(F,\Phi_\X(M_i))\to \Ext_{\underline{\F}}^{k}(F,\varinjlim_I\Phi_\X(M_i))$ and $\varinjlim_I\Ext_{\underline{\F}}^{k}(F,F_i)\to \Ext_{\underline{\F}}^{k}(F,\varinjlim_IF_i)$ are isomorphisms for all $k\geq 0$. By the Five Lemma one concludes that canonical map $\varinjlim_I\Ext_{\underline{\F}}^{k}(F,M_i)\to \Ext_{\underline{\F}}^{k}(F,\varinjlim_IM_i)$  is an isomorphism, for all $k\geq 0$. 
\end{proof}

Let us conclude the section with the following observation:
\begin{rem}\label{interpretation_of_results}
Let $\G$ be a Grothendieck category and $\t=(\T,\F)$ a torsion pair of finite type in $\G$. In this section we have given several characterizations of the objects in $\fp(\Ht)$: using the general description in Section~\ref{subs5.1}, we have identified two subcategories $\F_0\subseteq \F$ and $\T_0\subseteq \T$ such that 
\[
\fp(\Ht)\cap \F[1]=\F_0[1]\qquad\text{and}\qquad\fp(\Ht)\cap \T[0]=\T_0[0]
\]
(see Section~\ref{subs5.2} and \ref{subs5.3}, respectively). Under suitable conditions (see Proposition~\ref{char_of_fp_under_dot+restricts}), one even gets the satisfactory formula $\fp(\Ht)=\F_0[1]*\T_{0}[0]$. On the other hand, there is a fundamental asymmetry between the classes $\T_0$ and $\F_0$: the inclusion $\T_0\subseteq \T\cap \fp(\G)$ always holds, with equality under reasonable hypotheses (see Lemma~\ref{lem.subcategories of Tcapfp}) leading one to expect a strong relation between $\F_0$ and $\F\cap \fp(\G)$. On the other hand, in this generality, things are more subtle: to get a description that mirrors that of $\T_0$, one should consider the category $\underline \F$, which is itself a Grothendieck category (coreflective in $\G$) and, under strong enough hypotheses (see Lemma~\ref{useful_lema_segunda_parte}), one has $\F_0=\F\cap \fp(\underline \F)$; the naive guess that $\F_0=\F\cap \fp(\G)$ then follows when, in addition, $\F$ is generating, that is, $\G=\underline \F$.
\end{rem}

\section{Locally finitely presented hearts}  \label{lfp_section}

In this section we start with a locally finitely presented Grothendieck category $\G$, a torsion pair $\mathbf{t}=(\T,\F)$ of finite type in $\G$, and we look for conditions under which the heart $\Ht$  of the associated HRS $t$-structure in $\Der(\G)$ is a locally finitely presented Grothendieck category.
Our strongest result in this direction is the following theorem, that gives several equivalent conditions for this to happen. Let us underline that we are only able to prove the equivalence of all the conditions under either of two extra hypotheses that we have labeled here by ($\dag^\sharp$) (a stronger condition than the ($\dag$)  used in Proposition~\ref{char_of_fp_under_dot+restricts}) and ($\bullet$). In fact, these conditions are quite general as, for example, ($\dag^\sharp$) is always satisfied in case  $\G $ is locally coherent, while ($\bullet$) is always satisfied when $\G$ is a category of modules over a small preadditive category (e.g., over a  ring). 
\begin{teor}\label{thm.locally-fp-hearts}
Let $\G$ be a locally finitely presented Grothendieck category, $\mathbf{t}=(\T,\F)$ a torsion pair in $\G$ and $\Ht$ the heart of the associated HRS $t$-structure in $\Der(\G)$. Consider the following assertions:
  \begin{enumerate}[\rm (1)]
  \item $\H_\mathbf{t}$ is a locally finitely presented Grothendieck category;
  \item $\mathbf{t}$ is strongly  generated by finitely presented objects, i.e. $\mathcal{T}={\varinjlim} (\mathcal{T}\cap\fp(\G))$;
  \item there is a set $\mathcal{S}\subseteq\fp(\G)$ such that $\mathcal{T}=\Gen(\mathcal{S})$ (or, equivalently, $\mathcal{T}=\Pres(\mathcal{S})$);
  \item $\mathbf{t}$ is generated by a set of finitely presented objects, i.e.\ there is a set $\Scal\subseteq \fp(\G)$ such that $\F = \S^{\perp}$. 
  \end{enumerate}
Then, the implications {\rm``(1)$\Rightarrow$(2)$\Leftrightarrow$(3)$\Rightarrow$(4)''} hold true. Furthermore, the implication {\rm``(2)$\Rightarrow$(1)''} holds true whenever the following condition is satisfied:
\begin{enumerate}
\item[\rm ($\dag$)]  $\Ext_\G ^k(T,-)$ preserves direct limits of objects in $\mathcal{F}$, for all $T\in\mathcal{T}\cap\fp(\G)$ and $k=1,\,2$.
\end{enumerate}
Moreover, when either of the following two conditions is satisfied, all assertions are equivalent:
\begin{enumerate}
\item[\rm ($\dag^\sharp$)] condition $(\dag)$ holds and $\mathcal{T}\cap\fp(\mathcal{G})\subseteq\fp_2(\G)$;
\item[\rm ($\bullet$)]  $\G $ has a set of finitely presented generators which are compact in $\Der(\G )$.
\end{enumerate}
\end{teor}
\begin{proof}
The implications ``(2)$\Leftrightarrow$(3)$\Rightarrow$(4)''  follow by Proposition~\ref{prop_fixing_first_part_of_old_4.1}. 

\smallskip\noindent
(1)$\Rightarrow$(2). Let $T$ be an object in $\T$. Since $\Ht$ is locally finitely presented, there is a direct system $(P_\lambda)_{\Lambda}$  in $\fp(\Ht)$ such that ${\varinjlim}_{\Lambda}^{(\Ht)}P_{\lambda}=T[0]$. By \cite[Theorem~4.8(4)]{PS1} we deduce the isomorphism ${\varinjlim}_{\Lambda}^{(\G)} H^{0}(P_\lambda)=T$, so that assertion (2) is clear from Corollary~\ref{cor.fpobjects-in-heart}.

\smallskip\noindent 
{(2)$\Rightarrow$(1), {\em assuming  $(\dag)$}. Take $X\in\Ht$ and let us show that $X\in \Gen(\fp(\Ht))$. Let $T:=H^{0}(X)$ and fix a direct system $(T_{\lambda})_{\Lambda}$ in $\T\cap \fp(\G)$ such that ${\varinjlim}_\Lambda^{(\G)} T_\lambda\cong T$; this can be done by (2). For each $\lambda \in \Lambda$, take a pullback diagram in $\Ht$ as follows:
\[
\xymatrix@C=40pt@R=20pt{ 
0 \ar[r] & H^{-1}(X)[1] \ar[r] \ar@{=}[d] & X_{\lambda} \ar[r] \ar[d]_{\tilde{f_{\lambda}}} \ar@{}[dr]|{\text{P.B.}} & T_{\lambda}[0] \ar[r] \ar[d]^{f_\lambda} & 0\\ 
0 \ar[r] & H^{-1}(X)[1] \ar[r] & X \ar[r] & T[0] \ar[r] & 0.
}
\]  
We get a direct system $(X_{\lambda})_{\Lambda}\subseteq\Ht$ such that ${\varinjlim}_\Lambda^{(\Ht)}X_{\lambda}=X$, so there is no loss of generality in assuming that $T\in \fp(\G)\cap \T$. Let now $F:=H^{-1}(X)$ and take a direct system $(M_{\lambda})_{\Lambda}$ in $\fp(\G)$ such that ${\varinjlim}_\Lambda^{(\G)}M_{\lambda}\cong F$. By Lemma~\ref{preservation_colimits_fin_type}, $F\cong{\varinjlim}_\Lambda^{(\G)}(1:t)(M_{\lambda})$ and, by ($\dag$), we have an isomorphism:
\begin{equation}\label{epi_in_ext_2_eq}
{\varinjlim}_\Lambda \Ext^{2}_{\G}(T,(1:t)(M_{\lambda})) \tilde\longrightarrow \Ext^{2}_{\G}(T,{\varinjlim}_\Lambda (1:t)(M_{\lambda}))=\Ext^{2}_{\G}(T, F).
\end{equation}
%Write $X$ as a complex:
%\[
%\xymatrix{ &\cdots \ar[r] & 0 \ar[r] & X^{-1} \ar[r]^{d} & X^{0} \ar[r] & 0 \ar[r] & \cdots,
%}
%\]
Consider the element $[\varepsilon]\in\Ext_\G^2(T,F)$ represented by the following exact sequence:
\begin{equation}\label{equation_1_theorem3.1}
\xymatrix{
\epsilon: & 0 \ar[r] & F \ar[r] & X^{-1} \ar[r] & X^{0} \ar[r] & T \ar[r] & 0.
}
\end{equation}
By \eqref{epi_in_ext_2_eq}, there is $\beta \in \Lambda$ and $[\epsilon_{\beta}]\in \Ext^{2}_{\G}(T,(1:t)M_{\beta})$ such that $\Ext^{2}_{\G}(T,u_{\beta})([\epsilon_{\beta}])=[\epsilon]$, where $u_{\beta}\colon (1:t)M_{\beta} \rightarrow F$ is the canonical map to the direct limit. Fix the following exact sequence in $\G$:
\[
\xymatrix{
\epsilon_{\beta}\colon&0 \ar[r] & (1:t)M_{\beta} \ar[r] & Y_{\beta}^{-1} \ar[r] & Y_{\beta}^{0} \ar[r] & T \ar[r] & 0
}
\]
 which represents $[\epsilon_{\beta}]$ {and let $\Lambda_\beta:=\{\lambda\in\Lambda:\lambda\geq \beta\}$.} For each ${\lambda \in \Lambda_{\beta}}$,  take the  pushout diagram: 
\[
\xymatrix@R=20pt{
\epsilon_{\beta}\colon&0 \ar[r] & (1:t)M_{\beta} \ar[r] \ar[d]_{u_{{\beta ,\lambda}}} & Y^{-1}_{\beta} \ar[r] \ar[d]& Y^{0}_{\beta} \ar[r] \ar@{=}[d] & T \ar[r] \ar@{=}[d] & 0\\
&0 \ar[r] & (1:t)M_{\lambda} \ar[r] & Y^{-1}_{\lambda} \ar[r] \ar@{}[ul]|{P.O.} & Y^{0}_{{\beta}} \ar[r] & T \ar[r] & 0,
}
\]
where {$u_{\beta,\lambda}\colon(1:t)(M_\beta) \to (1:t)(M_{\lambda})$ is the transition map in $((1:t)(M_\lambda))_{\lambda \in \Lambda}$}. Thus, the second row in the diagram represents the extension  $\Ext^{2}_{\G}(T,{u_{{\beta ,\lambda}}})([\epsilon_{\beta}])$. 
Take the following direct system
\[
(0 \longrightarrow (1:t)M_{\lambda} \longrightarrow Y^{-1}_{\lambda} \longrightarrow Y^{0}_{\beta} \longrightarrow T \longrightarrow 0)_{\lambda \in \Lambda_\beta}
\] 
 in $\Ch(\G)$. As {$\Lambda_\beta\subseteq\Lambda$ is cofinal,} taking the colimit gives the following exact sequence:
\begin{equation}\label{equation_2_theorem3.1}
\xymatrix{
0 \ar[r] & F \ar[r] & {\varinjlim}_{\lambda \in \Lambda_\beta} Y^{-1}_{\lambda} \ar[r] & Y^{0}_{\beta} \ar[r] & T \ar[r] & 0
}
\end{equation}
which represents the same $[\epsilon]\in\Ext^{2}_{\G}(T,F)$. Then the original exact sequence \eqref{equation_1_theorem3.1} can be obtained from this one by a finite ``zig-zag'' of diagrams like the following one (see \cite{MacLane}):
\[
\xymatrix@R=20pt{
0 \ar[r] & F \ar[r] \ar@{=}[d] & \tilde{X}^{-1} \ar[r] \ar[d] & \tilde{X}^{0} \ar[r] \ar[d] & T \ar[r] \ar@{=}[d] & 0 \\ 
0 \ar[r] & F \ar[r] & \bar{X}^{-1} \ar[r] & \bar{X}^{0} \ar[r] & T \ar[r] & 0.
}
\]
The associated complexes concentrated in degrees $-1$ and $0$, all have the same cohomologies, and they are in fact  all quasi-isomorphic. Thus, \eqref{equation_2_theorem3.1} shows that $X$ is isomorphic in $\Ht$ to the complex:
\[
\xymatrix{ 
\cdots \ar[r] & 0 \ar[r] & {\varinjlim}_{\lambda \in \Lambda_\beta} Y^{-1}_{\lambda} \ar[r] & Y^{0}_{\beta} \ar[r] & 0 \ar[r] & \cdots.
}
\]
By \cite[Lemma~4.4]{PS1},  $X$ is a direct limit of $(Y_\lambda:=( \cdots \to 0 \to Y_\lambda^{-1} \to Y_{\beta}^0 \to 0 \to \cdots))_{\lambda \in \Lambda_\beta}$ in $\Ht$. Moreover, for each $\lambda \in \Lambda_\beta$, there is an exact sequences $0 \to H^{-1}(Y_\lambda)[1] \to Y_\lambda \to H^0(Y_\lambda)[0] \to 0$ in $\Ht$, where $H^{-1}(Y_\lambda)=(1:t)(M_\lambda) \in \F_0$, so $H^{-1}(Y_\lambda)[1] \in\fp(\H_\t)$ by Corollary~\ref{cor.F0}, and $H^0(X_\lambda)\in\fp(\G)$, so $Y_\lambda\in  \fp(\H_\t)$,  by Proposition~\ref{char_of_fp_under_dot+restricts}.}

\smallskip\noindent
(4)$\Rightarrow$(2), {\em assuming $(\dag^\sharp)$}. {It} follows by Lemma~\ref{ext_closed_fp_gives_torsion}.

\smallskip\noindent
(4)$\Rightarrow$(1), {\em assuming $(\bullet)$}. In this case, the associated HRS $t$-structure is compactly generated (for this, see \cite[Proposition~6.4]{SSV} and \cite[Theorem~2.3]{BPa}). One the concludes by one of the main results in \cite{Saorin-Stovicek}, which states that the heart of any compactly generated $t$-structure is a locally finitely presented.\end{proof}

Note that the assertions (2) and (4) in the above theorem are general statements about a torsion pair in a Grothendieck category so it seems plausible that, under suitable assumptions, there should be a more direct proof of their equivalence. Hence, the following question naturally arises:

\begin{ques}\label{ques_for_mod_gen_fp}
Suppose that $\G$ is a category of modules over a ring or, more generally, over a small pre-additive category, so that  $(\bullet)$ in Theorem~\ref{thm.locally-fp-hearts} is clearly verified. In this case, is there a purely module-theoretic proof of the implication {\rm``(4)$\Rightarrow$(2)''}? That is, can one find an argument that does not go through  $(1)$ and, hence, that does not rely on the heavy machinery of  \cite{Saorin-Stovicek}?
\end{ques}

\section{Locally coherent hearts}\label{Section_loc_coh_hearts}

In this section we start with a locally finitely presented Grothendieck category $\G$, a torsion pair $\mathbf{t}=(\T,\F)$ of finite type in $\G$, and we look for conditions under which the heart $\Ht$ of the associated HRS $t$-structure is locally coherent. This is a difficult problem in general so we restrict our investigation to two, fairly general, special cases: we  first tackle the problem under the additional hypothesis that $\bar\t:=(\F[1],\T[0])$ restricts to $\fp(\Ht)$ and, secondly, under the assumption that $\F$ is generating in $\G$.

\subsection{When the tilted torsion pair restricts to finitely presented objects}

A source of trouble in the study of the local coherence of the heart comes from the fact that the characterization in Corollary~\ref{cor.fpobjects-in-heart}  of the objects  $M\in\fp(\H_{\mathbf{t}})$ is not handy unless we can guarantee that  the stalks $H^{-1}(M)[1]$ and $H^0(M)[0]$ are also in $\fp(\H_\mathbf{t})$. This happens exactly when the tilted torsion pair $\bar{\mathbf{t}}=(\F[1],\T[0])$ restricts to  $\fp(\H_\mathbf{t})$, that is the situation that we shall deal with in this subsection. The following result builds on the characterizations of  local coherence from Propositions~\ref{prop. Manolo}  and \ref{prop.localcoherence-via-torsion}.

\begin{teor} \label{thm.local-coherent-heart-general}
Let $\mathcal{G}$ be a locally finitely presented Grothendieck category, $\mathbf{t}=(\mathcal{T},\mathcal{F})$ a torsion pair of finite type in $\mathcal{G}$, defined by the $2$-term cosilting complex 
\[
E:\qquad \cdots \longrightarrow 0\longrightarrow E^{-1}\stackrel{\sigma}{\longrightarrow}E^0\longrightarrow 0 \longrightarrow \cdots,
\] 
and $Q:=\Ker(\sigma)$ the corresponding quasi-cotilting object in $\mathcal{G}$. Then, the following are equivalent:
\begin{enumerate}[\rm (1)]
 \item $\H_\mathbf{t}$ is locally coherent and $\T_0=\T\cap \fp(\G)$;%,  for each $T\in\mathcal{T}\cap\fp(\G)$  and each direct system $(F_i)_{i\in I}$ in $\mathcal{F}$, the canonical morphism ${\varinjlim}\Ext_\G ^k(T,F_i)\to\Ext_\G ^k(T,{\varinjlim} F_i)$ is an isomorphism for $k=1$ and a monomorphism for $k=2$.
\item $\H_\mathbf{t}$ is locally coherent and the tilted torsion pair  $\bar{\mathbf{t}}=(\mathcal{F}[1],\mathcal{T}[0])$ restricts to $\fp(\H_\mathbf{t})$;
\item $\mathbf{t}$ is strongly generated by finitely presented objects and the following conditions hold:
\begin{enumerate}
\item[\rm ($\dag$)] $\Ext_\G^k(T,-)$ preserves direct limits of objects in $\F$, for all $T\in\T\cap\fp(\G)$ and $k=1,\, 2$;
\end{enumerate}
\begin{enumerate}[\rm ({3.}1)]
\item $\Ker(f)\in \F_0$ and $t(\Coker(f))\in \fp(\G)$, for all $f\colon F\to F'$ in $\F_0$;
\item $t(\Ker(g))\in \fp(\G)$, for all $g\colon T\to T'$ in $\T\cap \fp(\G)$;
\item $t(K)\in\fp(\G)$, for all $K\in \F_0*(\T\cap \fp(\G))$;
%\item $t(K)\in\fp(\G)$ if $K$ is a $1$-extension of a $T\in\T\cap \fp(\G)$ by an $F\in\F_0$;
%\item $\Ext_\G^k(T,-)\colon \G\to\Ab$ preserves direct limits of objects in $\F$, for $T\in\T\cap\fp(\G)$ and $k=1,2$;
\end{enumerate}
\item
$\mathbf{t}$ is strongly generated by finitely presented objects and the following conditions hold:
\begin{enumerate}[\rm ({4.}1)]
\item {\rm ($\dag$)} holds and the canonical map $\xi_{T[0]}\colon\Ext_\G^1(T,{\varinjlim}_{\Lambda} t(H^0(E_\lambda)))\to\Ext_\G^3(T,{\varinjlim}_{\Lambda} H^{-1}(E_\lambda))$  is injective,
 for all $(E_\lambda)_{\Lambda}\subseteq\Prod_{\Ch(\G)}(E)$ directed and $T\in\T\cap\fp(\G)$;
\item for some (resp., each) set of generators $\S\subseteq \fp(\G)$ of $\G$, let $\X:=(1:t)(\S)$, then:
\begin{enumerate}[\sc (i)]
\item $\Ext_\G^1(X,{\varinjlim}_{\Lambda} Q_\lambda)=0$, for all $X\in\mathcal{X}$ and $(Q_\lambda)_{\Lambda}\subseteq\Prod(Q)$ directed;
\item $\Ker(p)\in \F_0$, for all epimorphism $p\colon X\twoheadrightarrow F$, with $F\in\F_0$ and $X\in\mathrm{sum}(\X)$;
\item the canonical map $\xi_{X[1]}\colon\mathcal{G}(X,{\varinjlim}_{\Lambda} t(H^0(E_\lambda)))\to\Ext_\mathcal{G}^2(X,{\varinjlim}_{\Lambda}  H^{-1}(E_\lambda))$ is injective, for all $(E_\lambda)_{\Lambda}\subseteq\Prod_{\Ch(\G)}(E)$ directed and $X\in\mathcal{X}$.
\end{enumerate}
\end{enumerate}
\end{enumerate}
\end{teor}
\begin{proof}
By Lemma~\ref{lem.subcategories of Tcapfp} and Theorem~\ref{thm.locally-fp-hearts}, each of the conditions (2),  (3) or (4) implies $\mathcal{T}\cap\fp(\mathcal{G})=\T_0$, e.g., ``(2)$\Rightarrow$(1)'' follows  by Lemma~\ref{lem.subcategories of Tcapfp}. We shall assume that  $\mathcal{T}\cap\fp(\mathcal{G})=\T_0$  throughout the proof. 

\smallskip\noindent
(1)$\Rightarrow$(2). We want to verify that $\bar{\mathbf{t}}$ restricts to $\fp(\H_\mathbf{t})$, that is, given $M\in \fp(\Ht)$ we should show that $H^{-1}(M)[1]\in\fp(\Ht)$. As $\Ht$ is locally coherent by hypothesis, this is equivalent to show that the $\bar\t$-torsionfree part of $M$ is finitely presented, that is, $H^0(M)[0]\in\fp(\H_\mathbf{t})$ or, equivalently (see Corollary~\ref{cor.fp-torsion-stalks}), $H^0(M)\in\T_0$. But $H^0(M)\in \T\cap \fp(\G)$ by Corollary~\ref{cor.fpobjects-in-heart} and, in our case, $\T\cap \fp(\G)=\T_0$. 

\smallskip\noindent
(1,2)$\Leftrightarrow$(3). By Theorem~\ref{thm.locally-fp-hearts} and Lemma~\ref{char_of_fp_under_dot+restricts}, (3) forces the heart $\Ht$ to be locally finitely presented and $\bar{\mathbf{t}}=(\F[1],\T[0])$ to restrict to $\fp(\Ht)$. To conclude let us verify that, in this situation, the conditions (1--3) of Proposition~\ref{prop.localcoherence-via-torsion} are equivalent to (3.1--3) in the statement:
\begin{enumerate}[]
\item
(\ref{prop.localcoherence-via-torsion},(1))$\Leftrightarrow$(3.1). Any morphism in $\F_0[1]$ is of the form $f[1]\colon F[1]\to F'[1]$, for some morphism $f\colon F\to F'$ in $\F_0$. By the explicit construction of kernels in $\Ht$ (see the discussion at the end of Section~\ref{generalities_on_t}), we know that $Z:=\Ker_{\Ht}(f[1])$ is a complex of objects in $\F$, concentrated in degrees $-1$ and $0$, such that $H^{-1}(Z)=\Ker(f)$ and $H^0(Z)=t(\Coker(f))$. Now, the fact that $\bar{\mathbf{t}}$ restricts to $\fp (\Ht)$ tells us that $Z\in\fp (\Ht)$ if, and only if, $\Ker(f)[1]$ and $t(\Coker(f))[0]$ are in $\fp (\Ht)$. By Corollary~\ref{cor.F0} and \ref{cor.fp-torsion-stalks}, this happens exactly when condition (3.1) is verified. 
\item
(\ref{prop.localcoherence-via-torsion},(2))$\Leftrightarrow$(3.2). Similarly, a morphism in $(\T\cap \fp(\G))[0]$ is of the form $g[0]\colon T[0]\to T'[0]$, for some morphism $g\colon T\to T'$ in $\T\cap \fp(\G)$. Now, $\Ker_{\Ht}(g[0])\cong t(\Ker(g))[0]$ and so $\Ker_{\Ht}(g[0])\in\fp (\Ht)$ if, and only if, condition (3.2) is verified. 
\item
(\ref{prop.localcoherence-via-torsion},(3))$\Leftrightarrow$(3.3). Consider a morphism $h\colon T[0]\to F[1]$ in $\fp (\Ht)$, where $T[0]\in(\T\cap \fp(\G))[0]$ and $F[1]\in\F_0[1]$. This morphism is represented by an extension  in $\G$ of the form $0\rightarrow F\to K\to T\rightarrow 0$. By Lemma~\ref{lem.kernel of mixed morphism},  $\Ker_{\Ht}(h)\cong t(K)[0]$, so this object belongs in $\fp (\Ht)$ if, and only if, $t(K)\in\fp(\G)$. 

\end{enumerate}
 
\smallskip\noindent
{(2)$\Leftrightarrow$(4). By  Theorem~\ref{thm.locally-fp-hearts} and Lemma~\ref{char_of_fp_under_dot+restricts}, $\H_\mathbf{t}$ is locally finitely presented and $\bar{\mathbf{t}}$ restricts to $\fp(\H_\mathbf{t})$ under both sets of assumptions. Hence,  (2) holds if and only if {the following two inclusions hold (see Corollary~\ref{coro_1_fpn}(1) and Proposition~\ref{prop. Manolo}):}
\begin{enumerate}
\item[($*_{\T}$)] \mbox{$\T_0[0]{\ =(\T\cap\fp(\G))[0]}=\T[0]\cap\fp(\H_\mathbf{t})\subseteq\fp_2(\H_\mathbf{t})$};
\item[($*_{\F}$)] $\F_0[1]=\F[1]\cap\fp(\H_\mathbf{t})\subseteq\fp_2(\H_\mathbf{t})$.
\end{enumerate}
By the already proved equivalence of assertions (1), (2) and (3), condition $(\dagger)$ holds in both assertions (2) and (4). Then, by Proposition~\ref{prop.objects in fp2(H)}, {($*_\T$) is equivalent to (4.1)}. {Let us verify that ($*_\F$) is equivalent to (4.2). We first show that ``($*_\F$)$\Rightarrow$(4.2)'', with $\S\subseteq\fp(\G)$  an (arbitrary) set of generators:}  {given the exact sequence $0\to K\to X\stackrel{p}{\to} F\to 0$ associated to an epimorphism $p$ as in {\sc (ii)}, we have that $K\in \F$ and so we get an exact sequence $0 \to K[1] \to X[1] \to F[1] \to 0$ in $\H_\t$, where $X[1],\, F[1]\in \fp_2(\H_\t)$ by ($*_{\F}$). We then have $K[1]\in \fp(\Ht)$ by Lemma~\ref{lem.n-fp subcategories}, which implies that $K\in \F_0$ by Corollary~\ref{cor.F0}. {Condition  {\sc (i)} is also clear since $X[1]\in\fp_2(\Ht)$ and $\Ext_\G^1(X,-)_{\restriction\F}\cong\Ext_{\Ht}^1(X[1],-[1])_{\restriction\F}$ vanishes on $\varinjlim( \Prod(Q))$ due to Lemma \ref{char_of_2-term_lem}.} {Applying Proposition~\ref{prop.objects in fp2(H)} to $M=X[1]$, one shows that $\xi_{X[1]}$  is monic, for all $X\in\X$, so {\sc (iii)} holds too.}\\
We finally consider a set $\S\subseteq \fp(\G)$ of finitely presented generators of $\G$ for which condition (4.2) holds, where $\X=(1:t)(\S)$, and prove that $(*_\F)$ holds. By Lemma~\ref{useful_lema_segunda_parte}, we know that $\Ext^{1}_{\G}(X,-)_{\restriction \F}:\F \to \Ab$ preserves direct limits, for all $X\in \X$ ($\subseteq \F_0$). Applying Proposition~\ref{prop.objects in fp2(H)} with $M=X[1]$ and $X\in \X$, we get that $\X[1]\subseteq \fp_2(\H_{\t})$. If now $F\in \F_0$ and we consider the sequence $0 \to K[1]  \to X[1] \to F[1] \to 0$ in $\H_\t$ of the last paragraph associated to any epimorphism $p:X \to F$ as in {\sc (ii)}, then we get that $K[1]\in \F_0[1]\subseteq \fp(\H_\t)$ and, by Lemma~\ref{lem.n-fp subcategories}(1), we conclude that $F[1]\in \fp_2(\H_\t)$. Hence,  $(*_\F)$ holds.}}
\end{proof}

 \subsection{When the torsionfree class is generating }\label{subsection_needs this!}
Let $\G$ be locally finitely presented Grothendieck category $\G$. It was proved in \cite[Proposition~5.7]{PS1} that, for a torsion pair $\t=(\T,\F)$ for which $\F$ is a generating class, that is, $\underline \F=\G$, the following  are equivalent:
\begin{itemize}
\item the heart $\Ht$ of the associated Happel-Reiten-Smal\o\ $t$-structure is a Grothendieck category;
\item $\t$ is of finite type;
\item $\t$ is  cotilting.
\end{itemize}
Note that, if $Q\in \G$ is cotilting, for us $\F:=\Cogen(Q)$ is generating by definition (unlike in \cite{PS1}).

\begin{teor} \label{thm.main_thm_6c}
{Let $\G$ be a locally finitely presented Grothendieck category, $Q$ a  cotilting object in $\G$, and $\t= (\T,\F:=\Cogen(Q))$ the associated torsion pair. Then, the following  are equivalent:
\begin{enumerate}[\rm (1)]
\item $\mathcal{H}_\mathbf{t}$ is locally coherent;
\item for some (resp., each) set of generators $\mathcal{S}\subseteq \fp(\G)$ of $\mathcal{G}$, let $\X:=(1:t)(\S)$. Then:
  \begin{enumerate}[\sc (i)]
 \item $\Ext_\mathcal{G}^1(X,{\varinjlim}_I Q_i)=0$, for all $X\in\mathcal{X}$ and all direct systems $(Q_i)_{I}$ in $\Prod(Q)$;
 \item $\Ker(p)\in \F_0$, for all epimorphism $p\colon X\twoheadrightarrow F$, with $F\in\F_0$ and $X\in\mathrm{sum}(\mathcal{X})$;
\end{enumerate}
\item there is a set $\mathcal{X}\subseteq \F\cap \fp(\G)$ of generators of $\mathcal{G}$ that satisfies conditions {\sc (i)} and {\sc (ii)} in part (2);
\item  $\mathbf{t}$ restricts to $\fp(\G)$ and $\F\cap\fp(\G)\subseteq\fp_\infty(\G)$.
\end{enumerate}
When conditions (1--4) hold,  $\bar{\mathbf{t}}=(\F[1],\T[0])$ restricts to $\fp(\Ht)$ if, and only if, $\G$ is locally coherent. }
\end{teor}
\begin{proof}
{(1)$\Rightarrow$(2). Clearly, $\F[1]\cap\fp(\Ht)\subseteq\fp_2(\Ht)$, so $(*_\F)$ in the proof of Theorem~\ref{thm.local-coherent-heart-general} is verified and, as in that proof, one shows that $(*_\F)$ implies {\sc (i)} and  {\sc (ii)}, for all sets of generators $\S\subseteq \fp(\G)$.}

\smallskip\noindent
{(2)$\Rightarrow$(3)   Let $\mathcal{S}\subseteq\fp (\G)$ be a set of generators as in (2) and put  $\mathcal{X}:=\{(1:t)(S):S\in\mathcal{S}\}\subseteq \F_0$ (see Corollary~\ref{cor.F0}). Given $F_0\in \F_0$, there is $F\in \fp(\G)$ such that $F_0$ {is a direct summand of} $(1:t)(F)$ (see Corollary~\ref{cor.F0}). Then, by Proposition~\ref{fp_prop}, $F\in \mathrm{gen}(\S)$, so that $F_0\in \mathrm{gen}((1:t)\S)=\mathrm{gen}(\X)$. Now, using condition {\sc (i)}, one sees that $\F_0\subseteq \pres_2(\X)$. This, together with {\sc(ii)}, allows us to apply Lemma~\ref{useful_lema_segunda_parte} to the set $\X=(1:t)(\S)$ to see that $\mathcal{F}_0=\mathcal{F}\cap\fp(\mathcal{G})$, since $\underline{\F}=\G$ in our case. Then, $\mathcal{X}\subseteq \mathcal{F}\cap\fp(\mathcal{G})$ is the set of generators required in (3).}

%{(2)$\Rightarrow$(3). Let $\S$ be a skeleton of $\fp(\G)$, so that $\X:=(1:t)(\S)$ satisfies {\sc (i)} and  {\sc (ii)} by (2). Furthermore, $\F_0=\mathrm{add}(\X)$ by Corollary~\ref{cor.F0}, so $\F_0=\F\cap\fp_\infty(\underline{\F})\subseteq \F\cap\fp(\G)$ by Lemma~\ref{useful_lema_terza_parte} (as $\underline\F=\G$). } 

 \smallskip\noindent
 {(3)$\Rightarrow$(4). By Lemma~\ref{useful_lema_terza_parte}, $\F_0=\F\cap\fp(\G)\subseteq\fp_\infty(\G)$ (use that $\G=\underline \F$). Let us verify that $\mathbf{t}$ restricts to $\fp(\G)$: take $N\in\fp(\G)=\mathrm{pres}_1(\mathcal{X})$ (see Proposition~\ref{fp_prop}) and fix a presentation $ X'\to X\to N\to 0$, with $X',\, X \in \mathrm{sum}(\mathcal{X})$. Consider the following commutative diagram with exact rows and columns: 
\[
\xymatrix@C=35pt@R=13pt{
&X'\ar@{=}[r]\ar[d]&X'\ar[d]\\
0\ar[r]&F\ar@{}[dr]|{\text{P.B.}}\ar@{->>}[d]\ar[r]&X\ar@{->>}[d]\ar[r]&(1:t)(N)\ar@{=}[d]\ar[r]&0\\
0\ar[r]&t(N)\ar[r]&N\ar[r]&(1:t)(N)\ar[r]&0.
}
\] 
We deduce by {\sc (ii)} that $F\in\mathcal{F}_0$, so we get an exact sequence $X'\to F\to t(N)\rightarrow 0$, where $X',\, F\in\fp(\G)$ as $\F_0\subseteq\fp(\G)$. It follows that $t(N)\in\fp(\G)$ and hence $\mathbf{t}$ restricts to $\fp(\G)$. }
 
\smallskip\noindent
{(4)$\Rightarrow$(1).  By an unbounded analog of \cite[Proposition~3.2]{HRS}, $\F[1]$ is cogenerating in $\Ht$ and, furthermore, $\F[1]=\varinjlim^{(\H_\t)}(\F_0[1])$, since $\F=\varinjlim\F_0$ (see Corollary~\ref{cor.F0} and \cite[Proposition~4.2]{PS1}). Given $M\in \fp(\H_\t)$, there is an embedding $f\colon M\to F[1]$ with $F\in \F_0$ (as $\varinjlim^{(\H_\t)}(\F_0[1])$ cogenerates $\H_\t$ and $M\in\fp(\H_\t)$). Furthermore,  $F'[1]:=\Coker_{\H_\t}(f)\in  \F[1]\cap \fp(\H_\t)=\F_0[1]$ (as $\F[1]$ is a torsion class and by Corollary~\ref{coro_1_fpn}). We get an exact sequence $0\to M\to F[1]\to F'[1]\to 0$ in $\H_\t$, with $F,\, F'\in \F_0$. Given a direct system $(Q_i)_{I}$ in $\Prod(Q)$, we have that $\Ext_{\H_\t}^1(F[1],{\varinjlim}_I Q_i[1])=\Ext_{\G}^1(F,{\varinjlim}_I Q_i)=0$ and there is a monomorphism $\Ext_{\H_\t}^2(F'[1],{\varinjlim}_I Q_i[1])\to \Ext_{\G}^2(F',{\varinjlim}_I Q_i)=0$, by Lemma~\ref{useful_lema_terza_parte} and \cite[R\'emarque~3.1.17]{BBD}. Hence, we deduce from the following exact sequence that $\Ext_{\H_\t}^1(M,{\varinjlim}_I Q_i[1])=0$:
\[
\xymatrix@C=15pt{
\cdots\ar[r]&\Ext_{\H_\t}^1(F[1],{\varinjlim}_I Q_i[1])\ar[r]&\Ext_{\H_\t}^1(M,{\varinjlim}_I Q_i[1])\ar[r]&\Ext_{\H_\t}^2(F'[1],{\varinjlim}_I Q_i[1])\ar[r]&\cdots.
}
\]
Since $Q[1]$ is an injective cogenerator of $\H_\mathbf{t}$ ({see \cite[Theorem 3.11]{PS4}}), we conclude, by Lemma~\ref{lema_BP}, that $M\in\fp_2(\H_t)$. We have shown that $\fp(\H_\t)=\fp_2(\H_\t)$ so, to conclude, it is enough to verify that $\H_\t=\varinjlim \fp(\H_\t)$. We have already mentioned that $\F[1]=\varinjlim \F_0[1]$ is cogenerating and closed under quotients so, given $X\in \H_\t$, there is a short exact sequence $0\to X\to \varinjlim_IF_i[1]\to \varinjlim_jF'_j[1]\to 0$ in $\H_t$, with $F_i,\ F_j'\in\F_0$, for all $i\in I$ and $j\in J$. Using an argument which is analogous to the one used in the proof of Lazard's Trick, we
deduce that $X$ is a direct limit of objects which are kernels of (not
necessarily epic) morphisms in $\F_0[1]$. By the Abelian exactness of $\fp(\H_\t)=\fp_2(\H_\t)$ (see Corollary~\ref{coro_2_fpn}), those kernels
belong in $\fp(\H_\t)$, as desired.}

\medskip
{For the last statement: if $\G$ is locally coherent, then $\bar{t}=(\F[1],\T[0])$ restricts to $\fp(\H_\t)$ by Proposition~\ref{char_of_fp_under_dot+restricts}. Conversely, suppose that $\bar{t}=(\F[1],\T[0])$ restricts to $\fp(\H_\t)$, i.e., $\T\cap\fp(\G)\subseteq\T_0$ (see Theorem~\ref{thm.local-coherent-heart-general}). Since $\t$ restricts to $\fp(\G)$, to show that $\fp(\G)\subseteq \fp_2(\G)$ it is enough to show that $\T\cap\fp(\G)\subseteq \fp_2(\G)$ and $\F\cap\fp(\G)\subseteq \fp_2(\G)$, the latter being trivial as $\F\cap\fp(\G)\subseteq \fp_\infty(\G)$. Choose $\X$ as in part (3), and take $T\in\T\cap\fp(\G)$. Then there is a short exact sequence $0\rightarrow F\to X\to T\rightarrow 0$, with $X\in\mathrm{sum}(\X)$. Since clearly $F\in \F$ we get an exact sequence $0\to T[0]\to F[1]\to X[1]\to 0$ in $\Ht$, where $T[0],\, X[1]\in\fp(\Ht)$. We then get that $F[1]\in\fp(\Ht)$ so that, by Corollary~\ref{cor.F0} and assertion (4), we conclude that  $X,\, F\in \fp_\infty(\G)$ and so also $T\in \fp_{\infty}(\G)$}.
%{For the last statement: if $\G$ is locally coherent, then $\bar{t}=(\F[1],\T[0])$ restricts to $\fp(\H_\t)$ by Proposition~\ref{char_of_fp_under_dot+restricts}. Conversely, suppose that $\bar{t}=(\F[1],\T[0])$ restricts to $\fp(\H_\t)$, i.e., $\T\cap\fp(\G)\subseteq\T_0$ (see Theorem~\ref{thm.local-coherent-heart-general}). Since $\t$ restricts to $\fp(\G)$, to show that $\fp(\G)\subseteq \fp_2(\G)$ it is enough to show that $\T\cap\fp(\G)\subseteq \fp_2(\G)$ and $\F\cap\fp(\G)\subseteq \fp_2(\G)$, the latter being trivial as $\F\cap\fp(\G)\subseteq \fp_\infty(\G)$. Choose $\X$ as in part (3), take $T\in\T\cap\fp(\G)$, and an exact sequence $0\rightarrow F\to X\to T\rightarrow 0$, with $X\in\mathrm{sum}(\X)$, corresponding to the exact sequence $0\to T[0]\to F[1]\to X[1]\to 0$ in $\Ht$, with $T[0],\, X[1]\in\fp(\Ht)$. Thus, $F[1]\in\fp(\Ht)$, so $X,\, F\in \F_0=\F\cap\fp(\G)\subseteq\fp_\infty(\G)$, by Lemma~\ref{useful_lema_terza_parte}, and $T\in\fp_\infty(\G)\subseteq\fp_2(\G)$.}
\end{proof}

As a direct consequence of the above theorem, \cite[Proposition~3.4]{HRS} and its dual, we get:

\begin{cor}
The HRS tilting process defines a bijection, up to equivalence, between
\begin{enumerate}[\rm (1)]
\item pairs $(\G,\mathbf{t})$ such that $\G$ is a locally coherent Grothendieck category, $\mathbf{t}=(\T,\F)$ is a torsion pair of finite type that restricts to $\fp(\G)$ and $\T$ is  cogenerating  in $\G$;
\item pairs $(\G',\mathbf{t}')$ such that $\G'$ is a locally coherent Grothendieck category, $\mathbf{t}'=(\T',\F')$ is a torsion pair of finite type that restricts to $\fp(\G')$ and $\F'$ is  generating  in $\G'$.
\end{enumerate}
\end{cor}

\section{Some consequences and examples}\label{Sec_consequences}

The goal of this section is to exhibit special situations where either the characterization of local coherence of the heart becomes simpler or, at least, where some interesting necessary condition may be derived. Furthermore, we discuss several clarifying examples and, towards the end, we explore some connections with the theory of elementary cogenerators.

\subsection{Studying $\H_\t$ through the restricted torsion pair $\t'$ in $\underline\F$}
As usual, we start with  a locally finitely presented Grothendieck category $\G$ and  a torsion pair of finite type $\mathbf{t}=(\T,\F)$. We then use the results in Subsection~\ref{subsection_needs this!} to characterize the local coherence of the heart $\mathcal{H}_{\mathbf{t}'}$, associated with the ``restricted torsion pair'' $\t':=(\T\cap \underline \F,\F)$ in $\underline\F$.

{\begin{lema}\label{rem_tech_lemma}
Let $\mathcal{G}$ be a locally finitely  presented Grothendieck category and $\t=(\T,\mathcal{F}:=\Cogen(Q))$ a quasi-cotilting torsion pair induced by the quasi-cotilting object $Q\in\G$. Suppose that either the heart $\Ht$ or the ``restricted heart'' $\mathcal{H}_{\mathbf{t}'}$ is locally coherent, and let $\X\subseteq \F_0$ be any subset that generates $\F$, e.g. $\X=(1:t)(\S)$ for some set of generators $\S\subseteq \fp(\G)$. Then:
\begin{enumerate}[\sc (i)]
\item $\Ext_\mathcal{G}^1(X,{\varinjlim}_I Q_i)=0$, for all $X\in\mathcal{X}$ and all direct systems $(Q_i)_{I}$ in $\Prod(Q)$;
 \item $\Ker(p)\in \F_0$, for all epimorphisms $p\colon Y\twoheadrightarrow F$, with $F\in\F_0$ and $Y\in\mathrm{sum}(\mathcal{X})$.
\end{enumerate}
\end{lema}
\begin{proof}
Let $X\in\X$ and take a direct system $(Q_i)_{I}$ in $\Prod(Q)$. If $\H_\mathbf{t}$ is locally coherent, then:
\[
\Ext_\G^1(X,{\varinjlim}_I Q_i)\cong\Ext_{\H_\mathbf{t}}^1(X[1],{\varinjlim}^{(\H_\mathbf{t})}_I(Q_i[1]))\cong{\varinjlim}_I\Ext_{\H_\mathbf{t}}^1(X[1],Q_i[1])\cong{\varinjlim}_I\Ext_\G^1(X,Q_i)=0,
\] 
as $\F_0[1]\subseteq\fp(\H_\mathbf{t})=\fp_2(\H_\mathbf{t})$  and $\F\subseteq {}^{\perp_1}\Prod(Q)$. Analogously, if $\H_{\mathbf{t}'}$ is locally coherent, one gets that: 
\[
\Ext_{\underline \F}^1(X,{\varinjlim}_I Q_i)\cong\Ext_{\H_{\mathbf{t}'}}^1(X[1],{\varinjlim}^{(\H_{\mathbf{t}'})}_I(Q_i[1]))\cong{\varinjlim}_I\Ext_{\H_{\mathbf{t}'}}^1(X[1],Q_i[1])\cong{\varinjlim}_I\Ext_{\underline \F}^1(X,Q_i)=0.
\] 
Note that $\underline \F$ is closed under extensions both in $\G$ and in $\underline{\F}$,  so that $\Ext_{\G}^1(X,Q_i)=\Ext_{\underline \F}^1(X,Q_i)=0$. We have then verified {\sc (i)} in both cases. %, so $\underline{\F}$ is locally finitely presented by Lemma~\ref{useful_lema_segunda_parte}. 
Let now $p\colon Y\twoheadrightarrow F$ be an epimorphism, with $Y,\ F\in\F_0=\mathrm{sum}(\X)$, then $F':=\Ker(p)\in\F$ and we have an exact sequence $0\to F'[1]\to  Y[1]\to F[1]\to 0$ (both in $\Ht$ and in $\mathcal{H}_{\mathbf{t}'}$). By local coherence of $\Ht$ (resp., $\mathcal{H}_{\mathbf{t}'}$) and since $X[1],\, F[1]\in \fp(\Ht)$ (resp., $\in\fp(\H_{\t'})$), we get that $F'[1]\in \fp(\Ht)$ (resp., $\in\fp(\H_{\t'})$), so that $F'\in\F_0$, showing that also {\sc (ii)} holds.
\end{proof}}

{When $\Ht$ is locally coherent,  condition (3.1) of Theorem~\ref{thm.local-coherent-heart-general} does not hold in general, since the tilted torsion pair $\bar{\mathbf{t}}=(\F[1],\T[0])$ in $\Ht$ needs not restrict to $\fp (\Ht)$.  However, we have the following restricted version of the condition, that will be very helpful from now on. }

\begin{lema}\label{lem.(3.1) restricted}
Let $\mathcal{G}$ be a locally finitely  presented Grothendieck category, $Q\in\G$ and a quasi-cotilting object, and $\t=(\T,\mathcal{F}:=\Cogen(Q))$.
Suppose that $\X\subseteq\F_0$ generates $\F$ and satisfies conditions {\sc (i)} and {\sc (ii)} in Lemma~\ref{rem_tech_lemma}. Then, the following statements hold true:
\begin{enumerate}[\rm (1)]
\item $\F_0=\F\cap\fp (\underline{\F})=\F\cap\fp_\infty (\underline{\F})$  and it generates $\underline{\F}$;
%\item $\ker_\G(f)=\ker_{\underline F}(f)\in \F_0$ for all epimorphism $f\colon F\to F'$, with $F,\ F'\in \F_0$;
\item $\underline{\F}$ is locally finitely presented with $\fp (\underline{\F})=\pres_1(\X)=\pres_1(\F_0)$;
\item $\mathbf{t}':=(\T\cap\underline{\F},\F)$ restricts to $\fp(\underline\F)$.
\end{enumerate}
Furthermore, if our set $\X\subseteq \F_0$ also satisfies the following strengthening  of {\sc (ii)}:
\begin{enumerate}[\ \ ]
\item\begin{enumerate}[\rm (\textsc{ii}$^{\sharp}$)]
 \item $\Ker(p)\in \F_0$, for any (not necessarily epic) $p\colon Y\to F$, with $F\in\F_0$ and $Y\in\mathrm{sum}(\mathcal{X})$.
\end{enumerate}
\end{enumerate}
then, the following statement holds true: 
\begin{enumerate}[\rm (1)]\setcounter{enumi}{3}
\item $\underline{\F}$ is locally coherent.
\end{enumerate} 
\end{lema}
\begin{proof}
(1) follows by  Lemmas~\ref{F0generates_underlineF} and~\ref{useful_lema_terza_parte}, while (2) follows by Lemma~\ref{useful_lema_segunda_parte}  and Proposition~\ref{fp_prop}. At this point, we can apply Theorem~\ref{thm.main_thm_6c} to $\t'$ to get (3).

%\smallskip\noindent
%\textcolor{blue}{ 
%(3). except for the closure under kernels of epimorphisms, which follows by Corollary~\ref{coro_1_fpn}, as both $\F$ and $\fp_\infty (\underline{\F})$ have that property. }
%
%\smallskip\noindent
%\textcolor{blue}{(3) FOLLOWS BY THEOREM 7.12.
%\\
%. Let $C\in\fp(\underline \F)$ and, using (2), take a presentation $F_1\to F_0\to C\to 0$, with $F_0,\, F_1\in \F_0$. Consider now the pullback (which is also a pushout, since $F_0\to C$ is epic) of this presentation along the canonical inclusion $t(C)\to C$, to get  the following commutative diagram with exact rows and columns:
%\[
%\xymatrix@C=40pt@R20pt{
%&F_1\ar[d]\ar@{=}[r]&F_1\ar[d]\\
%0\ar[r]&P\ar@{}[dr]|{\text{\footnotesize{P.O. \  \ P.B.}}}\ar[r]\ar@{->>}[d]&F_0\ar[r]\ar@{->>}[d]&(1:t)(C)\ar@{=}[d]\ar[r]&0\\
%0\ar[r]&t(C)\ar[r]&C\ar[r]&(1:t)(C)\ar[r]&0
%}
%\]
%Since $C\in \fp(\underline \F)$, then $(1:t)(C)\in\F_0$ by Corollary~\ref{cor.F0}, so  $P\in \F_0$, as $\F_0$ is closed under kernels of epimorphisms (as both $\fp_\infty(\underline \F)$ and $\F$ have this property). In particular, $t(C)\in \mathrm{pres}_1(\F_0)=\fp(\underline \F)$.
%}

\smallskip\noindent
(4){, \em assuming {\sc (ii$^{\sharp}$)}}. For $C\in\fp(\underline \F)=\pres_1(\X)$, take an exact sequence $0\to K\to X_1\to X_0\to C\to 0$, with $X_0,\, X_1\in\mathrm{sum}(\X)$. Then, $K\in \F_0\subseteq\fp_\infty(\underline \F)$ by {\sc (ii$^{\sharp}$)}. To conclude, apply  Corollary~\ref{coro_1_fpn}(3) twice: first to show that $X_1/K\in\fp_\infty(\underline \F)$ and, secondly, that $C\in \fp_\infty(\underline \F)$ as well.
\end{proof}

\begin{cor} \label{cor.loc-coherence-implies-restriction}
Let $\G$ be a locally finitely presented Grothendieck category and $\mathbf{t}=(\T,\F)$ a torsion pair of finite type in $\G$. If  the heart $\Ht$ is locally coherent, then:
\begin{enumerate}[\rm (1)]
\item $\underline{\F}$ is  a locally finitely presented Grothendieck category, with   ${\F_0=}\F\cap\fp(\underline{\F})\subseteq\fp_\infty (\underline{\F})$;
\item $\mathbf{t}'=(\T\cap\underline{\F},\F)$ restricts to $\fp(\underline{\F})$;
\item $\mathcal{H}_{\mathbf{t}'}$ is locally coherent.
\end{enumerate}
\end{cor}
\begin{proof}
By Lemmas~\ref{rem_tech_lemma} and~\ref{lem.(3.1) restricted}, $\underline{\F}$ is locally finitely presented, with  ${\F_0=}\F\cap\fp(\underline{\F})\subseteq\fp_\infty (\underline{\F})$,  and  there exists a set $\X\subset\F_0$ that satisfies assertion (3)  of Theorem~ \ref{thm.main_thm_6c}. Applying this last theorem to $(\underline{\F},\mathbf{t}')$ instead of $(\G,\mathbf{t})$, we conclude that $\mathcal{H}_{\mathbf{t}'}$ is locally coherent and that $\mathbf{t}'$ restricts to $\fp(\underline{\F})$.
\end{proof}

{In what follows we relax the conditions of Theorem~\ref{thm.local-coherent-heart-general}, in exchange for the local coherence of $\G$. In this setting, our  strategy  is to look for  connections between the local coherence of $\Ht$ and that of $\mathcal{H}_{\mathbf{t}'}$.}

\begin{prop} \label{prop.after-referee-report}
Let $\G$ be a locally finitely presented Grothendieck category, $\mathbf{t}=(\mathcal{T},\mathcal{F})$ be a torsion pair  {generated by finitely presented objects (see Definition~\ref{def.torsion-pair-fp(strongly)generated}) and let 
$\mathbf{t}'=(\T\cap\underline{\F},\F)$ be the restricted torsion pair in the subcategory $\underline{\F}$. Consider the following assertions}:
\begin{enumerate}[\rm (1)]
\item $\mathcal{H}_\mathbf{t}$ is a locally coherent Grothendieck category;
\item the following hold for some (resp., every)  set of generators $\mathcal{S}\subseteq \fp(\G)$, where $\X:=(1:t)(\S)$:
\begin{enumerate}
\item[\sc (i)] $\Ext_\mathcal{G}^1(X,{\varinjlim}_I Q_i)=0$, for all $X\in\mathcal{X}$ and all direct systems $(Q_i)_{I}$ in $\Prod(Q)$;
\item[\sc (ii$^\sharp$)] $\Ker(p)\in \F_0$, for any (not necessarily epic) $p\colon Y\to F$, with $F\in\F_0$ and $Y\in\mathrm{sum}(\mathcal{X})$;
\item[\rm ($\dag^{\flat}$)] $\G(X,-)_{\restriction \T}:\T\to\Ab$ preserves direct limits, for all $X\in\mathcal{X}$;
\end{enumerate}
\item $\mathbf{t}$ satisfies the following conditions (see part $(3)$ of Theorem~\ref{thm.local-coherent-heart-general}):
\begin{enumerate}[\rm ({3.}1)]
\item $\Ker(f)\in \F_0$ and $t(\Coker(f))\in \fp(\G)$, for all $f\colon F\to F'$ in $\F_0$;
\item $t(\Ker(g))\in \fp(\G)$, for all $g\colon T\to T'$ in $\T\cap \fp(\G)$;
\item $t(K)\in\fp(\G)$, for all $K\in \F_0*(\T\cap \fp(\G))$;
%\item $\Ext_\G^k(T,-)\colon \G\to\Ab$ preserves direct limits of objects in $\F$, for $T\in\T\cap\fp(\G)$ and $k=1,2$;
\end{enumerate}
\item $\mathcal{H}_{\mathbf{t}'}$ is locally coherent and $(\dag^\flat)$ holds for some (resp., every) set of generators $\mathcal{S}\subseteq \fp(\G)$;
\item $\underline{\F}$ and $\mathcal{H}_{\mathbf{t}'}$ are locally coherent and $(\dag^\flat)$ holds for some (resp., every) set of generators $\mathcal{S}\subseteq \fp(\G)$;
\item $\mathbf{t}$ satisfies condition $(3.1)$;
\item $\underline{\F}$ is locally coherent, $\mathbf{t}'$ restricts to $\fp(\underline{\F})$, and $\T\cap\fp(\underline{\F})\subseteq\fp(\G)$. 
\end{enumerate} 
The implications {\rm``(1)$\Rightarrow$(4)$\Leftarrow$(5)$\Leftrightarrow$(2)$\Rightarrow$(7)$\Rightarrow$(6)$\Leftarrow$(3)''} hold true.  When $\G$ is locally coherent, the assertions $(1)$ to $(5)$ are all equivalent. If, in addition, $\mathbf{t}$ is hereditary, then all the assertions are equivalent.
\end{prop}
\begin{proof}
%\textcolor{green}{Whenever in the proof of an implication we can guarantee that, for some set  $\mathcal{X}\subset\F_0$ that generates $\F$, conditions (3.i') and (3.ii') of Thm. \ref{thm.main_thm_6c} hold,  we can apply  Lemmas \ref{F0generates_underlineF}, \ref{useful_lema_segunda_parte} and \ref{useful_lema_terza_parte}    to conclude that  $\F_0=\F\cap\fp (\underline{\F})=\F\cap\fp_\infty (\underline{\F})$ is a skeletally small class of  generators of $\underline{\F}$. In particular, we will get that $\F_0$ is closed under  kernels of epimorphism (in $\underline{\F}$ or $\G$) and that $\Ext_\G^1(F,-)_{| \F}=\Ext_{\underline{\F}}^1(F,-)_{| \F}:\F\to\Ab$ preserves direct limits, for all $F\in\F_0$. On the other hand, by Proposition~\ref{fp_prop},  we will get also that $\fp (\underline{\F})=\pres_1(\F_0)=\pres_1((1:t)(\mathcal{S}'))$, for any choice of a set $\mathcal{S}'$ of finitely presented generators of $\G$.} 
{(1)$\Rightarrow$(4). $\mathcal{H}_{\mathbf{t}'}$ is locally coherent  by Corollary~\ref{cor.loc-coherence-implies-restriction}. In particular,  for  $\mathcal{S}\subseteq \fp(\G)$ a set of generators, we have that $\G((1:t)(S),-)_{\restriction \T}\cong\Ext_{\Ht}^1((1:t)(S)[1],-[0])$, with $(1:t)(\mathcal{S})[1]\subseteq\fp (\Ht)=\fp_2(\Ht)$ due to Corollary~\ref{cor.F0}  and the local coherence of $\Ht$. Hence, $(\dag^{\flat})$ holds, since $\varinjlim^{\Ht}(T_i[0])\cong (\varinjlim T_i)[0]$, for any direct system $(T_i)_I$ in $\T$, by Lemma~\ref{rem. stalks}.}

\smallskip\noindent
{(5)$\Rightarrow$(4) is trivial. }%Take some set $\mathcal{S}$ as in assertion 3. As in the  proof of $(3)\Longrightarrow (4)$, we get that, for any choice $\mathcal{S}'$ of a set of finitely presented generators of $\G$,  the functor  $\G((1:t)(S'),-)_{| \T}:T\to\Ab$ preserves direct limits, for all $S'\in\mathcal{S}'$.  }

\smallskip\noindent
(2)$\Rightarrow$(5). Let $\mathcal{S}\subseteq\fp(\G)$ be the set of generators given in (2). Then, $\underline \F$ is locally coherent and $\F_0\subseteq \pres_1(\X)$, by Lemma~\ref{lem.(3.1) restricted}, so that $\H_{\t'}$ is locally coherent by Theorem~\ref{thm.main_thm_6c} applied to $\t'$. We conclude with the following observation:

\smallskip\noindent
[Obs.$(\dag^\flat)$]: suppose that $\F_0\subseteq \pres_1(\X)$ and $\X$ satisfies $(\dag^\flat)$. Then, given $F\in \F_0$, there is  an exact sequence $0\to\G(F,-)_{\restriction \T}\to\G(X_0,-)_{\restriction \T}\to \G(X_1,-)_{\restriction \T}$, with $X_0,\, X_1\in \X$. Thus, since both  $\G(X_0,-)_{\restriction \T}$ and $\G(X_1,-)_{\restriction \T}$ preserve direct limits by $(\dag^\flat)$, $\G(F,-)_{\restriction \T}$ is forced to do the same.

\smallskip\noindent
{(5)$\Rightarrow$(2). Let $\mathcal{S},\, \S'\subseteq\fp(\G)$ be the set of generators given in (5) and an arbitrary set of generators, respectively, $\X:=(1:t)(\S)$ and $\X':=(1:t)(\S')$. 
As $\H_{\t'}$ is locally coherent, $\S'$ satisfies {\sc (i)} and {\sc (ii)} by Lemma~\ref{rem_tech_lemma}, and any morphism $f\colon Y\to F$, with $Y\in \mathrm{sum}(\X')$ and $F\in \F_0$, is a morphism in $\F\cap \fp(\underline{\F})$ by Lemma~\ref{lem.(3.1) restricted}. It follows that $\Ker(f)\in \F \cap \fp(\underline{\F})=\F_0$ by the local coherence of $\underline{\F}$. Hence  {\sc (ii$^\sharp$)} holds.}

\smallskip\noindent
{(2)$\Rightarrow$(7). By Lemma~\ref{lem.(3.1) restricted} we have that $\underline{\F}$ is locally coherent, $\mathbf{t}'$ restricts to $\fp (\underline{\F})$, and $\fp (\underline{\F})=\text{pres}_1(\X)$. Thus, by $(\dag^\flat)$, $\G(T,-)_{\restriction\T}$ preserves direct limits, for all $T\in\T\cap\fp (\underline{\F})$. One concludes by Lemma~\ref{prop.fp torsion objects}.  }

\smallskip\noindent
{(3)$\Rightarrow$(6) is trivial. }

\smallskip\noindent
(7)$\Rightarrow$(6). Given  $f\colon F\to F'$ in $\F_0\subseteq \fp(\underline \F)$, we have that $\Ker(f),\, \Coker(f)\in \fp(\underline \F)$, as $\underline\F$ is locally coherent, so $\Ker(f)\in \F_0=\F\cap \fp(\underline \F)$, since $\F$ is closed under subobjects, and $t(\Coker(f))\in \T\cap \fp(\underline \F)$ since $\t'$ restricts to $\fp(\underline{\F})$. One concludes using that $ \T\cap \fp(\underline \F)\subseteq \fp(\G)$.

\smallskip
\begin{center}
\em In the rest of the proof we assume that $\G$ is locally coherent.
\end{center}

\smallskip\noindent
{(3)$\Rightarrow$(1). As $\G$ is locally coherent, condition ($\dag^\sharp$)  in Theorem~\ref{thm.locally-fp-hearts}  holds, so $\mathbf{t}$ is strongly generated by finitely presented objects. Hence, $\Ht$ is locally coherent by Theorem~\ref{thm.local-coherent-heart-general}.}

\smallskip\noindent
(4)$\Rightarrow$(3). Consider the classes $\C_\T$ and $\C_\F\subseteq \G$ of all the objects $X$ in $\G$ such that $\G(X,-)$ commutes with direct limits in $\T$ and in $\F$, respectively. Additionally, we denote by $\widehat{\C_{\F}}$ the class of all objects $X\in \G$ such that $\Ext^{k}_{\G}(X,-)_{\restriction \F}:\F \to \Ab$ preserve direct limits, for $k=0,1$. Observe that:
\begin{enumerate}[\rm ({Obs.}1)]
\item $\fp(\G)=\fp_2(\G)\subseteq \C_\F\cap \C_\T$ and $\C_\F\cap \F=\F_0$ (by definition of $\F_0$);
\item $\C_\T\cap \T=\fp(\G)\cap \T$ (by Lemma~\ref{prop.fp torsion objects}) and $\F_0\subseteq \C_\T$ (follows by [Obs.$(\dag^\flat)$]);
\item both $\C_\F$ and $\C_\T$ are closed under cokernels and under extensions by objects in $\fp_2(\G)$, that is, $\C_\F*\fp_2(\G)\subseteq\C_\F$ and $\C_\T*\fp_2(\G)\subseteq \C_\T$;
\item $(\fp(\G)\cap \underline{\F})*\F_0 \subseteq \widehat{\C_{\F}}$.
\end{enumerate}
Indeed $({\rm Obs.1\text{-}3})$  are clear and it follows  by Lemmas~\ref{rem_tech_lemma} and~\ref{lem.(3.1) restricted} that $\F_0\subseteq \widehat{\C_{\F}}$. Consider an exact sequence $0 \to X \to Y \to Z \to 0$, where $X\in \fp(\G)\cap \underline{\F}$ and $Z\in \F_0$. Clearly, $\G(Y,-)_{\restriction \F}\cong \Ext_{\G}^{0}(Y,-)_{\restriction \F}$ preserves direct limits. Moreover, we get the following exact sequence of additive functors $\F \to \Ab$: 
$$\underline{\F}(X,-)_{\restriction \F}  \longrightarrow \Ext^{1}_{\underline{\F}}(Z,-)_{\restriction \F} \longrightarrow \Ext^{1}_{\underline{\F}}(Y,-)_{\restriction \F} \longrightarrow \Ext^{1}_{\underline{\F}}(X,-)_{\restriction \F} \longrightarrow \Ext^{2}_{\underline{\F}}(Z,-)_{\restriction \F} $$
In this sequence the first, second and fifth functors preserve directs limits because $Z\in \F_0\subseteq  \fp_{\infty}(\underline{\F})$. To prove that the third one also does, and so $Y\in \widehat{\C_{\F}}$, it is enough to prove that $\Ext^{1}_{\underline{\F}}(X,-)_{\restriction \F} $ preserves direct limits. In fact, this is true, as $\Ext^{1}_{\underline{\F}}(X,-)_{\restriction \F} \cong \Ext^{1}_{\G}(X,-)_{\restriction \F}$ by the dual  of \cite[Lemma 5.11(1)]{PSV}, and $X\in \fp(\G)=\fp_2(\G)$. Hence, also ${\rm (Obs. 4)}$ is verified.% reduced to prove the inclusion $\fp(\G)\cap \underline{\F} \subseteq \widehat{\C_{\F}}$. 
\begin{enumerate}
\item[(3.1):] Let $f\colon F \to F'$ be a morphism in $\F_0$. By Lemmas~\ref{rem_tech_lemma} and~\ref{lem.(3.1) restricted}, we know that $\F_0=\F \cap \fp(\underline{\F})$ and that $\fp(\underline{\F})=\pres_1(\F_0),$ so that $C:=\Coker(f) \in \fp(\underline{\F}).$ But we know by Theorem~\ref{thm.main_thm_6c}, applied to $(\underline{\F},\t')$ instead of $(\G,\t)$, that $\t'$ restricts to $\fp(\underline{\F})$, so that $t(C)\in \fp(\underline{\F})=\pres_1(\F_0)\subseteq \C_{\T}$ by ({\rm Obs.2}) and ({\rm Obs.3}). Thus $t(C)\in \C_{\T} \cap \T\subseteq \fp(\G)$ by ({\rm Obs. 2}). Note also that $(1:t)(C)\in \F_0$ and so $C\in (\fp(\G)\cap \underline{\F})*\F_0\subseteq \widehat{\C_{\F}}.$ 
Putting $X:=\Im(f)$, we get a short exact sequence $0 \to X \to F' \to C \to 0$, that lives in $\underline{\F}$ and has $X\in \F$. Consider the following exact sequence of additive functors $\F \to \Ab$
$$0 \longrightarrow (C,-)_{\restriction \F}^{0} \longrightarrow (F',-)_{\restriction \F}^{0} \longrightarrow (X,-)_{\restriction \F}^{0} \longrightarrow (C,-)^1_{\restriction \F} \longrightarrow (F',-)^1_{\restriction \F} $$
where $(=,-)^k$ denotes the $k$-th Ext group either in $\G$ or in $\underline{\F}$, for $k=0,1$. Since the first, second, fourth and fifth functors in this sequence preserve direct limits (because $F',C\in \widehat{\C_{\F}}$) we conclude that $(X,-)^{0}_{\restriction \F}=\G(X,-)^{0}_{\restriction \F}$ preserves direct limits, so $X\in \F_0$. By Lemma~\ref{rem_tech_lemma}, with $\X$ a skeleton of $\F_0$,  the epimorphism $\overline{f}\colon F \to X$ has $\Ker(\overline{f})=\Ker(f)\in \F_0$, so (3.1) holds.     
%Consider the following exact sequence, where $F_0,\, F_1\in \F_0\subseteq \C_\T$, and $X\cong\mathrm{Im}(f)\cong \Ker(\gamma)$: 
%\[
%\xymatrix@C=40pt@R=0pt{
%0\ar[r]& K\ar[r]^\kappa& F_1\ar@{->>}[dr]_{\pi}\ar[rr]^f&& F_0\ar[r]^\gamma &C\ar[r]& 0.\\
%&&&X\ \ar@{^(->}[ur]_{\bar f}
%}
%\] 
%As $\t'$ restricts to $\fp(\underline \F)=\pres_1(\F_0)$, then $t(C)\in\pres_1(\F_0)\cap \T\subseteq \C_\T\cap \T\subseteq \fp(\G)=\fp_2(\G)$.
%Similarly, $(1:t)(C)\in \fp(\underline \F)\cap \F\subseteq \fp_\infty(\underline \F)$. Using that $\fp_2(\G)\cap\underline \F\subseteq \fp(\underline \F)$, we deduce that $C\in \fp_2(\underline \F)*\fp_\infty(\underline \F)\subseteq \fp_2(\underline \F)$. Hence, $X\in\fp(\underline \F)\cap \F=\F_0$, so $K=\Ker(\pi)\in \F_0$ by Lemma~\ref{rem_tech_lemma}.
\item[(3.2):] Consider the following exact sequence, where $T_0,\, T_1\in \T\cap\fp(\G)$, and $Y\cong\mathrm{Im}(g)$: 
\[
\xymatrix@C=45pt@R=-2pt{
0\ar[r]& K\ar[r]^\kappa& T_1\ar@/_+4pt/@{->>}[dr]_{\pi}\ar[rr]^g&& T_0.\\
&&&Y\ \ar@/_+4pt/@{^(->}[ur]
}
\] 
As $\G$ is locally coherent, $Y\in \fp(\G)\cap \T$, $K\cong \Ker(\pi)\in\fp(\G)=\fp_2(\G)$, and $(1:t)(K)\in \F_0\subseteq \C_\T$ (by Corollary~\ref{cor.F0}). Consider now the following pushout diagram:
\[
\xymatrix@C=45pt@R=18pt{ 
0 \ar[r] & K\ar@{}[dr]|{\text{P.O.}} \ar[r] \ar@{>>}[d] & T \ar[r]^\pi \ar@{>>}[d]^(.4)q & Y \ar@{=}[d] \ar[r] & 0 \\ 0 \ar[r] & (1:t)(K) \ar[r] & \tilde{T} \ar[r]  & Y \ar[r] & 0 
}
\]
By (Obs.3), $\tilde T\in (\C_\T*\fp_2(\G))\cap \T\subseteq \C_\T\cap \T\subseteq \fp(\G)=\fp_2(\G)$. Hence, $t(K)\cong\Ker(q)\in\fp(\G)$. 
\item[(3.3):] Given an exact sequence $0\to F\to K\to T\to 0$, with $F\in\F_0\subseteq \C_\F$ and $T\in \fp(\G)\cap \T\subseteq \fp_2(\G)$, we have that $K\in \C_\F*\fp_2(\G)\subseteq \C_\F$, so that $(1:t)(K)\in \C_\F\cap \F=\F_0$, by the natural isomorphism $\G(K,-)_{\restriction \F}\cong \F((1:t)(K),-)$ . Consider the pushout diagram: 
\begin{equation}\label{PO_3.3}
\xymatrix@C=45pt@R=18pt{
0 \ar[r] & F \ar[r]^-\iota \ar@{>>}[d]_(.4)f & K \ar[r]^-{\pi} \ar@{>>}[d]_(.4)p \ar@{}[dr]|-{\text{P.O.}} & T \ar@{>>}[d]^(.4)h \ar[r] & 0 \\ 
0 \ar[r] & \overline{f} \ar[r]_-{\bar\iota} & (1:t)(K) \ar[r]_-{\bar\pi} & \overline T\ar[r] & 0.
}
\end{equation}
Since $h$ is epic, $\overline T\in \T$. Furthermore, the morphism $\bar\iota\circ f\colon F\to (1:t)(K)$ is in $\F_0$ and so, by the already verified part (3.1), $\overline T=\Coker(\bar\iota)=\Coker(\bar\iota\circ f)\in \fp(\G)\cap \T$ (as $f$ is epic), $\Ker(f)=\Ker(\bar{\iota} \circ f)\in \F_0 \subseteq \C_\T$ and $\Ker(h)\in \fp(\G)=\fp_2(\G)$. To conclude, note that, by the Snake Lemma, we have a short exact sequence $0\to \Ker(f)\to t(K)\to \Ker(h)\to 0$, so that, $t(K)\in \T\cap (\C_\T*\fp_2(\G))=\T\cap\C_\T\subseteq \fp(\G)$.
\end{enumerate}

\smallskip\noindent
{(3-4)$\Rightarrow$(5). By Lemmas~\ref{rem_tech_lemma} and \ref{lem.(3.1) restricted}, (4) implies (5) if {\sc (ii$^\sharp$)} is verified, for some $\X\subseteq \F_0$ that generates $\F$. But (3.1) clearly implies {\sc (ii$^\sharp$)} for $\X$ a skeleton of $\F_0$.}

%Note that $T'\in\T$ since $h$ is an epimorphism, and that we have an isomorphism $T'\cong\Coker (p_{| F}:F\to (1:t)(K)$). By the verified condition (3.1) of  Theorem~\ref{thm.local-coherent-heart-general}, we get that $T'\in\T\cap\fp (\G)$ and then $\Ker (h)\in\fp (\G)$ due to the local coherence of $\G$.  On the other hand,   %Then this latter functor preserves  direct limits since, by applying Thm.\ref{thm.main_thm_6c} to $\mathbf{t}'$, we know that  $\F_0=\F\cap\fpmod\bar{\A}\subseteq\fp_\infty(\mod\bar{\A})$ . 
%we can then apply (EDL) to the sequence of functors $\F\to\Ab$
%\begin{align*}
%0=\G(T',-)_{\restriction\F}\to\F((1:t)(K),-)\to\F(\tilde{F},-)\to
%\Ext_\G^1(T',-)_{\restriction\F}\to\Ext_\G^1((1:t)(K),-)_{\restriction\F}
%\end{align*}
%to conclude that  $\F(\tilde{F},-)\colon \F\to\Ab$ preserves direct limits, and so $\tilde{F}\in\F_0$. This implies that $\Ker (f)\in\F_0$  by the verified condition 3.1 of  Theorem~\ref{thm.local-coherent-heart-general}.
%Consider now the exact sequence 
%\[
%0\to \Ker(f)\to t(K)\to\Ker (h)\to 0
%\] 
%given by the Snake Lemma. We can apply (EDL) to the associated sequence 
%\[
%0\rightarrow\G(\Ker (h),-)_{\restriction\T}\to\T(t(K),-)\to\G(\Ker(f),-)_{\restriction\T}\to\Ext_\G^1(\Ker(h),-)_{\restriction\T}
%\] 
%of functors $\T\to\Ab$ to conclude that $\T(t(K),-)\colon\T\to\Ab$ preserves direct limits and, by  Lemma~\ref{fp_torsion_lema},   condition (3.3) of Theorem~\ref{thm.local-coherent-heart-general} holds. 
%

\smallskip\noindent
(6)$\Rightarrow$(3), {\it when $\mathbf{t}$ is hereditary}.  Condition (3.2) is automatic in this case as, given $g\colon T\to T'$ in $\T\cap\fp (\G)$, we have that $\Ker(g)\in \T$ since $\t$ is hereditary, and $\Ker(g)\in \fp (\G)$ as $\G$ is locally coherent. To verify condition (3.3), consider an exact sequence $0\to F\to K\to T\to 0$ in $\G$, with $F\in\F_0$ and $T\in\T\cap\fp(\mathcal{G})=\T_0$. Consider the pushout diagram in \eqref{PO_3.3}: as $\t$ is hereditary, $f$ is an isomorphism (it was already epic and $\Ker(f)\in \T\cap \F=0$, since $\Ker(f)$ is a subobject of $t(K)$ and also of $F$) and, consequently, the square on the right-hand side of that diagram is also a pullback. Hence, $t(K)=\Ker(p)\cong \Ker(h)$ and, from the exact sequence $0\to F\to (1:t)(K)\to \overline T\to 0$ in $\G$ we obtain the exact sequence $0\rightarrow \overline T[0]\to F[1]\to (1:t)(K)[1]\rightarrow 0$ in $\Ht$. Furthermore, being $\mathbf{t}$ hereditary, we know that $h[0]\colon T[0]\to \overline T[0]$ is an epimorphism in $\Ht$. We then get an exact sequence $T[0]\to F[1]\to (1:t)(K)[1]\rightarrow 0$, where $T[0],\, F[1]\in\fp(\Ht)$. Hence, $(1:t)(K)[1]\in\fp (\Ht)$, that is, $(1:t)(K)\in\F_0$.  Applying condition (3.1)  to  $\bar \iota\circ f\colon F\to (1:t)(K)$, we conclude that $\overline T\in\T\cap\fp (\G)$. Hence, $ t(K)=\Ker (p)\cong \Ker(h)$ is in $\fp(\G)$, since $\G$ is locally coherent. 
\end{proof}

\begin{rem} \label{rem.substitute of local coherence of Htprima}
{One of the goals of the above proposition is to clarify the relation between the local coherences of $\Ht$ and $\mathcal{H}_{\mathbf{t}'}$. The reader not so interested in this connection, may replace the condition ``\,$\mathcal{H}_{\mathbf{t}'}$ is locally coherent''  in $(3)$ and $(5)$ by the equivalent condition: ``\,$\mathbf{t}'$ restricts to $\fp (\underline{\F})$ and $\F\cap\fp (\underline{\F})\subset\fp_\infty (\underline{\F})$'',  applying  Theorem~\ref{thm.main_thm_6c} to the ``restricted torsion pair ''$\mathbf{t}'$ in  $\underline{\F}$. }
\end{rem}

The following result complements \cite[Theorem~5.2]{Sa}.

\begin{cor} \label{cor.generalization-of-PJM}
Let $\G$ be a locally coherent Grothendieck category, $\mathbf{t}=(\T,\F)$ a torsion pair of finite type in $\G$ and $\mathcal{S}$ a set of finitely presented generators of $\G$. The following conditions are equivalent:
\begin{enumerate}[\rm (1)]
\item $\mathbf{t}$ restricts to $\fp(\G)$;
\item $\H_\mathbf{t}$ is locally coherent and $t(\S)\subseteq\fg(\G)$;
\item $\mathbf{t}'=(\T\cap\underline{\F},\F)$ restricts to $\fp(\underline{\F})$ and $t(\S)\subseteq\fg(\G)$.
\end{enumerate}
In such case the subcategory $\underline{\F}$ is locally coherent. 
\end{cor}
\begin{proof}
(1)$\Rightarrow$(2) is contained in \cite[Theorem 5.2]{Sa}. % If $\t$ restricts to $\fp(\G)$, then $t(\S)\subseteq\fp(\G)\subseteq \fg(\G)$ and $\F_0=\F\cap \fp(\G)$ (see Corollary~\ref{cor.F0}), so condition (3) of Proposition~\ref{prop.after-referee-report} is verified (by the local coherence of $\G$). Finally, note that ``(1)$\Leftrightarrow$(3)'' in Proposition~\ref{prop.after-referee-report} if $\G$ is locally coherent, so $\H_\t$ is locally coherent. 
 
\smallskip\noindent
(2)$\Rightarrow$(3) follows by Corollary~\ref{cor.loc-coherence-implies-restriction}.

\smallskip\noindent 
(3)$\Rightarrow$(1). Letting $\X:=(1:t)(\S)$, we have that $\X\subseteq \fp(\G)$, as $t(\S)\subseteq \fg(\G)$. By Proposition~\ref{fp_prop}, $\fp(\underline \F)=\pres_1(\X)= \fp(\G)\cap \underline \F$ (so $\fp(\underline \F)$ is an Abelian exact subcategory), where the inclusion $\pres_1(\X)\subseteq \fp(\G)\cap \underline{\F}$ holds since both $\fp(\G)$ and $\underline{\F}$ are closed under cokernels and $\mathrm{sum}(\X)\subseteq \fp(\G)\cap \underline{\F}$. In particular, $\F_0=\fp(\underline \F)\cap \F=\fp(\G)\cap \F$. It is now easy to verify that condition (2) of Proposition~\ref{prop.after-referee-report} holds for our $\S$ and, as $\G$ is locally coherent, we know that ``(2)$\Rightarrow$(1)'' in that proposition.
\end{proof}

Recall that a locally finitely presented Grothendieck category is {\bf locally Noetherian} when each finitely presented object is Noetherian, i.e., it has the ACC on subobjects. Equivalently, each subobject of a finitely presented object is finitely generated. Such a category is always locally coherent.
\begin{cor}
let $\G$ be a locally Noetherian Grothendieck category and $\mathbf{t}=(\T,\F)$ a torsion pair. Then $\mathbf{t}$ is of finite type if, and only if, $\Ht$ is a locally coherent Grothendieck category.  
\end{cor}
\begin{proof}
As $\mathbf{t}$ restricts to $\fp(\G)$, we can conclude by Corollary~\ref{cor.generalization-of-PJM} and \cite[Theorem~1.2]{PS2}. 	
\end{proof}

The following proposition suggests a strategy to construct a torsion pair of finite type $\t$ in a Grothendieck category $\G$, such that $\H_\mathbf{t}$ is locally coherent and $\t$ does not restrict to $\fp(\G)$: one should look inside the family of left constituents of TTF triples in $\G$. We postpone the exposition of explicit examples of (even locally coherent) Grothendieck categories admitting TTF triples of this kind to the end of the following subsection. As it turns out, such examples are quite numerous. 

\begin{prop} \label{ejems.left-constituent-TTFtriple}
Let $\G$ be a locally coherent Grothendieck category and $\mathbf{t}=(\T,\F)$ a hereditary torsion pair  which is the left constituent of a TTF triple in $\G$. Then, the heart $\Ht$ is locally coherent.
\end{prop}
\begin{proof}
Under our hypotheses, $\F=\underline{\F}$, so $\F$ is a Grothendieck category that sits as an Abelian exact subcategory in $\G$. Furthermore, as $\t$ is hereditary, $\F$ is closed under injective envelopes  in $\G$. In particular, the injectives in $\F$ are exactly the torsionfree injectives in $\G$. Therefore, the inclusion $\F\to \G$ sends injectives to injectives, showing that its left adjoint $(1:t)\colon \G\to\F$ is exact. Hence, $\F\cong \G/\Ker(1:t)\cong \G/\T$ is a Giraud subcategory of $\G$ (we refer the reader  to \cite{St} for all the needed terminology), so $\F$ is locally coherent by \cite[Theorem~11.1.33]{Prest}, since $\G$ is locally coherent and $\t$ is of finite type. To conclude, note that condition (7) of Proposition~\ref{prop.after-referee-report} is verified, since $\F=\underline{\F}$ is locally coherent and $\t':=(\T\cap \underline\F,\F)=(0,\F)$ is the trivial torsion pair. Furthermore, under our hypotheses, the implication ``(7)$\Rightarrow$(1)'' in Proposition~\ref{prop.after-referee-report} holds.
\end{proof}

\subsection{When $\G$ is a category of modules}
Starting from this subsection, we are now going to specialize some of the results about the local coherence of the heart of an HRS $t$-structure to the case when the ambient category is a category of modules over a preadditive category. Let us start fixing some notation:
\begin{itemize}
\item[$\bullet$] $\A$ and $R$ denote, respectively, a small (pre)additive category and a unitary ring;
\item[$\bullet$] $\A(a,b)$ is the group of morphisms $a\to b$ in the category $\A$;
\item[$\bullet$] $\mod \A$ is the category of unitary right $\A$-modules{, i.e.\ of additive functors $\A^{\op}\to\Ab$};
\item[$\bullet$] $\Hom_\A(M,N):=(\mod\A)(M,N)$ and $\End_\A(M):=(\mod \A)(M,M)$;
\item[$\bullet$] $\fg(\A):=\fg(\mod \A)$, $\fpmod \A:=\fp(\mod \A)$, and $\fp_n(\A):=\fp_n(\mod \A)$, for  $n\in \mathbb N_{\geq2}\cup\{\infty\}$;
\item[$\bullet$] $\proj\A$ is the class of finitely generated projective $\A$-modules;
\item[$\bullet$] $\Flats \A$ is the class of flat $\A$-modules;
\item[$\bullet$] $\Der(\A):=\Der(\mod \A)$ denotes the unbounded derived category of $\mod \A$.
\end{itemize}

Given a small preadditive category $\A$ and $a\in \A$, we let $H_a:=\A(-,a)\colon \A^{\op}\to \Ab$ be the right $\A$-module represented by $a$. The family 
\begin{equation} \label{eq.XXIII}
\P:=\{H_a:a\in \A\}
\end{equation}
is a family of finitely generated projective generators of $\A$, so $\proj(\A)=\add(\P)$. Furthermore, for a given torsion pair $\t=(\T,\F)$ in $\mod \A$:
\begin{itemize}
\item $t(\A)$ denotes the torsion ideal of $\A$ (see Example~\ref{example_modules_underline});
\item $\bar \A:=\A/t(\A)$ denote the preadditive quotient category of $\A$ over the ideal $t(\A)$.
\end{itemize}

\begin{cor}\label{first_coro_for_modules}
Let $\A$ be a small preadditive category and $\mathbf{t}$ a torsion pair of finite type in $\mod\A$, whose heart $\Ht$ is a locally coherent Grothendieck category. Then, $\mathbf{t}':=(\T\cap\mod{\bar \A},\F)$ is a torsion pair in $\mod{\bar \A}$ that restricts to $\fpmod{\bar \A}$ and such that $\F\cap\fpmod{\bar \A}\subseteq\fp_\infty({\bar \A})$. 
\end{cor}
\begin{proof}
Note that $\underline{\F}=\mod{\bar \A}$ (see Example~\ref{example_modules_underline}) and apply Corollary~\ref{cor.loc-coherence-implies-restriction}.
\end{proof}

As for the general case, when $\F$ is generating, we can improve the above corollary:

\begin{cor} \label{cor.cotiltingpair-locallycoh-heart-modcats}
Let $\mathcal{A}$ be a small preadditive category and $\mathbf{t}=(\mathcal{T},\mathcal{F})$ a torsion pair of finite type such that $\mathcal{F}$ is generating in $\mod \A$. The following assertions are equivalent:
\begin{enumerate}[\rm (1)]
\item $\H_\mathbf{t}$ is locally coherent;
\item $\Ker(p)\in\F_0$ for each epimorphism $p\colon P\twoheadrightarrow F$, with $F\in\mathcal{F}_0$ and $P\in\mathrm{sum}(\P)$;
\item the torsion pair $\mathbf{t}$ restricts to $\fpmod \A$ and $\F\cap\fpmod\A\subseteq\fp_\infty(\A)$.
\end{enumerate}
\end{cor}
\begin{proof}
It follows by the equivalences ``(1)$\Leftrightarrow$(3)$\Leftrightarrow$(4)'' of Theorem~\ref{thm.main_thm_6c}, by taking $\mathcal{X}:=\P$ in part (3) of Theorem~\ref{thm.main_thm_6c}, and observing that condition {\sc (i)} is trivially satisfied by $\P$. 
\end{proof}

Recall that if $f\colon X\to Y$ is a morphism in an additive (not just preadditive) category $\A$, then a {\bf pseudo-kernel} of $f$ is a morphism $g:K\to X$ such that the sequence of functors 
\[
\xymatrix@C=40pt{
\A(-,K)\ar[r]^-{\A(-,g)}&\A(-,X)\ar[r]^-{\A(-,f)}&\A(-,Y)
}
\] 
is exact. When every morphism has a pseudo-kernel we say that $\A$ {\bf has pseudo-kernels}. In fact, one can prove that the locally finitely presented category $\mod \A$ is locally coherent if and only if $\A$ has pseudo-kernels. Specializing to the case when $\A=R$ is a unitary ring, one can consider the Morita equivalence $\mod R\cong \mod {(\proj R)}$, then $R$ is a right coherent ring if and only if the additive $\proj R$ has pseudo-kernels, if and only if $\mod R$ is locally coherent. Similarly, for a general small preadditive category $\A$,  $\mod \A$ is locally coherent if and only if the additive category $\proj \A$  has pseudo-kernels, {which is equivalent to say that $\mathrm{sum}(\P)$ has pseudo-kernels (see \cite[Corollary~1.11]{PSV2}).}

\begin{cor} \label{cor.coherent ring}
Let $\A$ be a small preadditive category such that $\mod\A$ is locally coherent and let $\mathbf{t}=(\T,\F)$ be a torsion pair of finite type in $\mod\A$. The following assertions are equivalent:
\begin{enumerate}[\rm (1)]
\item $\mathbf{t}$ restricts to $\fpmod\A$;
\item $t(H_a)\in\fg(\A)$, for all $a\in\A$, and the heart $\H_\mathbf{t}$ is locally coherent;
\item  $t(H_a)\in\fg(\A)$, for all $a\in\A$, and $\mathbf{t}'=(\T\cap\mod{\bar \A},\F)$ restricts to $\fpmod{\bar \A}$. 
\end{enumerate}

In this case $\mod{\bar \A}$ is also locally coherent.
\end{cor}
\begin{proof}
It follows directly from Corollary~\ref{cor.generalization-of-PJM}, by taking $\S:=\P$.
\end{proof}

Our next two results will be immediate consequences of Proposition~\ref{prop.after-referee-report}. But we first need to introduce some concept familiar to module theorists.

\begin{defi}
Let $\mathcal{I}$ be an ideal of the small preadditive category $\mathcal{A}$ and let $M$ be an $\mathcal{A}$-module. The {\bf annihilator of $\mathcal{I}$ in $M$}, denoted by
$\ann_M(\mathcal{I})$,  is the $\mathcal{A}$-submodule of $M$ acting on objects as follows:
\[
(\ann_M(\mathcal{I}))(a):=\bigcap\{\Ker(M(\alpha )): b\in\mathcal{A},\,\alpha\in\mathcal{I}(b,a)\}.
\] 
\end{defi}
It is routine to check that the assignment $M\mapsto\ann_M(\mathcal{I})$ induces an additive functor 
\[
\ann_{(-)}(\mathcal{I})\colon\mod \mathcal{A}\longrightarrow\mod{{\mathcal{A}}/{\mathcal{I}}}\subseteq \mod \mathcal{A}.
\] 
%Actually, if $\pi_I\colon \mathcal{A}\to\bar{\mathcal{A}}:={\mathcal{A}}/{\mathcal{I}}$ is the obvious projection, then $\ann_{(-)}(\mathcal{I})$ is the right adjoint to the restriction of scalars $(\pi_I)_*\colon \mod{\bar{\mathcal{A}}}\to\mod \mathcal{A}$. 
When $\mathcal{A}=R$ is a ring, viewed as a preadditive category with just one object, $I$ is an ideal of $R$ and $M$ is an $R$-module, then we re-obtain the classical annihilator $\ann_M(I)=\{x\in M : xI=0\}$.

\begin{cor} \label{cor.after-referee-report}
{Let $\A$ be a preadditive category and $\mathbf{t}=(\T,\F)$ a torsion pair  in $\mod \A$. Consider the following assertions: }
\begin{enumerate}[\rm (1)]
\item {the heart $\Ht$ is a locally coherent Grothendieck category;}
\item {the following conditions hold: }
\begin{enumerate}[\rm({2.}1)]
\item {$\mathbf{t}$ is generated by finitely presented $\A$-modules; }
\item {$\mod \bar{\A}$ and the restricted heart $\mathcal{H}_{\mathbf{t}'}$ are locally coherent; }
\item {the functor $\mathrm{ann}_{t(\A)}(-):\mod\A\to\mod\A$ preserves direct limits of modules in $\T$; }
\end{enumerate}
\item {conditions $(2.1)$ and $(2.3)$ hold and $\Ker(f)\in\F_0$, for all  $f\colon P\to F$, with $P\in\mathrm{sum}(\P)$ and $F\in\F_0$; }
\item {conditions $(2.1)$ and $(2.3)$ hold and  $\mathcal{H}_{\mathbf{t}'}$ is locally coherent.}
\end{enumerate}
{Then the implications {\rm``(1)$\Rightarrow$(4)$\Leftarrow$(3)$\Leftrightarrow$(2)''} hold true. When $\mathrm{sum} (\P)$ (or $\proj\A$) has pseudo-kernels (i.e., when $\mod\A$ is locally coherent), all assertions are equivalent. }
\end{cor}
\begin{proof}
 For each $a\in\mathcal{A}$ and each $\mathcal{A}$-module $M$, the canonical projection 
$
p\colon H_a\twoheadrightarrow (1:t)(H_a)
$
induces the following monomorphism in $\Ab$:
\[
p_*\colon\hom_\mathcal{A}\left((1:t)(H_a),M\right)\longrightarrow \hom_\mathcal{A}\left(H_a,M\right)\cong M(a).
\] 
 It is easy to see that $\Im(p_*)=(\ann_{t(\mathcal{A})}(M))(a)$.  Therefore, condition (2.3) above is equivalent to say that, for each $a\in\mathcal{A}$, the functor $\hom_\A((1:t)H_a,-)\colon \mod \mathcal{A}\to\Ab$ preserves direct limits of objects in $\mathcal{T}$, which is precisely condition $(\dag^\flat)$ in Proposition~ \ref{prop.after-referee-report}, when we take $\S:=\P=\{H_a\text{: }a\in\A\}$. Moreover, the implication ``(1)$\Rightarrow$(2.1)'' follows by Theorem~\ref{thm.locally-fp-hearts} (since the hypothesis $(\bullet)$ is verified). Note then that the conditions (1), (2), (3) and (4) correspond, respectively, to the conditions (1), (5), (2) and (4) of Proposition~ \ref{prop.after-referee-report}, by taking  $\S=\P$.
\end{proof}

\begin{cor} \label{cor.hereditary-coherent-heart for modcats}
Let $\A$ be a small preadditive category such that $\mod \A$ is locally coherent and let $\mathbf{t}=(\T,\F)$ be a hereditary torsion pair of finite type in $\mod\A$. The following assertions are equivalent: 
\begin{enumerate}[\rm (1)]
\item $\H_\mathbf{t}$ is locally coherent;
\item $\mathbf{t}$ satisfies condition (3.1) of Theorem~\ref{thm.local-coherent-heart-general};
\item $\mod{\bar\A}$ is locally coherent, $\mathbf{t}'=(\T\cap\mod{\bar\A},\F)$ restricts to $\fpmod{\bar\A}$ and $\T\cap\fpmod{\bar\A}\subseteq\fpmod\A$.
\end{enumerate}
\end{cor}
\begin{proof}
Since $\t$ is hereditary of finite type, it is generated by finitely presented objects (see Proposition~\ref{prop_fixing_first_part_of_old_4.1}). The result then follows by the equivalences ``(1)$\Leftrightarrow$(6)$\Leftrightarrow$(7)''  of Proposition~\ref{prop.after-referee-report}, that do hold since $\G=\mod \A$ is locally coherent and $\t$ is hereditary.
\end{proof}

\begin{rem} \label{rem.coherent ring}
When $ R\cong \A$ is a ring, viewed as a preadditive category with just one object, say $\Ob(\A)=\{a\}$ and $R=\A(a,a)$, then $t(H_a)=t(R)=t(\A)$ is the torsion ideal of $R$ and $\bar{\A}\cong R/t(R)$. So, for instance, condition (3) of  Corollary~\ref{cor.hereditary-coherent-heart for modcats} reads as follows:
\begin{enumerate}[\rm (3$_{0}$)]
\item $R/t(R)$ is a right coherent ring, $\mathbf{t}'$ restricts to $\fpmod (R/t(R))$ and $\T\cap\fpmod (R/t(R))\subseteq\fpmod R$.
\end{enumerate}
We leave to the reader the interpretation of the results of this subsection in this particular case.
\end{rem}

Recall from \cite{krause2015krull}, that an additive category {with split idempotents} is called {\bf Krull-Schmidt} if every object decomposes into a finite direct sum of objects having local endomorphism rings. The following lemma is key for the construction of most of the examples we discuss in the rest of the subsection.

\begin{lema}\label{ex_2_modules_idempotents}
Let $\A$ be a small additive category with pseudo-kernels, let $\mathcal{B}\subseteq\A$ be a subcategory, $\mathcal{P}_\B:=\{H_b:b\in\mathcal{B}\}\subseteq \P$, and $\mathbf{t}=(\T_\B,\F_\B):=(\Gen(\mathcal{P}_\B),\mathcal{P}_\B^\perp)$, the associated torsion pair in $\mod\A$, that is well-known to be the left constituent of a TTF triple (see \cite{PSV2}). The following assertions hold:
\begin{enumerate}[\rm (1)]
\item $\mathbf{t}$ restricts to $\fpmod\A$ if, and only if,  $\mathrm{tr}_{\mathcal{P}_B}(H_a)$ is finitely generated, for all $a\in\A$, if and only if $\mathrm{sum}(\mathcal{B})$ is precovering in $\A$. When $\mathcal{B}=\{b_1,\dots,b_n\}$ is finite, this is equivalent to say that $\A(\coprod_{i=1}^nb_i,a)\cong\bigoplus_{i=1}^n\A(b_i,a)$ is finitely generated as a right $\End_\A(\coprod_{i=1}^nb_i)$-module, for all $a\in\A$;
\item the following assertions are equivalent:
\begin{enumerate}[\rm (2.1)]
\item $\mathbf{t}$ is a hereditary torsion pair;
\item for each morphism $\alpha \colon a\to b$ in $\A$, where $b\in\mathcal{B}$, there are morphisms $\beta_i\colon b_i\to a$ and $\gamma_i\colon a\to b_i$, for some family $\{b_1,\dots,b_n\}$ of objects of $\mathcal{B}$, such that $\alpha =\alpha\circ (\sum_{i=1}^n\beta_i\circ\gamma_i)$.
\end{enumerate}
When, in addition, $\A$ is a Krull-Schmidt category, they are also equivalent to 
\begin{enumerate}[\rm (2.3)]
 \item if $\alpha\colon a\to b$ is a non-zero morphism with $b\in\B$ and $a\in \A$ indecomposable, $a\in{\mathrm{add}}(\B)$.
\end{enumerate}
\end{enumerate}
\end{lema}
\begin{proof}
By \cite{PSV2}, the torsion radical associated with $\mathbf{t}$ is the trace of $\mathcal{P}_\B$, i.e. 
\begin{equation}\label{tor_rad_formula}
\xymatrix{
t(M)=\tr_{\mathcal{P}_\B}(M):=\sum_{b\in\mathcal{B}, f\in \hom_\A(H_b,M)}\Im (f).
}
\end{equation}

\noindent
(1). As $\P_\B\subseteq\proj\A$, we readily see that $t$ preserves epimorphisms. Furthermore, since each $X\in\fpmod\A$ is epimorphic image of $\coprod_{i=1}^nH_{a_i}$, for some family ${\{}a_1,\dots,a_n{\}}$ of objects of $\A$, we conclude that $\mathbf{t}$ restricts to $\fpmod\A$ if, and only if,  $t(H_a)={\tr_{\mathcal{P}_\B}(H_a)}$ is  finitely generated for each $a\in\A$. By \eqref{tor_rad_formula}, this happens if and only if there exists an epimorphism $q\colon \coprod_{k=1}^m H_{b_k}\twoheadrightarrow t(H_a)$, with $b_1,\dots,b_m\in \B$. By the Yoneda Lemma, we have a representation
\[
(\xymatrix@C=43pt{
\coprod_{k=1}^mH_{b_k}\ar[r]^-{q}&t(H_a)\ar@{^(->}[r]^-{\mathrm{incl}}&H_a})\ =\ (\xymatrix@C=43pt{
\coprod_{k=1}^mH_{b_k}\ar[r]^-{\A(-,\beta)}&H_a}),
\] 
for some $\beta\colon \coprod_{k=1}^mb_k\to a$. Then, $\beta$ is a $\mathrm{sum}(\mathcal{B})$-precover in $\A$ if and only if $q$ is an epimorphism.
%Conversely, if $\beta\colon \coprod_{k=1}^mb_k\to a$ is a  $\mathrm{sum}(\mathcal{B})$-precover, one easily sees that the image of  
%\[
%\xymatrix{
%\A(-,\beta)\colon \A(-,\coprod_{k=1}^mb_k)\cong\coprod_{k=1}^m\A(-,b_k)\longrightarrow\A(-,a)
%}
%\] 
%is $t(\A(-,a))$, which is then finitely generated.

\smallskip\noindent
(2). As in the case of module categories over an (associative unital) ring $R$, one can see that $\mathbf{t}$ is a hereditary torsion pair if, and only if, each ``cyclic'' submodule of $H_b$ is in $\Gen(\mathcal{P}_\B)=\T_\B$, for all $b\in\mathcal{B}$. That is, if and only if the image of any morphism $H_a\to H_b$, with $a\in\A$ and $b\in\mathcal{B}$, is in $\T_\B$.  Namely, one first reduces the problem to check that any submodule of a coproduct $\coprod_{I}H_{b_i}$, with $b_i\in\mathcal{B}$ for all $i\in I$, is in $\T_\B$. {But if $K\leq \coprod_I H_{b_i}$ is a submodule then the (Ab.5) condition of $\mod{\A}$ gives that $K=\bigcup_{J\subseteq I, {} J \text{ finite}}(K\cap (\coprod_J H_{b_i}))$ (see \cite[Chapter~5]{St}). Then,} the problem further reduces to prove that any finitely generated submodule of a finite coproduct  $\coprod_{i=1}^nH_{b_i}$ is in $\T_\B$. Finally, one inductively reduces the  problem to check  that the cyclic submodules of $H_b$ are in $\T_\B$, for all $b\in\mathcal{B}$.

\smallskip\noindent
(2.2)$\Rightarrow$(2.1). Any morphism  $H_a\to H_b$, with $a\in\A$ and $b\in\B$, is of the form $\A(-,\alpha)$, for some morphism $\alpha\colon a\to b$ in $\A$. Fix  $\beta_i\colon b_i\to a$ and $\gamma_i\colon a\to b_i$, for some family $\{b_1,\dots,b_n\}$ of objects of $\mathcal{B}$, such that $\alpha =\alpha\circ (\sum_{i=1}^n\beta_i\circ\gamma_i)$. Then $\Im(\A(-,\alpha))=\sum_{i=1}^n\Im(\A(-,\alpha\circ\beta_i))$, so the implication holds.  

\smallskip\noindent
(2.1)$\Rightarrow$(2.2). Let $\alpha\colon a\to b$, with $a\in\A$ and $b\in\B$, suppose that $Y:=\Im (\A(-,\alpha))\in\T_\B$ and fix an epimorphism $\coprod_{i=1}^n\A(-,b_i)\cong\A(-,\coprod_{i=1}^nb_i)\twoheadrightarrow Y$, with $b_1,\dots,b_n\in \mathcal{B}$. As $\A(-,\coprod_{i=1}^nb_i)$ is projective in $\mod\A$, by the Yoneda Lemma we  get the following commutative diagram, for appropriate morphisms $\beta_i\colon b_i\to a$ and $\beta'_i\colon b_i\to b$ ($i=1,\dots,n$):
\[
\xymatrix{
&& \coprod_{i=1}^n \A(-,b_i) \cong\A(-,\coprod_{i=1}^n b_i) \ar@{>>}[d]_-\pi \ar@{.>}[dll]|-{\coprod_{i=1}^n\A(-,\beta_i)} \ar@{.>}[drr]|-{\A(-,(\beta_1',\dots,\beta_n'))} \\ 
\A(-,a) \ar@{>>}[rr]|-{\ p\ } \ar@/_14pt/[rrrr]|{\ \A(-,\alpha)\ } & &Y \ar@{^(->}[rr]|-{\ \text{inc.}\ } && \A(-,b)
}
\]
Then,  $\alpha\circ\beta_i=\beta'_i$, for all $i=1,\dots,n$. Note that $\alpha\in Y(a)$ since $\A(-,\alpha )(\id_a)=\alpha$. Therefore, $\alpha\in\Im(\pi_a)$, so that we have a morphism $\gamma =(\gamma_1,\cdots,\gamma_n)^t\colon a\to\coprod_{i=1}^nb_i$ in $\A$ such that 
\[
\xymatrix{
\alpha =\pi(\gamma )=\sum_{i=1}^n\beta'_i\circ\gamma_i=\alpha\circ (\sum_{i=1}^n\beta_i\circ\gamma_i).
}\]

\noindent
(2.2)$\Rightarrow$(2.3), {\em assuming that $\A$ is Krull-Schmidt}. Let $\alpha\colon a\to b$ be a non-zero morphism with $b\in\B$ and $a\in \A$ indecomposable. By (2.2), there are morphisms $\beta_i\colon b_i\to a$ and $\gamma_i\colon a\to b_i$, for some family $\{b_1,\dots,b_n\}$ of objects of $\mathcal{B}$, such that $\alpha =\alpha\circ (\sum_{i=1}^n\beta_i\circ\gamma_i)$.
The equality $\alpha\circ (\id_a-\sum_{i=1}^n\beta_i\circ\gamma_i)=0$ and $\alpha\neq 0$ imply that $\id_a-\sum_{i=1}^n\beta_i\circ\gamma_i$ is not invertible in the local (as $a$ is indecomposable) ring $\A(a,a)$. Then, $\sum_{i=1}^n\beta_i\circ\gamma_i$ is invertible, which implies that $(\beta_1,\dots,\beta_n)\colon \coprod_{i=1}^nb_i\to a$ is a retraction in $\A$.  By the uniqueness of decompositions in the Krull-Schmidt category $\A$,  $a$ is isomorphic to a direct summand of some $b_i$. 
 
  \smallskip\noindent
(2.3)$\Rightarrow$(2.1), {\em assuming that $\A$ is Krull-Schmidt}. Let $b\in\B$, then the cyclic submodules of $H_b$ are the {finite sums of} images of  non-zero morphisms of the form  $\A(-,\alpha )\colon H_a\to H_b$, with $b\in\mathcal{B}$ and $a$ indecomposable. By hypothesis, we have a retraction  $\beta\colon b'\to a$, with  $b'\in\mathcal{B}$. We then have that $\Im(\A(-,\alpha))=\Im(\A(-,\alpha\circ\beta))\in\T_P$. 
\end{proof}

\subsection{Examples} 
In this subsection we collect a series of examples to illustrate the main results of the paper.

\begin{ejem}\label{Ex_R_coh_P_proj}
Let $R$ be a right coherent ring,  $P$  a finitely generated projective right $R$-module and $\mathbf{t}=(\T_P,\F_P):=(\Gen(P),P^\perp)$  the associated torsion pair in $\mod R$. The following assertions hold:
\begin{enumerate}[\rm (1)]
\item $\mathbf{t}$ restricts to $\fpmod R$ if, and only if, $P^\star:=\hom_R(P,R)$ is finitely generated as a right $\End_R(P)$-module {if, and only if, the trace $tr_P(R)$ of $P$ in $R$ is finitely generated as a right ideal};
\item $\mathbf{t}$ is a hereditary torsion pair if, and only if, each cyclic submodule of $P$ is in $\Gen(P)$. If  $R$ is also semiperfect, this is equivalent to say that, if there is a non-zero morphism $Q\to P$, with $Q$ indecomposable in $\proj R$, then $Q$ is a direct summand of $P$. 
\end{enumerate}
\end{ejem}
\begin{proof}
Both assertions easily follow by Lemma~\ref{ex_2_modules_idempotents}, considering the additive category $\A:=\proj R$ and $\mathcal{B}:=\{P\}$, since  $\mod R\cong\mod (\proj R)$ (see \cite[Corollary~1.6]{PSV2}). In fact, if $Q\in\proj R$, then  $(\proj R)(P,Q)=\hom_R(P,Q)$ is a direct summand of $(P^\star)^n$ as a right $\End_{R}(P)$-module. Furthermore,  $\proj R$ is a Krull-Schmidt category precisely when $R$ is semiperfect.
\end{proof}

\begin{ejem}
The following are choices for $R$ and $P$ that satisfy the assumptions of Example~\ref{Ex_R_coh_P_proj}, whence with $\mod R$ locally coherent, for which $\mathbf{t}=(\T_P,\F_P)$ is hereditary, does not restrict to $\fpmod R$, but $\Ht$ is locally coherent (see Proposition~\ref{ejems.left-constituent-TTFtriple}): 
\begin{enumerate}[\rm (1)]
\item (With non-semiperfect $R$) Let $K$ be a field, $V$ an infinite dimensional $K$-vector space. Then take $R=\End_K(V)^{\op}$ and $P:=V$, viewed as a right $R$-module;
\item (With semiperfect $R$) Let $L/K$ be an infinite field extension and consider:
\[
R=\left[\begin{smallarray}{cc}K & 0\\ L & L \end{smallarray}\right]\quad\text{and}\quad e_1=\left[\begin{smallarray}{cc} 1 & 0\\ 0 & 0 \end{smallarray}\right].
\]
Then take $R$ and $P:=e_1R$.
\end{enumerate}
\end{ejem}
\begin{proof}
(1). It is well-known that $R$ is a Von Neumann regular ring which is not semisimple (whence it is two-sided coherent and not semiperfect). Furthermore, $P$ is simple projective as a right $R$-module and $\Soc (R_R)$ contains infinitely many copies of $P$, so  $\tr_P(R)$ is not finitely generated as a right ideal. 

\smallskip\noindent
(2). By \cite[Proposition~3.1.3]{HRS}, we know that $R$ is two-sided hereditary (whence two-sided coherent). Moreover, $P=e_1R$ is simple projective and 
\[
\Hom_R(e_1R,R)\cong Re_1=\left[\begin{smallarray}{cc}K & 0\\ L & 0 \end{smallarray}\right]
\] 
is infinitely generated as a right module over $\End(P_R)\cong e_1Re_1\cong K$.
\end{proof}

Lemma~\ref{ex_2_modules_idempotents} and Proposition~\ref{ejems.left-constituent-TTFtriple} can be applied to interesting situations where $\G=\mod\A$, but $\A$ has infinitely many objects. Let us start recalling the following definition:
\begin{defi}
An additive category with split idempotents $\A$ is said to {\bf have left almost split morphisms} when it is Krull-Schmidt and, for each indecomposable object $a\in\mathrm{Ind}\text{-}\A$, there is a {\bf left almost split} morphism $u:a\to b$, i.e., $u$ is not a section and any morphism $v\colon a\to c$ in $\A$ that is not a section factors through $u$. The concepts of {\bf right almost split morphism} and that of {\bf category with right almost split morphisms} are defined dually. We say that $\A$ {\bf has almost split morphisms} when it has both left and right almost split morphisms.
\end{defi}
 Note that, by the properties of Krull-Schmidt categories, if $u\colon a\to b$ is a left almost split morphism, then there is a decomposition 
 \[
 u=\left[\begin{smallarray}{c}{u'}\\ {0} \end{smallarray}\right]\colon a\longrightarrow b=b_1\sqcup b_2,
 \] 
 where $u'\colon a\to b_1$ is left minimal (and also left almost split). Then $u'$ is uniquely determined by $a$, up to isomorphism. The dual phenomenon is true for minimal right almost split morphisms. 
 
 \begin{defi}
 {Let $\A$ be an additive category that has almost split morphisms.} We will say that an additive subcategory $\X=\mathrm{add}(\X)$ of $\A$ is {\bf closed under immediate successors} (resp., {\bf immediate predecessors}) when, if $x\in\X$ is indecomposable and $x\to y$ (resp., $y\to x$) is its minimal left (resp., right) almost split morphism, then $y\in\X$. 
 \end{defi}
 The last property has an interpretation in terms of the Auslander-Reiten (AR)  quiver $\Gamma (\A)$ of $\A$ (we refer to \cite[Chapter~IV]{ASS} for the definition in case $\A=\mod\Lambda$, for $\Lambda$ an Artin algebra. That same definition applies to our more general $\A$). Now, the fact of $\X$ being closed under immediate successors (resp., immediate predecessors)  is equivalent to saying that if $x$ and $y$ are indecomposable objects, with $x\in\X$,  and there is an arrow $x\to y$ (resp., $y\to x$) in $\Gamma (\A)$, then  $x\in\X$ implies  $y\in\X$.

\begin{prop} \label{ejem.trisection}
Let $\A$ be an  skeletally small Abelian category that has almost split morphisms and  $(\C,\B)$  a split torsion pair in $\A$ such that $\B$ is closed under immediate successors and it is generating in $\A$. With the notation of Lemma~\ref{ex_2_modules_idempotents}, $\mod\A$ is a locally coherent category with a hereditary torsion pair $\mathbf{t}:=(\Gen(\P_\B),\P_B^\perp )$ that does not restrict to $\fpmod\A$, but whose associated heart $\Ht$ is locally coherent.
\end{prop}
\begin{proof}
Let us start noticing that $\mod \A$ is locally coherent, since $\A$ has kernels. Moreover, if $\alpha \colon a\rightarrow b$ is a non-zero morphism in $\A$, with $b\in\B$ and $a$ indecomposable, then $a\in\C\cup\B$, as $(\C,\B)$ is split. Thus, $a\in\B$ since $\alpha\neq 0$. By Lemma~\ref{ex_2_modules_idempotents},  $\mathbf{t}$ is hereditary, and so $\Ht$ is locally coherent by Proposition~\ref{ejems.left-constituent-TTFtriple}. 

We are now going to prove that $\B=\mathrm{sum}(\B)$  is not precovering in $\A$, something that, according to Lemma~\ref{ex_2_modules_idempotents},  implies that $\mathbf{t}$ does not restrict  to $\fpmod\A$, thus ending the proof. {Note that if $\B$ is precovering, then it is also covering, due to the Krull-Schmidt condition of $\A$. Let } $a\in\mathrm{Ind}\text{-}\A \setminus\B$ and  $\alpha\colon b\to a$ be a $\B$-cover. Then $\alpha\neq 0$, actually $\alpha$ is an epimorphism  since $\B$ is generating in $\A$. We now consider the indecomposable decomposition $b=\coprod_{j=1}^m b_j$ and, for each $j=1,\dots,m$, we consider  the minimal left almost split morphism $\beta_j \colon b_j\to b_j'$. By our hypotheses on $\B$, we know that $b_j'\in\B$. We put $ \beta=\coprod_{j=1}^m\beta_j\colon b=\coprod_{j=1}^mb_j\to\coprod_{j=1}^mb'_j=:b'$. By definition of left almost split morphism, the morphism $\alpha$ factors through $\beta$, i.e., we have a morphism $\gamma \colon b'\to a$ such that $\alpha =\gamma\circ\beta$. But since $\alpha$ is a $\B$-precover, we also have a morphism $\delta \colon b'\to b$ such that $\alpha\circ\delta =\gamma$. Therefore $\alpha =\alpha\circ \delta\circ\beta$, which implies that $\delta\circ\beta$ is an isomorphism by the right minimality of $\alpha$. But this implies that $\beta$, and hence all the $\beta_j$, are sections, which contradicts the fact that they are  left almost split morphisms.
\end{proof}

In the following two examples we write down a list of concrete examples of the setting described in Proposition~\ref{ejem.trisection}. In particular, Example~\ref{alegbraic_example} collects all the examples of algebraic origin, while examples of geometric origin are given in Example~\ref{geometric_example}.

\begin{ejem}\label{alegbraic_example}
Let {$\Lambda$ be a finite dimensional algebra over an algebraically closed field and} $\A:=\fpmod\Lambda$. {Suppose that the} subcategory of  indecomposables $\mathrm{Ind}\text{-}\Lambda\subseteq\A$ can be written as $\mathrm{Ind}\text{-}\Lambda =\B_0\sqcup\C_0$, with 
\begin{itemize}
\item $\B_0$ and $\mathcal{C}_0$ are unions of connected components of the AR quiver $\Gamma (\fpmod\Lambda)$;
\item $\Hom_\Lambda (C,B)=0$, for all $B\in\B_0$ and $C\in\C_0$;
\item $\B_0$ contains the indecomposable projective $\Lambda$-modules. 
\end{itemize}
Then, we can apply Proposition~\ref{ejem.trisection} with $\C:=\mathrm{sum}(\C_0)$ and $\B:=\mathrm{sum} (\B_0)$. The following is a list of particular examples where this situation is verified: 
\begin{enumerate}[\rm (1)]
\item when $\Lambda$ is a tame hereditary algebra of infinite representation type, a tame concealed algebra of tubular type  or any of Ringel's canonical algebras, and consider the trisection $\mathrm{Ind}\text{-}\Lambda =\P\sqcup\mathcal{T}\sqcup\mathcal{I}$ of \cite[Theorem~3.3.4]{Ringel} (see also Theorems~3.6(5), 3.7 and 4.3 in {\rm [op.\ cit.]}, to which we refer the reader for the undefined terminology). One can then choose $\B_0:=\P$ or $\B_0:=\P\sqcup\T$;
\item  when $\Lambda$ is a wild hereditary algebra and one takes $\mathcal{B}_0:=\P$ or $\B_0:=\P\sqcup\mathcal{R}$, where $\P$ consists of the preprojective  modules (also called postprojective modules in some texts)  and $\mathcal{R}$ consists of the regular  modules (see \cite[Section~VIII.2]{ASS} for the undefined terminology);
\item  when $\Lambda$ is a tilted algebra of infinite representation type, i.e. $\Lambda\cong\mathrm{End}(_AT)^{\op}$ for a hereditary finite dimensional algebra of infinite representation type $A$ and $_AT$ is a classical (=finitely presented) $1$-tilting module.  In this case one can take $\B_0:=\P(A)$ or $\B_0:=\P(A)\cup\mathcal{R}(A)$, with the terminology of \cite[Theorem~VIII.4.5]{ASS}.
\end{enumerate}
\end{ejem}
\begin{proof}
By \cite{Aus}, we know that $\fpmod \Lambda$ has almost split morphisms. It is then easy to conclude by using that, in any skeletally small Abelian Krull-Schmidt category,  the assignment 
\[
\xymatrix{
(\C,\B)\ar@{|->}[r]& (\C\cap\mathrm{Ind}\text{-}\A ,\B\cap\mathrm{Ind}\text{-}\A )
}
\] 
gives a bijection between the split torsion pairs in $\A$ and the pairs $(\C_0,\B_0)$ of subcategories of $\mathrm{Ind}\text{-}\A $ such that $\A(c,b)=0$, for all $c\in\C_0$ and $b\in\B_0$, and $\mathrm{Ind}\text{-}\A =\C_0\sqcup\B_0$. 
 \end{proof}

\begin{ejem}\label{geometric_example}
Let $\A=\mathrm{coh}(\mathbb{X})$ be the category of coherent sheaves over a weighted projective line, in the sense of \cite{Len}, or  $\A=\mathrm{coh}(\mathbb{X})$ for $\mathbb{X}$  a smooth projective curve, all over an algebraically closed field. Take then $\B_0$ to be the subcategory of indecomposable bundles and $\C_0$ that of indecomposable torsion (=finite length) sheaves. Note that $\B_0$ and $\C_0$ are denoted by $\mathcal{H}_+$ and $\mathcal{H}_0$, respectively, in {\rm [op.\ cit.]}. Then, as in Example~\ref{alegbraic_example}, we can apply Proposition~\ref{ejem.trisection} with $\C:=\mathrm{sum}(\C_0)$ and $\B:=\mathrm{sum} (\B_0)$.
\end{ejem}
\begin{proof}
It is enough to apply \cite[Proposition~10.1]{Len}, taking into account that $\mathrm{coh}(\mathbb{X})$, with $\mathbb{X}$ a smooth projective curve, is an Abelian category that satisfy properties H1--H6 in \cite[Subsection\,10.2]{Len}, and such that each tube $\mathcal{U}_x$ is homogeneous (see the comment introducing \cite[Lemma~10.3]{Len}). 
\end{proof}

The following is a very explicit example of a general phenomenon we discuss in Proposition~\ref{ejem.loc.coher.heart in non-loc-coher Groth.cat}:

\begin{ejem}\label{ej_triangular_VNr}
Let $A$ be a commutative Von Neumann regular ring that is not semisimple (e.g., an infinite product of fields) and $I$ an ideal which is not a  summand of $A$, then the triangular matrix ring 
\[
R:=\left[\begin{smallarray}{c c}A/I & A/I\\ 0 & A \end{smallarray}\right]
\] 
is  left semihereditary  but not right semihereditary (see \cite[Proposition~3.1]{Chase2}). Furthermore, such a ring $R$, cannot be right coherent, i.e., $\mod R$ cannot be locally coherent (see \cite[Theorem~4.1]{Chase}).  
\end{ejem}

The following proposition shows that Theorem~\ref{thm.main_thm_6c}  applies to situations where the ambient Grothendieck category is not locally coherent. 

\begin{prop} \label{ejem.loc.coher.heart in non-loc-coher Groth.cat}
Let $R$ be a left semihereditary ring that is not right semihereditary (such rings do exist by Example~\ref{ej_triangular_VNr}), and $\F:=\Flats R$.  Then,  $\mathbf{t}:=(_{}^{\perp}\F,\F)$ is a torsion pair of finite type in $\mod R$, with $\F$ generating. Furthermore, it restricts to $\fpmod R$ and the associated heart $\H_\mathbf{t}$ is locally coherent. However, $R$ is not right coherent, i.e., $\mod R$ is not locally coherent. 
\end{prop}
\begin{proof}
By a famous result of Chase (see \cite[Theorem~4.32]{Ro}), $R$ is left semihereditary if and only if submodules of flat modules are flat, and $R$ is left coherent. In particular, $R$ has weak dimension $\leq 1$  in the terminology of \cite{Ro}. Furthermore, since weak dimension is left-right symmetric (see \cite[Theorem~8.19]{Ro}), $\F$ is closed under taking subobjects. Moreover, $R$ is left coherent and then, by another result of Chase (see \cite[Theorem~2.1]{Chase}), $\F$ is closed under taking products. It then follows that $\F$ is a generating torsionfree class in $\mod R$, since it is clearly closed under extensions. It is also closed under direct limits so that  $\mathbf{t}=(_{}^{\perp}\F,\F)$ is a torsion pair of finite type in $\mod R$, so $\Ht$ is a Grothendieck category.
We claim that assertion (4) of Theorem~\ref{thm.main_thm_6c}   holds for $\mathbf{t}$, and hence $\H_\mathbf{t}$ is also locally coherent. Indeed, by \cite[Corollary~1.4]{L}, we have that $\F\cap\fpmod R=\proj R$ and this class is clearly contained in $\fp_\infty(R)$. On the other hand, if $M\in\fpmod R$ then the projection $p\colon M\twoheadrightarrow (1:t)(M)$ factors in the form 
\[
\xymatrix{p\colon M\ar[r]^-{u}&P\ar[r]^-{v}&(1:t)(M)},\quad\text{where $P\in\proj R$ (see \cite[Theorem~1.2]{L}).}
\] 
Since $P\in\F$, there is a map $q\colon (1:t)(M)\to P$ such that $q\circ p=u$. We then get $v\circ q\circ p=v\circ u=p$, and so $v\circ q=\id_{(1:t)(M)}$. Therefore $v$ is a retraction and $(1:t)(M)\in\proj R\subseteq\fpmod R$. Moreover, the sequence $0\rightarrow t(M)\to M\to (1:t)(M)\rightarrow 0$ splits, showing that $\mathbf{t}$ restricts to $\fpmod R$. 
\end{proof}

The above example and proposition give a negative answer to an extension of \cite[Question~{7.9}]{PSV}. Indeed, the original question asks whether the property of being locally coherent is invariant under classical \mbox{$1$-tilting} equivalences (between categories of modules). When this question is extended to locally finitely presented Grothendieck categories, Proposition~\ref{ejem.loc.coher.heart in non-loc-coher Groth.cat} can be used to give a negative answer. Indeed, in the notation of the example, $\Ht$ is a locally coherent Grothendieck category with a tilting torsion pair $\bar{\mathbf{t}}=(\F[1],\T[0])$, associated with the tilting object $R[1]$. The heart of the $t$-structure associated with $\bar{\mathbf{t}}$ is equivalent to $\mod R$. {We then have a triangulated equivalence $\Der(\Ht)\cong\Der(\mod R)$, where $\Ht$ is locally coherent but $\mod R$ is not.}

\subsection{Locally coherent hearts and elementary cogenerators}

%\textcolor{red}{\bf se debe cambiar}

In this final subsection, we deal with the notions of pure exact sequences, pure monomorphisms and pure epimorphisms in a category of modules $\mod \A$. For these notions, we refer to \cite{CB} and \cite{Prest}. An analogous theory of purity exists in compactly generated triangulated categories, for the notions of pure exact triangle, pure monomorphisms and  epimorphisms in the derived category $\Der(\mod \A)$ of $\mod \A$, we refer to \cite{Prest}.

\medskip
Let us also recall that, given a locally finitely presented Grothendieck category $\G$, a subcategory $\mathcal{Y}\subseteq\mathcal{G}$ is said to be {\bf definable} when it is closed under taking pure subobjects, products and direct limits. When $Y$ is a pure-injective object of $\mathcal{G}$, we shall denote by 
\[
\Cogen_*(Y):=\{\text{pure subobjects of objects in $\Prod(Y)$}\}
\] 
the subcategory of $\mathcal{G}$ of the objects isomorphic to pure subobjects of products of copies of $Y$. The crucial concept for us in this subsection is the following:
 
\begin{defi}
A pure-injective object $Y$ of $\mathcal{G}$ is called an {\bf elementary cogenerator} when the subcategory $\Cogen_*(Y)$ is definable.
\end{defi}

Note that a pure-injective object $Y$ of $\mathcal{G}$ is an elementary cogenerator if, and only if, $\Cogen_*(Y)$ is closed under direct limits. In fact, $\Cogen_*(Y)$ is always closed under pure subobjects, while one can show that it is closed under products using that, even if products may fail to be exact in $\G$, a product of pure monomorphisms is a pure monomorphism (for this use that products are exact in $\mod{(\fp(\G))}$). 

\begin{lema}\label{elementary_in_the_quotient_lem_1}
Let $\mathcal{A}$ be a preadditive category,  $\mathcal{I}$  a (two-sided) ideal of $\mathcal{A}$, and view $\mod{\bar{\mathcal{A}}}$ as a  subcategory of $\mod \mathcal{A}$ in the obvious way, where $\bar{\mathcal{A}}:=\mathcal{A}/\mathcal{I}$. Let $\mathcal{Y}$ and $Y$ be, respectively,  a  subcategory and an object of $\mod{\bar{\mathcal{A}}}$. The following assertions hold:
 \begin{enumerate}[\rm (1)]
 \item $\mathcal{Y}$ is definable in $\mod{\bar{\mathcal{A}}}$ if and only if it is definable in $\mod \mathcal{A}$; 
 \item $Y$ is pure-injective (resp., an elementary cogenerator) in $\mod{\bar{\mathcal{A}}}$ if and only if it is so in $\mod \mathcal{A}$.
 \end{enumerate}
 \end{lema}
\begin{proof}
Since $\mod{\bar{\mathcal{A}}}$ is closed under taking subobjects, quotients, coproducts and products in $\mod \mathcal{A}$, it follows that it is also closed under taking direct limits and so the pure submodules of products of copies of a given $\bar{\mathcal{A}}$-module are the same in $\mod{\bar{\mathcal{A}}}$ and $\mod \mathcal{A}$. Assertion (1) is then clear. Furthermore, the part of assertion (2) regarding pure-injectivity follows similarly (see also Remark~\ref{rem_purity_subcat}), while for the part regarding elementary cogenerators, it is enough to use the above observation that $\Cogen_*(Y)$ is definable if and only if it is closed under direct limits. 
\end{proof}

A  subclass $\X\subseteq \Der(\A)$ is said to be a {\bf definable class} in $\Der(\A)$ if it is closed under taking products, directed homotopy colimits and pure subobjects. Definable classes in $\Der(\A)$ have been recently characterized in \cite{LV} as the classes closed under  direct products, pure subobjects and pure quotients. 

\begin{defi}
A pure-injective object $Y$ in $\Der(\A)$ is said to be an {\bf elementary cogenerator} when the smallest definable subcategory of $\Der(\A)$ which contains $Y$ coincides with the following class
\[
\Cogen_{\Der(\A)}^*(Y):=\{X\in \Der(\A):X\text{ admits a pure monomorphism into }X',\text{ for some $X'\in\Prod (Y)$}\}.
\] 
\end{defi}
In fact, this is equivalent to say that $\Cogen_{\Der(\A)}^*(Y)$  is closed under taking directed homotopy colimits. For more details on these notions we refer to \cite{Laking}.

\begin{ejem}\label{gen_laking_ex}
Let $\A$ be a small preadditive category. Then, the subclass $(\mod \A)[0]\subseteq \Der(\A)$ is  definable in $\Der(\A)$ since
$
\mod \A=\bigcap\{\Ker(\Der(\A)(H_a[n],-)): {a\in \A,\, n\neq0}\}
$
so \cite[Theorem~3.11]{Laking} applies. By  \cite[Section~5]{Laking},  a cotilting module $X$ in $\mod \A$ is an elementary cogenerator in $\mod \A$ if and only if $X[0]$ is an elementary cogenerator in $\Der(\A)$ (for $\A=R$ a ring, this is \cite[Example~5.13]{Laking}).
\end{ejem}

\begin{lema}\label{lemma_38i}
Let $\A$ be a small preadditive category, $\t=(\T,\F)$ a torsion pair in $\mod \A$, and consider the torsion pair $t' := (\T\cap\mod{\bar \A},\F)$ induced in $\mod{\bar \A}$. The following assertions are equivalent:
\begin{enumerate}[\rm (1)]
\item $\mathcal{H}_{\mathbf{t}'}$ is a locally coherent Grothendieck category;
\item $\mathbf{t}$ is associated with a cosilting $\mathcal{A}$-module $Q$ which is an elementary cogenerator in $\mod \mathcal{A}$.
\end{enumerate}
When $\mathcal{H}_\mathbf{t}$ is a  locally coherent Grothendieck category, both conditions hold.
\end{lema}
\begin{proof}
Under either of the conditions (1) or (2),  $\mathbf{t}'$ is of finite type in $\mod{\bar{\mathcal{A}}}$ and $\mathcal{F}$ is a generating class in this category. Then, by \cite[Proposition~5.7]{PS1}, we know that there is a cotilting, whence pure-injective, $\bar{\mathcal{A}}$-module $Q$ such that $\mathcal{F}=\Ker(\Ext_{\mod{\bar{\mathcal{A}}}}^1(-,Q))$.  
The argument of \cite[Example~5.13]{Laking} is easily adapted to our situation, giving that $Q$ is an elementary cogenerator of $\mod{\bar{\mathcal{A}}}$ if, and only if, it is an elementary cogenerator of $\Der(\mod{\bar{\mathcal{A}}})$. Moreover, the final statement follows by Corollary~\ref{cor.loc-coherence-implies-restriction}.

\noindent
(1)$\Leftrightarrow$(2). Note that $Q$ is a quasi-cotilting $\mathcal{A}$-module which defines the torsion pair $\mathbf{t}$. Then, up to $\Prod$-equivalence, we can assume that $Q$ is a cosilting $\mathcal{A}$-module (see Corollary~\ref{main_cor_fintype_qcotilt_cosilt}). Moreover, by the previous paragraph and \cite[Theorem~5.12]{Laking}, we know that assertion (1) holds if, and only if, $Q$ is an elementary cogenerator in $\mod{\bar{\mathcal{A}}}$. But, by Lemma~\ref{elementary_in_the_quotient_lem_1}, this happens if and only if $Q$ is an elementary cogenerator of $\mod \mathcal{A}$. 
\end{proof}

The following question was communicated to us by Rosanna Laking:

\begin{ques}[R. Laking] 
Let $\A$ be a small preadditive category, $\mathbf{t}$ a torsion pair in $\mod \mathcal{A}$ such that $\mathcal{H}_\mathbf{t}$ is a locally coherent Grothendieck category. Is  $\t$ the torsion pair associated with a $2$-term cosilting complex $E$ (see Theorem~\ref{main_thm_fintype_qcotilt_cosilt}) which is an elementary cogenerator of $\Der(\mod \mathcal{A})$?
\end{ques}

In the rest of this subsection we give some results inspired by the above question, as we study cosilting objects that are elementary cogenerators in $\mod \A$, while leaving the more general case of cosilting complexes which are elementary cogenerators in $\Der(\mod \mathcal{A})$ for future investigation.

\begin{prop}\label{last_prop_laking}
Let $\mathcal{A}$ be a small preadditive category and  $\mathbf{t}=(\mathcal{T},\mathcal{F})$  a torsion pair in $\mod \A$.  Consider the following assertions:
\begin{enumerate}[\rm (1)]
\item the heart $\mathcal{H}_\mathbf{t}$  is a locally coherent Grothendieck category;
\item the following conditions hold:
\begin{enumerate}[\rm {(2.}1)]
\item $\mathbf{t}$ is generated by finitely presented modules;
\item $\mathbf{t}$ is associated with a cosilting $\mathcal{A}$-module $Q$ which is an elementary cogenerator in $\mod \mathcal{A}$;
\item the functor $\ann_{(-)}(t(\mathcal{A}))\colon\mod \mathcal{A}\to\mod \mathcal{A}$ preserves direct limits of object in $\mathcal{T}$;
\end{enumerate}
\item  conditions {\rm (2.1)} and {\rm(2.3)} above hold, as so does the following condition:
\begin{enumerate}[\rm {(3.}2)]
\item  $\mathbf{t'}=(\T\cap\mod\bar{\A},\F)$ restricts to $\fpmod\bar{\A}$ and $\F\cap\fpmod\bar{\A}\subset\fp_\infty (\mod\bar{\A})$.
\end{enumerate}
\end{enumerate}
The implications {\em``(1)$\Rightarrow$(2)$\Leftrightarrow$(3)''} hold true. When the additive category $\proj \A$ (or $\mathrm{sum}(\P)$) has pseudo-kernels (i.e., when $\mod \mathcal{A}$ is locally coherent),  all assertions are equivalent.
%The implication (1)$\Rightarrow$(2) holds true. When the additive closure $\hat{\mathcal{A}}$ has pseudo-kernels (i.e., $\mod \mathcal{A}$ is locally coherent) and the canonical morphism 

%\[
%{\varinjlim}\Ext_{\mod \mathcal{A}}^2((1:t)(\mathcal{A}(-,a)),Q_\lambda)\longrightarrow\Ext_{\mod \mathcal{A}}^2((1:t)(\mathcal{A}(-,a)),{\varinjlim} Q_\lambda)
%\] 
%is a monomorphism, for each $a\in\Ob(\mathcal{A})$ and all direct systems in $\Prod(Q)$, then the two assertions are equivalent. 
\end{prop}
\begin{proof} 
By Lemma \ref{lemma_38i}, (2.2) is equivalent to the local coherence of $\mathcal{H}_{\mathbf{t}'}$. {The implications} ``(1)$\Rightarrow$(2)'' {and ``(2)$\Rightarrow$(1)'', the latter when $\mod\A$ is locally coherent, translate to this situation  the implications ``(1)$\Rightarrow$(4)''  and ``(4)$\Rightarrow$(1)''  of  Corollary~\ref{cor.after-referee-report}.} 

\smallskip\noindent 
(2)$\Leftrightarrow$(3).  By Theorem \ref{thm.main_thm_6c}{, with $(\underline{\F}=\mod\bar{\A},\mathbf{t}')$ instead of $(\G,\mathbf{t})$,} we know that $\mathcal{H}_{\mathbf{t}'}$ is locally coherent if, and only if, condition (3.2) holds.
%From Theorem~\ref{thm.locally-fp-hearts} we get condition (2.1). Furthermore, for every objet $a\in\A$, we know that $(1:t)H_a\in \F_0$, so the functor $\Ext^{1}_{\Ht}(((1:t)H_a)[1],-)\colon \T[0] \to \Ab$ preserves direct limits (see Corollary~\ref{cor.F0}); let us show that this implies condition (2.3). Indeed, for each $a\in\mathcal{A}$ and each $\mathcal{A}$-module $M$, the canonical projection 
%\[
%p\colon H_a\twoheadrightarrow (1:t){H_a}\cong\frac{\mathcal{A}(-,a)}{t(\mathcal{A})(-,a)}
%\] 
%induces the following monomorphism in $\Ab$:
%\[
%p_*\colon\hom_\mathcal{A}\left(\frac{\mathcal{A}(-,a)}{t(\mathcal{A})(-,a)},M\right)\longrightarrow \hom_\mathcal{A}\left(H_a,M\right)\cong M(a).
%\] 
% It is easy to see that $\Im(p_*)=(\ann_M(t(\mathcal{A})))(a)$.  Therefore, condition (2.3) above is equivalent to say that, for each $a\in\mathcal{A}$, the functor $\hom_\A((1:t)H_a,-)\colon \mod \mathcal{A}\to\Ab$ preserves direct limits of objects of $\mathcal{T}$.
% Finally, note that  condition (2.2) follows by Lemma~\ref{lemma_38i}. 
\end{proof}

\begin{cor} \label{cor.elementary-cogenerator}
Let $\mathcal{A}$ be a small preadditive category such that $\proj \A${ (or $\mathrm{sum}(\P)$)} has pseudo-kernels, and let $\mathbf{t}=(\mathcal{T},\mathcal{F})$ be a torsion pair in $\mod \mathcal{A}$ such that $t(H_a)$ is finitely generated, for all $a\in\mathcal{A}$. Then, the following assertions are equivalent:
  \begin{enumerate}[\rm (1)]
  \item the heart $\mathcal{H}_\mathbf{t}$ is locally coherent;
  \item $\mathbf{t}$ is generated by finitely presented modules and $\mathbf{t}=(_{}^\perp Q,\Cogen (Q))$, for some  cosilting $\mathcal{A}$-module $Q$ which is an elementary cogenerator in $\mod \mathcal{A}$.  
 \end{enumerate}
\end{cor} 
\begin{proof}
Note that $t(H_a)\in\fpmod \A$, for every $a \in \A$, so that condition (2.3) in Proposition~\ref{last_prop_laking} is clear (see the first paragraph of the proof of {Corollary~\ref{cor.after-referee-report}}). Hence, the result follows by Proposition~\ref{last_prop_laking}.
% Moreover, dimension shifting gives that $\Ext_{\mod \mathcal{A}}^2(\bar{\mathcal{A}}(-,a),-)\cong\Ext_{\mod \mathcal{A}}^1(t(\mathcal{A}(-,a)),-)$, and this last functor preserves direct limits due to the local coherence of $\mod \mathcal{A}$. Note that $t(\mathcal{A}(-,a))$ is actually finitely presented. Then the implication ``(2)$\Rightarrow$(1)'' of Proposition~\ref{last_prop_laking} holds. 
%Condition (2.3) of last Proposition~\ref{last_prop_laking} is clear (see first paragraph of the proof of Proposition~\ref{last_prop_laking}). Moreover, dimension shifting gives that $\Ext_{\mod \mathcal{A}}^2(\bar{\mathcal{A}}(-,a),-)\cong\Ext_{\mod \mathcal{A}}^1(t(\mathcal{A}(-,a)),-)$, and this last functor preserves direct limits due to the local coherence of $\mod \mathcal{A}$. Note that $t(\mathcal{A}(-,a))$ is actually finitely presented. Then the implication ``(2)$\Rightarrow$(1)'' of Proposition~\ref{last_prop_laking} holds. 
\end{proof}


\begin{thebibliography}{9999}
{\footnotesize
%\bibitem{AR} {\sc ADAMEK, J.; ROSICKY, J.}: Locally presentable and accessible categories. London Math. Soc. Lect. Notes \textbf{189}. Cambridge Univ. Press (1994).

%\bibitem{AJS} {\sc ALONSO, L.; JEREM\'IAS, A.; SAOR\'IN, M.}: Compactly generated $t$-structure on the derived category of a Noetherian ring. J. Algebra \textbf{324}(4) (2010), 313-346.

\bibitem{AJSo} {\sc ALONSO, L.; JEREM\'IAS, A.; SOUTO, M.J.}: Localizations in categories complexes and unbounded resolutions. Canad. J. Math. \textbf{52}(2) (2000), 225--247.

\bibitem{A} {\sc  ANGELERI H\"UGEL, L.}: On the abundance of silting modules. Surveys in representation theory of algebras, Contemp. Maths \textbf{716} (2018), 1--23.

%\bibitem{AH} {\sc ANGELERI H\"UGEL, L.; HRBECK, M.}: Silting modules over commutative rings. International Mathematics Research Notices, 2017(13), 4131--4151.

%\bibitem{AMV} {\sc ANGELERI H\"UGEL, L.; MARKS, F.; VITORIA, J.}:  Torsion pairs in silting theory. Pacific J. Maths \textbf{291}(2) (2017), 257-278.

%\bibitem{AC} {{\sc ASSEM, I.; COELHO, F.U.}: Two-sided glueings of tilted algebras. J. Alg. \textbf{269}(2) (2003), 456--479.}

\bibitem{ASS} {{\sc ASSEM, I.; SIMSON, D.; SKOWRONSKI, A.}: Elements of the Representation Theory of Associative Algebras, vol. 1. London Math. Soc. Student Texts \textbf{65}. Cambridge Univ. Press (2006).}

\bibitem{Aus} {{\sc AUSLANDER, M.}: A survey of existence theorems for almost split sequences. In 'Representations of Algebras', Proc. Durham Symposium, 1985.  Cambridge Univ. Press (1986), 81--90}.

\bibitem{Bazz} {\sc BAZZONI, S.}: When are definable classes tilting and cotilting classes. J. Algebra \textbf{320}(12)  (2008), 4281--4299. 


\bibitem{BBD} {\sc BEILINSON, A.; BERNSTEIN, J.; DELIGNE, P.}: Faisceaux pervers, Ast\'erisque \textbf{100}, Soc. Math. France, Paris (1982), 5--171.

%\bibitem{Bo} {\sc BONDARKO, M.}: On torsion pairs, (well-generated) weight structures, adjacent $t$-structures, and related (co)homological functors. Preprint available at https://arxiv.org/abs/1611.00754
\bibitem{Bo} {\sc BONDARKO, M.}: On perfectly generated weight structures and adjacent t-structures. Math.Z.\ \textbf{300} (2022) 1421--1454. 

%\bibitem{Bor} {\sc BORCEUX, F.}:  Handbook of Categorical Algebra 1: Basic Category Theory. Encycl. Maths. and Appls. Cambridge Univ. Press (1994)\\

\bibitem{BPa} {\sc BRAVO, D.; PARRA, C.E.}: tCG torsion pairs, J.Alg.Appl.\ \textbf{18}(7), (2019) 15 pages. 


\bibitem{BP2} {\sc BRAVO, D.; PARRA, C.E.}: {Torsion pairs over $n$-hereditary rings}, Comm. Alg. \textbf{47}(5) (2019), 1892--1907. 

\bibitem{BGP} {\sc BRAVO, D.; GILLESPIE, J.; P\'EREZ, M.A.}: Locally type $FP_n$ and $n$-coherent categories, preprint (2019).\\
Preprint available at: \url{https://arxiv.org/pdf/1908.10987}

\bibitem{BP} {\sc BREAZ, S.; POP, F.}: Cosilting modules.  Algebras Repres. Theory \textbf{20}(5)  (2017), 1305--1321.

\bibitem{BZ} {\sc BREAZ, S.; ZEMLICKA, J.}: Torsion classes generated by silting modules. Arkiv Math. \textbf{56} (2018), 15--32.

\bibitem{Chase} {\sc CHASE, S.U.}: Direct products of modules. Trans. Am. Math. Soc. \textbf{97} (1960), 457--473.

\bibitem{Chase2} {\sc CHASE, S.U.}: A generalization of the ring of triangular matrices. Nagoya Math. J. \textbf{18} (1961), 13--25.

\bibitem{C} {\sc COLPI, R.}: Tilting in Grothendieck categories. Forum Mathematicum. \textbf{11}(6). Berlin; New York: De Gruyter (1999).

\bibitem{CG} {\sc COLPI, R.; GREGORIO, E.}: The heart of a cotilting torsion pair is a Grothendieck category. Unpublished preprint (2008).

\bibitem{CGM} {\sc COLPI, R.; GREGORIO, E.; MANTESE, F.}: On the heart of a faithful torsion theory. J. Algebra \textbf{307} (2007), 841--863. 

\bibitem{Coupek-Stovicek} {\sc \v{C}OUPEK, P., \v{S}T'OV\'I\v{C}EK, J.}: Cotilting sheaves on noetherian schemes. Math. Z. (2019), 1--38.

\bibitem{CB} {\sc CRAWLEY-BOEVEY, W.}: Locally finitely presented additive categories. Comm. Alg. \textbf{22}(5) (1994), 1641--1674.

\bibitem{EG} {\sc ESTRADA, S.; GILLESPIE, J.}: {Notes on absolutely clean quasi-coherent sheaves}, private communication (2020).

\bibitem{Ga} {\sc GARKUSHA, G.}: Classifying finite localizations of quasicoherent sheaves. St. Petersburg Math. J. \textbf{21} (2010), 433--458.

\bibitem{Groth} {\sc GROTH, M.}: Derivators, pointed derivators and stable derivators. Algebraic and Geometric Topology \textbf{13}(1) (2013), 313-374.

%\bibitem{Go} {\sc GOLAN, J.\,S.}: {Torsion theories}. Chapman and Hall, Monogr. and Surveys Pure and Appl. Maths. Longman Scientific and Technical (1986).

\bibitem{Gro} {\sc GROTHENDIECK, A.}: Sur quelques points d'Alg\`ebre Homologique. Tohoku Math. J. \textbf{9}(2)  (1957), 119--221.

\bibitem{HRS} {\sc HAPPEL, D.; REITEN, I.; SMALO, S.O.}: Tilting in Abelian categories and quasitilted algebras. Mem. Amer. Math. Soc. \textbf{120} (1996).

%\bibitem{Har} {\sc HARTSHORNE, R.}: Residues and duality. Lect. Notes Maths. \textbf{20}. Springer (1966).

%\bibitem{He} {\sc HERZOG, I.}: The Ziegler spectrum of a locally coherent Grothendieck category. Proc. London Math. Soc. \textbf{74}(3) (1997), 503-558.

\bibitem{Hoshino-Kato-Miyachi} {\sc HOSHINO, M.; KATO, Y.; MIYACHI, J. I.}: On $t$-structures and torsion theories induced by compact objects. Journal of Pure and Applied Algebra, 167(1) (2002), 15--35.

\bibitem{Jensen-Lenzing} {\sc JENSEN, C.U.; LENZING, H.}: Model theoretic algebra with particular emphasis on fields, rings, modules. Vol. 2. CRC Press (1989).

%\bibitem{keller1998introduction} {\sc KELLER, B.}:, {Introduction to Abelian and derived categories}, Cambridge: Cambridge University Press (1998).

%\bibitem{keller-nicolas} {\sc Keller, B.,  NICOL\'AS, P.}: Weight structures and simple dg modules for positive dg algebras. International Mathematics Research Notices, 2013 (5), 1028--1078.

%\bibitem{Kr} {\sc KRAUSE, H.}: The spectrum of a locally coherent category. J. Pure and Appl. Algebra \textbf{114} (1997), 259--271.

\bibitem{Kr_tel_inv} {\sc KRAUSE, H.}: Smashing subcategories and the telescope conjecture: an algebraic approach. Inventiones mathematicae, 139(1),  (2000) 99--133.

\bibitem{Kr_coherent} {\sc KRAUSE, H.}: Coherent functors in stable homotopy theory. Fund. Math. \textbf{173} (2002), 33-56.



\bibitem{krause2015krull}  {\sc KRAUSE, H.}: Krull--Schmidt categories and projective covers. {Expositiones Mathematicae}, \textbf{33} (2015), {535--549}.

%\bibitem{Ku} {\sc KUNZ, E.}: Introduction to Commutative Algebra and Algebraic Geometry. Birkh\"auser, Boston (1985). 

\bibitem{Laking} {\sc LAKING, R.}: Purity in compactly generated derivators and $t$-structures with Grothendieck hearts. Math. Z. \textbf{295}, (2020) 1615--1641.

\bibitem{LV} {\sc LAKING, R.; VITORIA, J.}: Definability and approximations in triangulated categories. Pacific J. Math. \textbf{306}, (2020) 557--586. 

\bibitem{L} {\sc LAZARD, D.}: Autour de la platitude. Bull. Soc. Math. France \textbf{97} (1969), 81--128.

\bibitem{Len}{{\sc LENZING, H.}: Hereditary Abelian categories. In `Handbook of Tilting Theory', by L. Angeleri-Hugel, D. Happel and H. Krause (edts). London Math. Soc. Lect. Not. Ser. \textbf{332}. Cambridge Univ. Press (2007)}


\bibitem{Lu} {\sc LURIE, J.}: {Higher algebra}. (2017) \ Available at: \url{http://www.math.harvard.edu/~lurie/papers/HA.pdf}

\bibitem{MacLane} {\sc MACLAINE, S.}: Homology. Springer Science \& Business Media, 2012.

%\bibitem{M} {\sc MATSUMURA, H.}: Commutative Ring Theory. Cambridg Studies Adv. Maths \textbf{8}. Cambridge Univ. Press (1997).

%\bibitem{mitchell} {\sc MITCHELL, B.}: {Rings with several objects}. Advances in Mathematics {\bf 8}(1) (1972), 1--161.

\bibitem{N} {\sc NEEMAN, A.}: Triangulated categories. Ann. Math. Stud. \textbf{148}. Princeton University Press (2001).

\bibitem{NSZ} {\sc NICOL\'AS, P., SAOR\'IN, M., ZVONAREVA, A.}: Silting theory in triangulated categories with coproducts. Journal of Pure and Applied Algebra 223.6 (2019): 2273--2319.

\bibitem{PS1} {\sc PARRA, C.E.; SAORIN, M.}: Direct limits in the heart of a $t$-structure: the case of a torsion pair. J. Pure and Appl. Algebra \textbf{219} (2015), 4117--4143.

\bibitem{PS2} {\sc PARRA, C.E.; SAORIN, M.}: Addendum to ``Direct limits in the heart of a $t$-structure: the case of a torsion pair''[J. Pure and Appl. Algebra \textbf{219}(9)(2015) 4117-4143]. J. Pure and Appl. Algebra \textbf{220}(6) (2016), 2467--2469.%\\  
%Preprint available at: \url{http://arxiv.org/pdf/1509.00258.pdf}

\bibitem{PS3} {\sc PARRA, C.E.; SAORIN, M.}: Hearts of $t$-structures in the derived category of a commutative Noetherian ring. Trans. Amer. Math. Soc. \textbf{369} (2017), 7789--7827.

\bibitem{PS4}  {\sc PARRA, C.E.; SAORIN, M.}: The HRS tilting process and Grothendieck hearts of t-structures. Contemporary Mathematics, \textbf{769} (2021), 209-241. 

\bibitem{PSV}  {\sc PARRA, C.E.; SAORIN, M.; VIRILI, S.}: Tilting preenvelopes and cotilting precovers in general Abelian categories. Algebras \& Repr. Theory, (2022). \url{https://doi.org/10.1007/s10468-022-10126-5}


\bibitem{PSV2}  {\sc PARRA, C.E.; SAORIN, M.; VIRILI, S.}: Torsion pairs in categories of modules over a preadditive category. Bull. Iran. Math. Soc. \textbf{47} (2021), 1135--1171.

\bibitem{Poli} {\sc POLISHCHUK, A.}: Constant families of $t$-structures on derived categories of coherent sheaves.
Mosc. Math. J., \textbf{7}(1) (2007), 109--134.

\bibitem{Porta} {\sc PORTA, M.}: The Popescu-Gabriel Theorem for triangulated categories. Adv. Maths. \textbf{225} (2010), 1669--1715.

\bibitem{Prest} {\sc PREST, M.}: Purity, spectra and localisation. Encycl. Maths and Appl. \textbf{121}. Cambridge Univ. Press (2009).

\bibitem{PV} {\sc PSAROUDAKIS, C.; V\'ITORIA, J.}: Realisation functors in tilting theory. Math.Z.\ 288.3-4 (2018), 965--1028.

%\bibitem{RSk} {{\sc REITEN, I.; SKOWRONSKI, A.}: Generalized double tilted algebras. J. Math. Soc. Japan \textbf{56}(1) (2004), 269--288. }
 
 \bibitem{Ringel}{{\sc RINGEL, C.M. }: Tame algebras and integral quadratic forms. Springer LNM \textbf{1099} (1984).}

\bibitem{Ro} {\sc ROTMAN, J.J.}: An introduction to Homological Algebra, 1st edition. Academic Press (1979). 

%\bibitem{Ro2} {\sc ROTMAN, J.J.}: An introduction to Homological Algebra, 2nd edition. Springer (2009). 

\bibitem{Sa}{\sc SAOR\'IN, M.}: Locally coherent hearts. Pacific J. Math. \textbf{287}(1) (2017), 199--221.

\bibitem{Saorin-Stovicek} {\sc SAOR\'IN, M.; STOVICEK, J.}: $t$-Structures with Grothendieck hearts via functor categories. (2020). Preprint available at: \url{https://arxiv.org/abs/1708.07540}.

\bibitem{SSV} {\sc SAOR\'IN, M.; STOVICEK, J.; VIRILI, S.}: $t$-structures on stable derivators and Grothendieck hearts. (2018).
Preprint available at: \url{https://arxiv.org/abs/1708.07540} 


%\bibitem{Si} {\sc SITTE, T.}: Local cohomology sheaves on algebraic stacks. Ph.D.\,Thesis. Universit\"at Regensburg\\
%\url{http://epub.uni-regensburg.de/30539/1/thesis.pdf}

\bibitem{St} {\sc STENSTR\"ON, B.}: Rings of quotients. Springer-Verlag (1975).

%\bibitem{Sto} {\sc STOVICEK, J.}: On purity and applications to coderived and singularity categories. \\
%Preprint available at: \url{http://arxiv.org/pdf/1412.1615.pdf}

\bibitem{Swan} {\sc SWAN, R.G.}: K-Theory of coherent rings.\ J.Alg.Appl.\ \textbf{18}(9) (2019) \url{https://doi.org/10.1142/S0219498819501615}

%\bibitem{V} {\sc VERDIER, J.L.}: Des cat\'egories d\'eriv\'ees des cat\'egories ab\'eliennes. Ast\'erisque \text{239} (1996)

%\bibitem{ZW1} {\sc ZHANG, P.; WEI, J.}: Quasi-cotilting modules and torsionfree classes. J. Algebra and Appl. (2016). DOI: 10.1142/S0219498819502141. Also available at https://arxiv.org/abs/1601.01387.

\bibitem{ZW2} {\sc ZHANG, P.; WEI, J.}: Cosilting complexes and AIR-cotilting modules. J. Algebra \textbf{491} (2017), 1--31.
}
\end{thebibliography}
\end{document}